\documentclass[11pt,reqno]{amsart}

\usepackage{amsmath,amssymb,amsthm,mathtools}
\usepackage{amscd}
\usepackage{dsfont}
\usepackage[margin=1.05in]{geometry}
\usepackage[shortlabels]{enumitem}
\usepackage{booktabs}
\usepackage{microtype}
\usepackage[colorlinks=true,linkcolor=blue,citecolor=blue,urlcolor=blue]{hyperref}

\theoremstyle{plain}
\newtheorem{theorem}{Theorem}[section]
\newtheorem{proposition}[theorem]{Proposition}
\newtheorem{lemma}[theorem]{Lemma}
\newtheorem{corollary}[theorem]{Corollary}
\newtheorem{conjecture}[theorem]{Conjecture}
\theoremstyle{definition}
\newtheorem{definition}[theorem]{Definition}

\theoremstyle{remark}
\newtheorem{remark}[theorem]{Remark}

\newcommand{\BQ}{\mathbb{Q}}
\newcommand{\BZ}{\mathbb{Z}}
\newcommand{\BF}{\mathbb{F}}
\newcommand{\CO}{\mathcal{O}}
\newcommand{\Res}{\mathsf{Res}}
\newcommand{\MatData}{\mathsf{MatData}}
\DeclareMathOperator{\corank}{corank}
\DeclareMathOperator{\JointFourRank}{Joint4rk}
\DeclareMathOperator{\Assemble}{Asm}
\DeclareMathOperator{\rank}{rank}
\DeclareMathOperator{\rk}{rk}
\DeclareMathOperator{\GL}{GL}
\DeclareMathOperator{\Cl}{Cl}
\DeclareMathOperator{\Gal}{Gal}
\DeclareMathOperator{\Aut}{Aut}
\DeclareMathOperator{\diag}{diag}
\DeclareMathOperator{\im}{im}
\DeclareMathOperator{\rad}{rad}

\DeclareMathOperator{\Mat}{Mat}
\DeclareMathOperator{\Sym}{Sym}
\DeclareMathOperator{\Alt}{Alt}
\DeclareMathOperator{\Unif}{Unif}
\DeclareMathOperator{\Law}{Law}
\DeclareMathOperator{\spn}{span}
\DeclareMathOperator{\supp}{supp}
\newcommand{\Prob}{\mathbb{P}}
\newcommand{\qerr}{q_{\mathrm{err}}}
\newcommand{\peq}{\overset{\mathbb P}{=}}
\newcommand{\pcong}{\overset{\mathbb P}{\cong}}
\newcommand{\Exp}{\mathbb{E}}
\newcommand{\ind}{\mathds{1}}
\newcommand{\units}[1]{\CO_{#1}^{\times}}
\newcommand{\aled}[2]{\left[\frac{#1}{#2}\right]}  
\newcommand{\asm}[2]{[\tfrac{#1}{#2}]}             
\newcommand{\bc}{\mathbf{c}}
\newcommand{\bu}{\mathbf{u}}
\newcommand{\bv}{\mathbf{v}}
\newcommand{\bw}{\mathbf{w}}
\newcommand{\bx}{\mathbf{x}}
\newcommand{\bX}{\mathbf{X}}
\newcommand{\by}{\mathbf{y}}
\newcommand{\bz}{\mathbf{z}}
\newcommand{\bone}{\mathbf{1}}
\newcommand{\bzero}{\mathbf{0}}
\newcommand{\bchi}{\boldsymbol{\chi}}
\newcommand{\bbeta}{\boldsymbol{\beta}}
\newcommand{\piCL}{\pi_{\mathrm{CL}}}
\newcommand{\pisym}{\pi_{\mathrm{sym}}}
\newcommand{\dtv}{d_{\mathrm{TV}}}
\newcommand{\gaussq}[2]{\genfrac{[}{]}{0pt}{}{#1}{#2}_{q}}
\newcommand{\gtwo}[2]{\genfrac{[}{]}{0pt}{}{#1}{#2}_{2}}
\newcommand{\momentseq}[1]{\mathbf m(#1)}
\newcommand{\momentSpace}[1]{\ell^{\infty}_{\theta}(\BZ_{\ge0}^{#1})}
\newcommand{\Cinv}{C_{\mathrm{inv}}}

\title[Asymptotic independence of class-group $4$-ranks]
{Asymptotic independence of class-group $4$-ranks in correlated pairs of imaginary quadratic fields}

\author{Yue Xu}
\address{School of Mathematics and Statistics, Xidian University, 266 Xinglong Section of Xifeng Road, Xi'an, Shaanxi 710126, China}
\email{xuyue@xidian.edu.cn}

\author{Xiuwu Zhu}
\address{Beijing Institute of Mathematical Sciences and Applications, Beijing 101408, China}
\email{xwzhu@bimsa.cn}


\date{}

\begin{document}

\begin{abstract}
Fix a squarefree integer $d_0>1$, and let $d$ range over the positive
squarefree integers coprime to $2d_0$.  Although
$\BQ(\sqrt{-d})$ and $\BQ(\sqrt{-d_0d})$ share all variable ramified
primes, we prove that their class-group $4$-ranks are asymptotically
independent.  Over the subfamily $d\le X$, their joint distribution
converges in total variation to the product of two copies of the
Cohen--Lenstra--Gerth distribution, with error bounded by a negative power
of $\log\log X$.  We further conjecture that the corrected $2$-primary groups
$2\Cl_{\BQ(\sqrt{-d})}[2^\infty]$ and
$2\Cl_{\BQ(\sqrt{-d_0d})}[2^\infty]$ are asymptotically independent,
each with the Cohen--Lenstra distribution.

Suppose in addition that the class number of
$\BQ(\sqrt{d_0})$ is odd.  For a density-one subset of this family,
we prove that extension of ideals to
$K(d)=\BQ(\sqrt{d_0},\sqrt{-d})$ induces
\[
  4\Cl_{K(d)}[2^\infty]\cong
  2\Cl_{\BQ(\sqrt{-d})}[2^\infty]\oplus
  2\Cl_{\BQ(\sqrt{-d_0d})}[2^\infty].
\]
Together with this decomposition, the group-valued conjecture predicts
that $4\Cl_{K(d)}[2^\infty]$ is distributed as the direct sum of two
independent Cohen--Lenstra $2$-groups, giving a corrected
Cohen--Lenstra--Martinet distribution for the biquadratic family.
Unconditionally, the $8$-rank of $\Cl_{K(d)}$ has limiting distribution
given by the convolution of two copies of the Cohen--Lenstra--Gerth
distribution.  The proof combines Smith's box method with quantitative
truncated Gaussian-binomial moment inversion for diagonally coupled, fixed-width
bordered R\'edei matrices.
\end{abstract}

\maketitle

\section{Introduction}\label{sec:intro}

\subsection{Motivation}\label{ssec:family}

The class number of $\BQ(\sqrt d)$ is not a multiplicative function of the
squarefree radicand $d$.
One may instead fix a positive squarefree multiplier $d_0$ and ask whether
the class groups of $\BQ(\sqrt d)$ and $\BQ(\sqrt{d_0d})$ remain
statistically dependent as $d$ varies.
Already for $d_0=2$, although $\BQ(\sqrt2)$ has class number one, no
pointwise factorization relates the two varying class numbers.
The Cohen--Lenstra--Martinet heuristics for biquadratic fields provide the
appropriate framework for this joint question.  Indeed, for coprime
squarefree integers $m$ and $n$, the fields $\BQ(\sqrt m)$,
$\BQ(\sqrt n)$ and $\BQ(\sqrt{mn})$ are the quadratic subfields
$k_0,k_1,k_2$ of $K=\BQ(\sqrt m,\sqrt n)$.  For every odd prime $\ell$,
extension of ideals and the norm maps give the decomposition
\[
  \Cl_K[\ell^\infty]\cong
  \Cl_{k_0}[\ell^\infty]\oplus
  \Cl_{k_1}[\ell^\infty]\oplus
  \Cl_{k_2}[\ell^\infty].
\]
The conductor--discriminant formula gives
$|\Delta_K|=|\Delta_{k_0}\Delta_{k_1}\Delta_{k_2}|\asymp m^2n^2$, so ordering
$K$ by discriminant gives the natural ordering for the simultaneous variation
of $m$ and $n$.  Under this ordering, the heuristics predict that, for every
odd prime $\ell$, the three $\ell$-primary class groups are distributed
independently \cite{CohenMartinet,WangWood}.  After one quadratic subfield is
fixed, the same heuristics predict independence of the two varying
subfield class groups.

The prime $2$ presents two distinct obstacles.  First, there is no
comparably general corrected Cohen--Lenstra--Martinet prediction for
biquadratic fields.  Known results, including the density-one $4$-rank
formula of Koymans--Morgan--Smit \cite{KMS}, concern lower ranks and do not
determine the higher $2$-primary structure in the family considered here.
Second, the direct-sum decomposition above need not hold integrally at $2$.
At the level of class numbers, Kuroda's formula records the discrepancy
through the Hasse unit index \cite{Lemmermeyer}; at the level of group
structure, non-split extension data provide a further obstruction.

We specialize to the following fixed-subfield family.  Fix a
squarefree integer $d_0>1$, let $d$ range over positive squarefree integers
coprime to $2d_0$, and write
\[
  k_0=\BQ(\sqrt{d_0}),\qquad
  k_1(d)=\BQ(\sqrt{-d}),\qquad
  k_2(d)=\BQ(\sqrt{-d_0d}),
\]
the three quadratic subfields of
$K(d)=\BQ(\sqrt{d_0},\sqrt{-d})$.  Every prime dividing $d$ ramifies in both
varying imaginary quadratic fields, whereas the primes dividing $d_0$ supply
fixed additional ramification to $k_2(d)$.  Genus theory consequently imposes
a deterministic relation between the two $2$-ranks, so the raw $2$-primary
class groups cannot be asymptotically independent.  Passing to their images
under multiplication by $2$ removes the genus-theoretic $2$-rank contribution
and leads instead to the pair
\[
  2\Cl_{k_1(d)}[2^\infty]\qquad\text{and}\qquad
  2\Cl_{k_2(d)}[2^\infty].
\]
The one-field marginals are known: Fouvry--Kl\"uners determined the
$4$-rank distribution \cite{FK4rank}, while Smith proved the full
one-field Cohen--Lenstra--Gerth distribution in his final two-part treatment
\cite[Theorem~1.10]{SmithTwistsI}; see also the general assembly theorem
\cite[Theorem~2.14]{SmithTwistsII}.  The question is whether the two
corrected groups are asymptotically independent, equivalently whether their
joint distribution converges to the product of the two marginals.

The joint question is better understood when the two fields are tied by an
additional arithmetic relation.  For $\BQ(\sqrt d)$ and
$\BQ(\sqrt{-d})$, the reflection principle both forces a correlation and
makes the joint distribution accessible.  Fouvry and Kl\"uners computed
this distribution at the level of $4$-ranks and showed that it does not
factor as the product of its marginals \cite{FKspiegel}.
Similar special relations underlie work on the $16$-ranks and joint-spin
statistics of $\BQ(\sqrt{-p})$ and $\BQ(\sqrt{-2p})$
\cite{Koymans16,KoymansMilovic16,KoymansMilovicSpins}.

The fixed-twist pair studied here lies outside this pattern.  Although
$k_1(d)$ and $k_2(d)$ share all variable ramified primes, no analogous
reflection relation is available.  Nor is the family a direct instance of
Smith's degree-$\ell$ twist families \cite{SmithTwistsI,SmithTwistsII}, since the
twisting classes are rational and the compositum remains $V_4$-Galois over
$\BQ$.  The results of Koymans--Liu on bad parts in odd-degree settings
likewise do not apply \cite{KoymansLiu}.  We nevertheless prove that the
two $4$-ranks are asymptotically independent, each with the
Cohen--Lenstra--Gerth distribution.

We order the family by the positive squarefree parameter
\begin{equation}\label{eq:F0}
  \mathcal F_{d_0}(X):=
  \{0<d\le X:\mu^2(d)=1,\ (d,2d_0)=1\}.
\end{equation}
The coprimality condition separates the variable ramification from the
fixed ramification contributed by $d_0$.  For a finite abelian group $A$
and $j\ge1$, write
\[
  \rk_{2^j}A:=\dim_{\BF_2}(2^{j-1}A/2^jA),
  \qquad r_4(n):=\rk_4\Cl\bigl(\BQ(\sqrt n)\bigr).
\]
For $a,b\ge0$, define the finite-level joint distribution by
\begin{equation}\label{eq:empirical-distribution}
  \mu_X^{(d_0)}(a,b)
  :=\frac{
    \#\{d\in\mathcal F_{d_0}(X):
       r_4(-d)=a,\ r_4(-d_0d)=b\}}
    {\#\mathcal F_{d_0}(X)}.
\end{equation}
Thus $\mu_X^{(d_0)}$ is a probability measure on $\BZ_{\ge0}^2$.
For $q>1$ and $m\ge0$, put
\[
  \eta_m(q):=\prod_{j=1}^{m}(1-q^{-j}),
  \qquad
  \eta_\infty(q):=\prod_{j\ge1}(1-q^{-j}),
\]
with $\eta_0(q)=1$.
We denote by $\piCL$ the Cohen--Lenstra--Gerth distribution on
$\BZ_{\ge0}$ \cite{CohenLenstra,Gerth4rank},
\begin{equation}\label{eq:piCL}
  \piCL(m)=2^{-m^2}\frac{\eta_\infty(2)}{\eta_m(2)^2}.
\end{equation}
We also write $\pisym$ for the limiting corank distribution of uniform symmetric matrices over
$\BF_2$,
\begin{equation}\label{eq:pisym}
  \pisym(m)
  =\frac{\eta_\infty(2)}{\eta_\infty(4)}
    \frac{2^{-m(m+1)/2}}{\eta_m(2)}.
\end{equation}

\subsection{Main results}

The principal theorem gives effective asymptotic independence in
the family~\eqref{eq:F0}.
For probability measures $\mu$ and $\nu$ on a countable set $S$, write
\[
  \dtv(\mu,\nu)
  =\frac12\sum_{x\in S}\bigl|\mu(\{x\})-\nu(\{x\})\bigr|.
\]

\begin{theorem}\label{thm:B}
Let $d_0>1$ be squarefree.  There exists $c=c(d_0)>0$ such that
\begin{equation}\label{eq:main-TV}
  \dtv\bigl(\mu_X^{(d_0)},\piCL\otimes\piCL\bigr)
  \ll_{d_0}(\log\log X)^{-c}.
\end{equation}
\end{theorem}

Projecting \eqref{eq:main-TV} onto either coordinate recovers the marginal
convergence to $\piCL$ recalled above.  The new assertion is their joint
asymptotic independence, expressed by convergence to
$\piCL\otimes\piCL$.

\paragraph{The case \(d_0=2\).}
Here \(d\) is odd.  As \(d\) ranges over the odd squarefree integers,
\(\BQ(\sqrt{-d})\) and \(\BQ(\sqrt{-2d})\) run respectively through the
families of imaginary quadratic fields with odd and even squarefree radicand;
together these two families exhaust all imaginary quadratic fields.
Theorem~\ref{thm:B} shows that, under the natural pairing \(d\mapsto2d\),
their class-group \(4\)-ranks are asymptotically independent.

Theorem~\ref{thm:B} is the $4$-rank projection of the following conjectural
joint distribution.  Let $\mathcal G_2$ denote the set
of isomorphism classes of finite abelian $2$-groups, identifying a group with
its isomorphism class, and equip $\mathcal G_2$ with the Cohen--Lenstra measure
\begin{equation}\label{eq:CL-group-measure}
  \mu_{\mathrm{CL},2}(G)
  =\frac{\eta_\infty(2)}{|\Aut(G)|},
  \qquad G\in\mathcal G_2.
\end{equation}
The classical identity
$\sum_{G\in\mathcal G_2}|\Aut(G)|^{-1}=\eta_\infty(2)^{-1}$ shows that this
is a probability measure.  Gerth's correction predicts that
$2\Cl_k[2^\infty]$, rather than $\Cl_k[2^\infty]$ itself, is governed by this
measure \cite{GerthExtension}; Smith proved this prediction for imaginary
quadratic fields \cite[Theorem~1.10]{SmithTwistsI}.  We propose the following
asymptotic independence conjecture for the two varying fields.

\begin{conjecture}\label{conj:fixed-twist}
Let $d_0>1$ be squarefree.  For every $G_1,G_2\in\mathcal G_2$,
\[
 \lim_{X\to\infty}\Prob_{d\in\mathcal F_{d_0}(X)}
 \bigl(2\Cl_{k_i(d)}[2^\infty]\simeq G_i\text{ for }i=1,2\bigr)
 =\prod_{i=1}^{2}\mu_{\mathrm{CL},2}(G_i).
\]
\end{conjecture}

This conjecture is specific to positive twists.  When $d_0=-1$, reflection
produces a genuine correlation between a real and an imaginary quadratic
field.  For every fixed $d_0>1$, Theorem~\ref{thm:B} proves the $4$-rank
projection of Conjecture~\ref{conj:fixed-twist}; the higher \(2^j\)-ranks are
not addressed here.

Under an odd-class-number hypothesis on the fixed real quadratic subfield, the
following direct-sum decomposition transfers both the theorem and the
conjecture to $K(d)$.

\begin{theorem}
\label{thm:biquad-decomp-intro}
Assume that the class number of $k_0=\BQ(\sqrt{d_0})$ is odd.  For all
but $o(\#\mathcal F_{d_0}(X))$ values $d\in\mathcal F_{d_0}(X)$, the Hasse unit
index
\[
 q_{K(d)}=
 [\CO_{K(d)}^\times:
   \CO_{k_0}^\times\CO_{k_1(d)}^\times\CO_{k_2(d)}^\times]
\]
equals $1$, and extension of ideals induces an isomorphism
\[
  4\Cl_{K(d)}[2^\infty]
  \cong\bigoplus_{i=0}^{2}2\Cl_{k_i(d)}[2^\infty].
\]
\end{theorem}

Theorem~\ref{thm:biquad-decomp-intro} turns
Conjecture~\ref{conj:fixed-twist} into the following corrected
Cohen--Lenstra--Martinet prediction for the biquadratic field.

\begin{conjecture}
\label{conj:biquad-clm}
Assume that the class number of $k_0$ is odd.  For every
$H\in\mathcal G_2$,
\[
 \lim_{X\to\infty}\Prob_{d\in\mathcal F_{d_0}(X)}
 \bigl(4\Cl_{K(d)}[2^\infty]\simeq H\bigr)
 =\sum_{\substack{G_1,G_2\in\mathcal G_2\\G_1\oplus G_2\simeq H}}
   \mu_{\mathrm{CL},2}(G_1)\mu_{\mathrm{CL},2}(G_2).
\]
\end{conjecture}

Taking $4$-ranks of the two varying quadratic class groups, and hence the
$8$-rank in the decomposition above, gives the following unconditional
consequence.

\begin{corollary}
\label{cor:biquad-intro}
Assume that the class number of $k_0$ is odd.  For all but
$o(\#\mathcal F_{d_0}(X))$ values $d\in\mathcal F_{d_0}(X)$,
\[
  \rk_8\Cl_{K(d)}=r_4(-d)+r_4(-d_0d).
\]
Moreover, for every $m\ge0$,
\[
  \lim_{X\to\infty}
  \Prob_{d\in\mathcal F_{d_0}(X)}
  \bigl(\rk_8\Cl_{K(d)}=m\bigr)
  =
  \sum_{a+b=m}\piCL(a)\piCL(b),
\]
and this convergence of distributions holds in total variation.
\end{corollary}

\subsection{Proof strategy and random-matrix results}

The $4$-rank of a quadratic class group is determined by the corank of its
R\'edei matrix.  For the two fields in our family, deleting one variable prime
produces a common core: the fixed twist by $d_0$ appears as a correlated
diagonal perturbation, while the prime divisors of $2d_0$ contribute a fixed
number of border rows and columns.  Smith's box method \cite{Smith2017}
shows that the symmetrised distribution of the relevant quadratic-residue
symbols is close to uniform.  Since the corank pair is invariant under
relabelling the variable primes, this reduces the arithmetic problem to the
joint corank distribution of a pair of coupled, fixed-width bordered R\'edei
matrices.  We use Smith's permutation-averaged equidistribution theorem for a
\emph{complete} residue assignment,
comprising the pair
symbols of the variable primes together with the columns $(-1/p_i)$,
$(2/p_i)$ and $(q_k/p_i)$ attached to the fixed integer $d_0$; the sign vector
is thereby part of the equidistributed data rather than a condition imposed on
the family.  For a published implementation of the box method in the
all-$1\bmod4$ family, see Chan--Koymans--Milovic--Pagano
\cite[\S\S4--5]{CKMP}; the mixed-sign specialization needed here is carried
out in Section~\ref{ssec:assign}.

For a single unbordered R\'edei matrix, Koymans--Pagano obtained effective
convergence by exposing one row and column at a time and comparing the resulting
corank jumps with an ergodic Markov chain \cite{KP2024}.  A different
single-matrix approach, based on the spectrum of the corank transition
operator, was developed in our earlier work \cite{XZ}.  These one-parameter
methods do not directly close for the coupled bordered matrix pair above: one must simultaneously
track two matrices sharing a core and diagonal data, and the borders impose
additional kernel conditions that are not preserved by a single corank
transition.

Two related moment normalizations have been used to study class-group and
random-cokernel distributions.  Ordinary moments of group size occur in the
work of Fouvry--Kl\"uners and Smith \cite{FK4rank,SmithTwistsII}; systematic
treatments of surjection moments in the random-group and random-cokernel
setting may be found in \cite{WoodRandomMatrices,WangWood}.  Mixed moments also
appear in the work of Pagano--Sofos on ray class sequences
\cite{PaganoSofos}.  In these
applications, moments are principally used to identify a limiting distribution through
a uniqueness theorem.  Here fixed-order mixed moments are not enough:
estimates through a growing order must be transferred to the joint distribution while
retaining a uniform error.  We do this by quantitative truncated inversion of
the Gaussian-binomial moment transform.

For a corank variable $X$, consider the Gaussian binomial moments
\[
  m_u(X)=\Exp\gtwo Xu,
\]
which count, on average, the $u$-dimensional subspaces of the kernel.  The
$q$-binomial transform taking a distribution to these moments is triangular
and has an explicit inverse.  On a suitable weighted supremum space, this
inverse is bounded into $\ell^1$.  Concretely,
if $\rho$ and $\nu$ are two probability distributions on $\BZ_{\ge0}$,
$\sigma=\rho-\nu$, and
$\theta_u=2^{-u(u+1)/2}$, then the one-dimensional truncated inversion gives
\[
 \|\sigma\|_1
 \le \Cinv\max_{0\le u\le D}
       \frac{|m_u(\sigma)|}{\theta_u}
   +C(D+1)2^{D(D+1)/2}m_{D+1}(|\sigma|).
\]
The first term uses only the moments through order $D$, while the second
controls the discarded tail by the $(D+1)$-st absolute moment.  For a signed
measure on $\BZ_{\ge0}^2$, applying the tensor product of the two
one-dimensional inverse operators and estimating the bivariate residual gives
an analogous truncated inversion inequality whose inputs are the genuinely
mixed moments
\[
  m_{u,v}(X,Y)=\Exp\!\left[\gtwo Xu\gtwo Yv\right].
\]
These mixed-moment estimates are proved directly for the coupled matrices;
they do not follow by multiplying one-dimensional error estimates.  Once they
are available through a growing order, balancing the truncation depth against
the matrix size gives a total-variation bound of the form $C\eta^n$ with
$\eta\in(0,1)$.

The unbordered case already gives an asymptotic independence theorem for diagonal
perturbations of a common core.  For $\bu_{-1}\in\BF_2^n$, put
\[
  \Omega_{\bu_{-1}}=\bu_{-1}\bu_{-1}^{\top}+\diag(\bu_{-1}),\qquad
  \Sym_n(\Omega_{\bu_{-1}})
  =\{A\in\Mat_n(\BF_2):A^{\top}-A=\Omega_{\bu_{-1}}\}.
\]
For a random element $W$, write $\Law(W)$ for its probability distribution.

\begin{theorem}
\label{thm:rmt-intro}
Let $A$ be uniform in $\Sym_n(\Omega_{\bu_{-1}})$, and let
$\bv\in\BF_2^n$ be uniform and independent of $A$.  Set
\[
  X=\corank A,\qquad
  Y=\corank\bigl(A+\diag(\bv)\bigr).
\]
There are constants $C<\infty$ and $\eta\in(0,1)$ such that, if
$\bu_{-1}=\bzero$, then
\[
  \bigl\|\Law(X,Y)-\pisym\otimes\pisym\bigr\|_1\le C\eta^n.
\]
For every $\alpha>0$, there are
$C_\alpha<\infty$ and $\eta_\alpha\in(0,1)$ such that, uniformly whenever
$\rank\Omega_{\bu_{-1}}\ge\alpha n$,
\[
  \bigl\|\Law(X,Y)-\piCL\otimes\piCL\bigr\|_1
  \le C_\alpha\eta_\alpha^n.
\]
\end{theorem}

When $\bu_{-1}=\bzero$, the matrix $A$ is uniform symmetric and
$\diag(\bv)$ is an independent uniform random diagonal matrix.
Thus the first estimate states directly that a random diagonal perturbation
makes the two coranks asymptotically independent, with common limiting
distribution $\pisym$.  This finite-field statement concerns only coranks;
the results of Lee and Jung--Lee discussed below instead determine joint
cokernel modules for $p$-adic matrix ensembles.

The diagonal need not be random.  If $A$ is a uniform symmetric matrix over
$\BF_2$, the same mixed-moment estimate
gives constants $C<\infty$ and $\eta\in(0,1)$ such that
\[
  \left\|
  \Law\bigl(\corank A,\corank(A+I_n)\bigr)
  -\pisym\otimes\pisym
  \right\|_1
  \le C\eta^n.
\]
Thus a fixed diagonal shift of full support can make the two coranks
asymptotically independent.

Theorem~\ref{thm:rmt-intro} is proved in the more precise form of
Theorem~\ref{thm:diag}.  Lee and Jung--Lee determine joint cokernel modules
for polynomial or shifted transforms of Haar random $\BZ_p$-matrices
\cite{LeeJoint,JungLeeII,LeeMixed}.  Those results concern different
$p$-adic ensembles and determine full joint cokernel modules.  Our theorem
treats only the coranks of symmetry-defect fibres over $\BF_2$ and gives an
error at most $C_\alpha\eta_\alpha^n$ in total variation, uniformly when the
defect has rank at least $\alpha n$.  In the arithmetic application,
$\bv=\bu_{d_0}$ records the fixed twist by \(d_0\), and
$A,A+\diag(\bu_{d_0})$ are the common cores of the two R\'edei matrices.

In the actual R\'edei matrices, every prime-discriminant factor of $2d_0$
contributes border and corner data of bounded width.
Theorem~\ref{thm:rmt0} treats the resulting bordered pairs uniformly in this
data and in every resulting $2$-adic stratum.  Writing \(r=\omega(d)\), it
gives an \(\ell^1\)-error \(O_{d_0}(\vartheta^r)\), for some
\(\vartheta\in(0,1)\), from \(\piCL\otimes\piCL\) whenever the symmetry defect
has rank at least a fixed positive proportion of \(r\).  The sign vectors not
covered by this rank condition have exponentially small proportion in the
uniform residue-symbol distribution.  Applying the resulting estimate on all
admissible boxes completes the proof of Theorem~\ref{thm:B}.

\subsection{Organization of the paper}

Section~\ref{sec:model} proves the R\'edei reduction, defines the residue and
matrix probability spaces, identifies their uniform measures by an affine
bijection, writes the corank-pair pushforward as an explicit sign mixture, and
collects the finite-field tools.  Section~\ref{sec:box}
decomposes the family~\eqref{eq:F0} into admissible boxes, proves
equidistribution of the symmetrised full residue assignment, and transfers
this estimate through the finite pushforward to the original family.
Section~\ref{sec:rmt} develops one- and two-dimensional quantitative truncated
inversion for Gaussian-binomial moments, proves the bordered joint-corank
theorem from conditional mixed moments on a good event, and records the
unbordered diagonal-perturbation consequences.  Section~\ref{sec:proof} applies
the family-transfer theorem and the bordered joint-corank estimate to prove
Theorem~\ref{thm:B}, and then determines the Hasse unit index and the
direct-sum defect needed for Theorem~\ref{thm:biquad-decomp-intro} and
Corollary~\ref{cor:biquad-intro}.

\subsection*{Acknowledgements}

The authors thank Yi Ouyang and Ye Tian for their support and helpful
discussions.  The first author was partially supported by the Tianyuan Fund for
Mathematics of the National Natural Science Foundation of China (Grant
No.~12526538) and by the Postdoctoral Fellowship Program of CPSF (Grant
No.~GZC20252038).
The authors used generative-AI tools for bibliographic searches, language
editing, and technical cross-checking, and take full responsibility for the
content.

\section{R\'edei reduction and finite-field counting}\label{sec:model}

This section collects the algebraic and finite-field preliminaries needed in
the proof.  We first recall the R\'edei--Reichardt description of the $4$-rank
and the coordinate deletion that converts it into a matrix corank.  Applying
this construction simultaneously to $-d$ and $-d_0d$ produces two matrices
over $\BF_2$ with a common core.  We then regard the residue symbols occurring
in these matrices as formal variables and introduce finite probability spaces
for the resulting pair of coranks.  The second part records the finite-field
counting results needed to analyze these spaces.

Throughout the remainder of the paper, vectors over finite fields are regarded
as column vectors, and we reserve $\rank$ and $\corank$ for matrices and
linear maps.  The bracket $\left[\frac an\right]\in\BF_2$ denotes the additive
Kronecker symbol, and hence the additive Jacobi or Legendre symbol in the
usual cases.

\subsection{From residue symbols to joint 4-ranks}
\label{ssec:corank-pair}

\subsubsection{R\'edei matrices and coordinate deletion}\label{ssec:RR}

Let $\Delta$ be a fundamental discriminant, put
$K_\Delta=\BQ(\sqrt\Delta)$, and choose an ordering
\[
  \Delta=\Delta_1\cdots\Delta_t
\]
of its prime discriminant factors, where
$\Delta_i\in\{-4,8,-8\}\cup\{p^*:p\text{ odd prime}\}$, with
$p^*=(-1)^{(p-1)/2}p$.  Let $p_i$ be the prime below $\Delta_i$, with
$p_i=2$ at the even prime discriminant.  Define
the R\'edei matrix $R_\Delta\in\Mat_t(\BF_2)$ by
\begin{equation}\label{eq:RRmatrix}
  (R_\Delta)_{ij}=\aled{\Delta_i}{p_j}\quad(i\ne j),
  \qquad
  (R_\Delta)_{jj}=\sum_{i\ne j}(R_\Delta)_{ij}.
\end{equation}
Thus every column of $R_\Delta$ sums to zero.

The following criterion is due to R\'edei and Reichardt
\cite{RedeiReichardt}.
\begin{theorem}[R\'edei--Reichardt criterion]
\label{thm:RR}
The $4$-rank of the narrow class group of $K_\Delta$ is given by
\begin{equation}\label{eq:RR}
  \rk_4\Cl^{+}(K_\Delta)=\corank R_\Delta-1.
\end{equation}
\end{theorem}

Suppose henceforth that the fundamental discriminant $\Delta$ is negative, and retain
the factorisation and notation above.  The ordinary and narrow class groups of
$K_\Delta$ then coincide.  Let $I$ be the index set of
the prime discriminants and write
$\Delta_{\mathrm{even}}\in\{1,-4,8,-8\}$ for the even prime discriminant, with
value $1$ when none occurs; in that case no even index is introduced.  For a
condition $E$, let
$\ind_{\{E\}}\in\{0,1\}$ denote its indicator, viewed in \(\BF_2\) when it
occurs in a vector or matrix.  Without a subscript,
$\mathbf1$ denotes the all-ones vector whose length is determined by context.
Define
\[
  \bu_{-1}(\Delta)
  =\bigl(\ind_{\{\Delta_i<0\}}\bigr)_{i\in I},
  \qquad
  \bigl(\bu_2(\Delta)\bigr)_i=
  \begin{cases}
    \aled{2}{p_i},&p_i\ne2,\\
    0,&p_i=2.
  \end{cases}
\]
At an odd index,
$\bigl(\bu_{-1}(\Delta)\bigr)_i=\left[\frac{-1}{p_i}\right]$.
At the even index the same formula cannot be used, since
$\left[\frac{-1}{2}\right]=0$ does not distinguish a positive even prime
discriminant from a negative one.  For any vector $\mathbf x$ indexed by a
finite set, let $\diag(\mathbf x)$ denote the diagonal matrix with diagonal
entries $x_i$, and write
\[
  \Omega_{\mathbf x}
  =\mathbf x\mathbf x^{\top}+\diag(\mathbf x).
\]
If an even index occurs, we label it by $0$ and denote its standard basis
vector by $\mathbf e_0$.

\begin{proposition}\label{prop:defect}
Let $\Delta<0$.  Then
\begin{equation}\label{eq:defect}
  R_\Delta^{\top}-R_\Delta
  =\Omega_{\bu_{-1}(\Delta)}
   +\ind_{\{\Delta_{\mathrm{even}}=-4\}}
    \bigl(\mathbf e_0\bu_2(\Delta)^{\top}
          +\bu_2(\Delta)\mathbf e_0^{\top}\bigr).
\end{equation}
In particular, the symmetry defect \(R_\Delta^{\top}-R_\Delta\) is determined
by the signs of the prime discriminants and the symbols at $2$.  It contains no pair symbol
$\aled{p_i^*}{p_j}$.
\end{proposition}

\begin{proof}
For distinct odd primes $p,q$, quadratic reciprocity and the definition of the
Kronecker symbol give
\begin{equation}\label{eq:local-reciprocity}
\begin{gathered}
  \aled{p^*}{q}+\aled{q^*}{p}
    =\aled{-1}{p}\aled{-1}{q},
  \qquad
  \aled{p^*}{2}=\aled{2}{p},\\[5pt]
  \aled{\Delta_{\mathrm{even}}}{p}+\aled{p^*}{2}
    =
    \begin{cases}
      \aled{-1}{p}+\aled{2}{p},&\Delta_{\mathrm{even}}=-4,\\
      0,&\Delta_{\mathrm{even}}=8,\\
      \aled{-1}{p},&\Delta_{\mathrm{even}}=-8.
    \end{cases}
\end{gathered}
\end{equation}
The first identity gives the assertion at two odd indices, and the second line
handles an even and an odd index.  Only $\Delta_{\mathrm{even}}=-4$
contributes the additional $\bu_2(\Delta)$ term.  Both sides have zero
diagonal.
\end{proof}

The defect formula turns the column-sum convention in \eqref{eq:RRmatrix} into
the two relations
\begin{equation}\label{eq:row-column}
  \mathbf1^{\top}R_\Delta=0,
  \qquad
  R_\Delta\mathbf1=(R_\Delta^{\top}-R_\Delta)\mathbf1.
\end{equation}
The first is the defining column-sum relation.  Its transpose gives
\(R_\Delta^\top\mathbf1=0\), and over \(\BF_2\) this implies the second
identity.  Together they allow one odd prime discriminant to be removed
without changing the relevant corank.

\begin{proposition}\label{prop:trunc}
Let $\rho$ be an odd index, so that $\Delta_\rho=p_\rho^*$.  If
$R_\Delta^{(\rho)}$ denotes the principal submatrix obtained by deleting row
and column $\rho$, then
\begin{equation}\label{eq:trunc-main}
  \rk_4\Cl(K_\Delta)=\corank R_\Delta^{(\rho)}.
\end{equation}
Thus the right-hand side is independent of the odd index chosen for deletion.
\end{proposition}

\begin{proof}
Add all columns of $R_\Delta$ other than column $\rho$ to column $\rho$.
By \eqref{eq:row-column}, the new column is
\((R_\Delta^{\top}-R_\Delta)\mathbf1\).
The same equation shows that row $\rho$ is the sum of the remaining rows, so
deleting it does not change the rank.  The resulting $(t-1)\times t$ matrix is
\[
  \begin{pmatrix}
    R_\Delta^{(\rho)}
    &\bigl((R_\Delta^{\top}-R_\Delta)\mathbf1\bigr)_{I\setminus\{\rho\}}
  \end{pmatrix}.
\]

Since $\Delta<0$, the number of negative prime discriminants is odd, and hence
$\mathbf1^{\top}\bu_{-1}(\Delta)=1$ in $\BF_2$.  Formula
\eqref{eq:defect} therefore gives
\[
  (R_\Delta^{\top}-R_\Delta)\mathbf1
  =\ind_{\{\Delta_{\mathrm{even}}=-4\}}
   \bigl(\mathbf e_0(\mathbf1^{\top}\bu_2(\Delta))
         +\bu_2(\Delta)\bigr).
\]
If $\Delta_{\mathrm{even}}\ne-4$, this vector is zero.  If
$\Delta_{\mathrm{even}}=-4$, its restriction to
$I\setminus\{\rho\}$ is column $0$ of $R_\Delta^{(\rho)}$.  Indeed, this
follows from $\aled{p_i^*}{2}=\aled2{p_i}$ and the column-sum definition of
the diagonal.
In either case the last column lies in the image of
$R_\Delta^{(\rho)}$.  Hence
\[
  \rank R_\Delta=\rank R_\Delta^{(\rho)},\qquad
  \corank R_\Delta=\corank R_\Delta^{(\rho)}+1,
\]
and \eqref{eq:RR} proves the result.
\end{proof}

The oddness of the deleted index is essential.  At an even index the last
column in the proof need not lie in the image of the principal submatrix.

\subsubsection{The coupled matrix pair}\label{ssec:block}

We apply Proposition~\ref{prop:trunc} simultaneously to the R\'edei matrices
for \(-d\) and \(-d_0d\).  After one common odd-prime coordinate is deleted,
the two reduced matrices share a variable block; their remaining coupling lies
in diagonal corrections and fixed-width borders.

Fix squarefree integers $d_0>1$ and $d$ with $(d,2d_0)=1$, and write
\[
  \begin{aligned}
  d_0&=2^{e}q_1\cdots q_s,
       &e&\in\{0,1\},&q_1&<\dots<q_s\ \text{odd},\\
  d&=p_1\cdots p_r,
       &&&p_1&<\dots<p_r,\qquad r=\omega(d)\ge1.
  \end{aligned}
\]
For a squarefree integer \(m\), write \(\Delta(m)\) for the fundamental
discriminant of \(\BQ(\sqrt m)\), with \(\Delta(1)=1\).  Put
\[
  \kappa=\#\{i:p_i\equiv3\bmod4\},\qquad
  P=\prod_{i\le r}p_i^*=(-1)^\kappa d.
\]
The prime-discriminant factorizations split as
\begin{equation}\label{eq:discriminant-split}
  \Delta(-d)=P\,\Delta((-1)^{\kappa+1}),\qquad
  \Delta(-d_0d)=P\,\Delta((-1)^{\kappa+1}d_0).
\end{equation}
Thus \(P\) is the common variable factor, while the remaining factors determine
the fixed widths.  The first is \(1\) when \(\kappa\) is odd and \(-4\) when
\(\kappa\) is even, so put
\[
  b_1=\ind_{\{2\mid\kappa\}}.
\]
Factor the second fixed discriminant as
\[
  \Delta((-1)^{\kappa+1}d_0)
  =\prod_{k\in\mathcal D_2}\Delta_k,\qquad
  b_2=|\mathcal D_2|,
\]
where \(\Delta_k=q_k^*\) for \(1\le k\le s\).  If an even prime
discriminant occurs, denote it by \(\Delta_0\), index it by \(0\), and put
\(q_0=2\).  Thus the two full R\'edei matrices have orders \(r+b_1\) and
\(r+b_2\), and their reductions after deleting one variable coordinate have
orders \(r-1+b_1\) and \(r-1+b_2\).

The corresponding fixed and variable R\'edei blocks are
\[
 R_{d_0,\kappa}:=R_{\Delta((-1)^{\kappa+1}d_0)}
 \in\Mat_{b_2}(\BF_2),
 \qquad
 R_d:=R_P\in\Mat_r(\BF_2),
\]
with the rows and columns of \(R_{d_0,\kappa}\) indexed by
\(\mathcal D_2\).
For every integer \(\lambda\) coprime to \(d\), write
\[
  u_{\lambda,i}=\aled{\lambda}{p_i},\qquad
  \bu_\lambda=(u_{\lambda,i})_{i\le r}\in\BF_2^r.
\]
In particular, \(\kappa=|\bu_{-1}|\), and
\begin{equation}\label{eq:variable-defect}
  R_d^{\top}-R_d=\Omega_{\bu_{-1}}.
\end{equation}
For a vector indexed by $p_1,\dots,p_r$, a superscript ${}^{\circ}$ means that
the last coordinate is deleted.  For a matrix whose rows have these indices,
it means that the last row is deleted; if its columns have the same indices,
the last column is deleted as well.

\begin{theorem}\label{thm:block}
For \(k\in\mathcal D_2\), put
\begin{equation}\label{eq:corner-c}
  c_k=\aled{P}{q_k}
  =
  \begin{cases}
    \aled{2}{d},&k=0,\\[2pt]
    \aled{(-1)^\kappa d}{q_k},&1\le k\le s.
  \end{cases}
\end{equation}
The two \(4\)-ranks are the coranks of the following reductions of the
R\'edei matrices:
\[
  r_4(-d)=\corank N_1,\qquad
  N_1=
  \begin{cases}
    R_d^\circ,&b_1=0,\\[4pt]
    \begin{pmatrix}
      R_d^\circ+\diag(\bu_{-1}^\circ) & \bu_2^\circ\\
      (\bu_{-1}^\circ)^{\top} & \aled{2}{d}
    \end{pmatrix},&b_1=1,
  \end{cases}
\]
and
\[
  r_4(-d_0d)=\corank N_2,\qquad
  N_2=
  \begin{pmatrix}
    R_d^\circ+b_1\diag(\bu_{-1}^\circ)+\diag(\bu_{d_0}^\circ)
      &(\bu_{q_k}^\circ)_{k\in\mathcal D_2}\\[3pt]
    \bigl((\bu_{\Delta_k}^\circ)_{k\in\mathcal D_2}\bigr)^{\top}
      &R_{d_0,\kappa}
       +\diag\bigl((c_k)_{k\in\mathcal D_2}\bigr)
  \end{pmatrix}.
\]
\end{theorem}

\begin{proof}
Between two variable indices, the off-diagonal entries of both full R\'edei
matrices are those of \(R_d\).  The quotient \(\Delta(-d)/P\) is \(1\) or
\(-4\), according as \(b_1=0\) or \(1\).  Its contribution to the diagonal
at \(p_i\) is therefore \(b_1u_{-1,i}\), which gives the displayed common
block.

The two fixed diagonal contributions differ by
\[
  \aled{\Delta(-d)/P}{p_i}
  +\aled{\Delta(-d_0d)/P}{p_i}
  =\aled{\Delta(-d)\Delta(-d_0d)/d^2}{p_i}
  =\aled{d_0}{p_i}.
\]
Indeed, $\Delta(-d)\Delta(-d_0d)/(d_0d^2)$ is one of $1,4,16$.
Consequently the second variable block differs from the first by
\(\diag(\bu_{d_0}^\circ)\).

It remains to identify the borders and corners.
The upper and lower border entries are the vectors
\(\bu_{q_k}\) and \(\bu_{\Delta_k}\), respectively, by quadratic
reciprocity, including at the even index.  When \(b_1=1\), the corresponding
vectors for the first matrix are \(\bu_2\) and \(\bu_{-1}\), and its corner
entry is \(\aled{P}{2}=\aled{2}{d}\).

For the second corner block, the column-sum convention gives
\[
  (R_{\Delta(-d_0d)})_{kk}
  =\sum_{i\le r}\aled{p_i^*}{q_k}
   +\sum_{\substack{\ell\in\mathcal D_2\\\ell\ne k}}
      \aled{\Delta_\ell}{q_k}
  =\aled{P}{q_k}
   +\sum_{\substack{\ell\in\mathcal D_2\\\ell\ne k}}
      \aled{\Delta_\ell}{q_k},
\]
which is the displayed lower-right diagonal entry; the off-diagonal entries
are \(\aled{\Delta_k}{q_\ell}\).

Finally, apply Proposition~\ref{prop:trunc} to $\Delta(-d)$ and
$\Delta(-d_0d)$, deleting in each case the coordinate corresponding to $p_r$.
This gives the displayed matrices.
\end{proof}

\subsubsection{Residue assignments and the joint 4-rank map}
\label{ssec:finite-model}

Retain the factorisation \(d_0=2^eq_1\cdots q_s\) of
Section~\ref{ssec:block}.  Since \(d_0\) remains fixed, we suppress it from
the notation whenever no ambiguity can arise.  We now replace the primes by
formal residue assignments and then convert these assignments into the block
data of Theorem~\ref{thm:block}.  For \(r\ge1\), define
\begin{equation}\label{eq:residue-space}
 \Res_r
 :=\BF_2^{\{(i,j):1\le i<j\le r\}}\times(\BF_2^r)^{s+2}.
\end{equation}
We write its elements as
\[
 \bchi=\bigl((\chi_{ij})_{i<j};
   \bu_{-1},\bu_2,\bu_{\Delta_1},\dots,\bu_{\Delta_s}\bigr).
\]
For an ordered tuple of distinct odd primes
\((p_1,\dots,p_r)\), all coprime to \(2d_0\), its residue assignment is
\begin{equation}\label{eq:residue-data}
 \bchi(p_1,\dots,p_r)
 =\Bigl(\aled{p_i^*}{p_j}_{\,i<j};\
 \bu_{-1},\bu_2,\bu_{\Delta_1},\dots,\bu_{\Delta_s}\Bigr)
 \in\Res_r.
\end{equation}
When \(d=p_1\cdots p_r\) and \(p_1<\cdots<p_r\), we write this assignment as
\(\bchi(d)\).  Thus
\begin{equation}\label{eq:Rr-size}
 \#\Res_r=2^{\binom r2+(s+2)r},
\end{equation}
and for now \(\Res_r\) is regarded only as a finite set.

We first pass from residue symbols to matrix coordinates.
Fix \(r\ge2\) and a sign vector
\(\bu_{-1}=(u_{-1,1},\dots,u_{-1,r})\in\BF_2^r\).  Let
\[
 \Res_r(\bu_{-1})
 :=\{\bchi\in\Res_r:\text{its sign vector is }\bu_{-1}\},
\]
and put
\[
  n=r-1,\qquad \kappa=|\bu_{-1}|,\qquad
  b_1=\ind_{\{2\mid\kappa\}},\qquad
  \Omega=\Omega_{\bu_{-1}^\circ}.
\]

\begin{definition}\label{def:coset}
For \(\Omega'\in\Mat_n(\BF_2)\), set
\begin{equation}\label{eq:symmetry-defect-fibre}
 \Sym_n(\Omega')
 :=\{A\in\Mat_n(\BF_2):A^\top-A=\Omega'\}.
\end{equation}
\end{definition}

\begin{definition}\label{def:matrix-data}
The coupled-matrix data space over the sign fibre \(\bu_{-1}\) is
\begin{equation}\label{eq:matrix-data-space}
 \MatData_r(\bu_{-1})
 :=\Sym_n(\Omega)\times(\BF_2^n)^{\,1+s}\times\BF_2^{\,1+s}.
\end{equation}
We write its elements as
\[
 x=(A,\bu_2,\bu_{\Delta_1},\dots,\bu_{\Delta_s},\bc),
 \qquad \bc=(c_0,\dots,c_s).
\]
\end{definition}

For \(\bchi\in\Res_r(\bu_{-1})\), define
\(R(\bchi)\in\Mat_r(\BF_2)\) by
\[
 R(\bchi)_{ij}
 =
 \begin{cases}
   \chi_{ij},&i<j,\\
   \chi_{ji}+u_{-1,i}u_{-1,j},&j<i,
 \end{cases}
 \qquad
 R(\bchi)_{jj}=\sum_{i\ne j}R(\bchi)_{ij}.
\]
Then \(R(\bchi)^\top-R(\bchi)=\Omega_{\bu_{-1}}\).  Put
\[
 A(\bchi)
 :=R(\bchi)^\circ+b_1\diag(\bu_{-1}^\circ),
\]
so that \(A(\bchi)\in\Sym_n(\Omega)\), and set
\begin{equation}\label{eq:coordinate-change}
 c_0(\bchi)=\sum_{i\le r}u_{2,i},\qquad
 c_k(\bchi)=\sum_{i\le r}
 \left(u_{\Delta_k,i}+\aled{-1}{q_k}u_{-1,i}\right)
 \quad(1\le k\le s).
\end{equation}
The coordinate map is
\begin{equation}\label{eq:coordinate-map}
 \operatorname{coord}_{\bu_{-1}}(\bchi)
 :=
 \bigl(A(\bchi),\bu_2^\circ,\bu_{\Delta_1}^\circ,\dots,
       \bu_{\Delta_s}^\circ,(c_0(\bchi),\dots,c_s(\bchi))\bigr).
\end{equation}

\begin{lemma}\label{lem:coordinate-bijection}
The map
\[
 \operatorname{coord}_{\bu_{-1}}:
 \Res_r(\bu_{-1})\xrightarrow{\ \sim\ }
 \MatData_r(\bu_{-1})
\]
is an affine bijection.
\end{lemma}

\begin{proof}
The pair-symbol coordinates not involving \(r\) give the off-diagonal entries
of \(A\), while those involving \(r\) determine its diagonal through the
column-sum rule.  Hence they parametrize \(\Sym_n(\Omega)\) bijectively.  For
each of the \(s+1\) character vectors, its first \(n\) entries together with
the corresponding coordinate \(c_k\) recover its last entry
by~\eqref{eq:coordinate-change}.  These disjoint affine changes of coordinates
prove the claim.
\end{proof}

We next assemble a point of \(\MatData_r(\bu_{-1})\) into the two matrices of
Theorem~\ref{thm:block}.  Put
\(\kappa_0=\#\{k:q_k\equiv3\bmod4\}\).  We call
\((e,\kappa\bmod2,\kappa_0\bmod2)\) the \(2\)-adic stratum; it determines
\(\mathcal D_2\), the even prime discriminant \(\Delta_0\) when present, and
the border widths.  Explicitly,
\[
 \Delta_0=
 \begin{cases}
  -4,&e=0,\ \kappa+\kappa_0\ \text{even},\\
  8,&e=1,\ \kappa+\kappa_0\ \text{odd},\\
  -8,&e=1,\ \kappa+\kappa_0\ \text{even},
 \end{cases}
\]
with no even index when \(e=0\) and \(\kappa+\kappa_0\) is odd, and
\begin{equation}\label{eq:b2-width}
 b_1=\ind_{\{2\mid\kappa\}},
 \qquad
 b_2=s+e+(1-e)\ind_{\{2\mid(\kappa+\kappa_0)\}}.
\end{equation}

For \(x=(A,\bu_2,\bu_{\Delta_1},\dots,\bu_{\Delta_s},\bc)
\in\MatData_r(\bu_{-1})\), define the blocks below.
If \(0\in\mathcal D_2\), define
\[
 \bu_{\Delta_0}=
 \begin{cases}
   \bu_{-1}^\circ,&\Delta_0=-4,\\
   \bu_2,&\Delta_0=8,\\
   \bu_{-1}^\circ+\bu_2,&\Delta_0=-8.
 \end{cases}
\]
The remaining blocks are
\begin{equation}\label{eq:u-d0}
 \bu_{d_0}=b_1\bu_{-1}^\circ
            +\sum_{k\in\mathcal D_2}\bu_{\Delta_k},
\end{equation}
\begin{equation}\label{eq:U-derived}
 \bu_{q_k}=
 \begin{cases}
   \bu_2,&k=0,\\[2pt]
   \bu_{\Delta_k}+\aled{-1}{q_k}\bu_{-1}^\circ,&1\le k\le s,
 \end{cases}
 \qquad(k\in\mathcal D_2),
\end{equation}
and
\[
 \mathsf L=(\bu_{\Delta_k})_{k\in\mathcal D_2},\qquad
 \mathsf U=(\bu_{q_k})_{k\in\mathcal D_2},\qquad
 C=R_{d_0,\kappa}+\diag\bigl((c_k)_{k\in\mathcal D_2}\bigr).
\]
If \(b_1=0\), let \(\mathsf U_1,\mathsf L_1,\mathsf C_1\) be empty blocks of
compatible sizes; if \(b_1=1\), put
\[
 \mathsf U_1=\bu_2,\qquad
 \mathsf L_1=\bu_{-1}^\circ,\qquad
 \mathsf C_1=(c_0).
\]
Define
\begin{equation}\label{eq:compact-reduced-pair}
 N_1(x)=
 \begin{pmatrix}
   A&\mathsf U_1\\
   \mathsf L_1^\top&\mathsf C_1
 \end{pmatrix},
 \qquad
 N_2(x)=
 \begin{pmatrix}
   A+\diag(\bu_{d_0})&\mathsf U\\
   \mathsf L^\top&C
 \end{pmatrix}.
\end{equation}
This is the assembly map
\begin{equation}\label{eq:assembly-map}
 \Assemble_{d_0,r,\bu_{-1}}:
 \MatData_r(\bu_{-1})
 \longrightarrow
 \Mat_{n+b_1}(\BF_2)\times\Mat_{n+b_2}(\BF_2),
 \qquad
 x\longmapsto\bigl(N_1(x),N_2(x)\bigr).
\end{equation}

For a matrix pair, write
\(\corank^{\times2}(N_1,N_2)=(\corank N_1,\corank N_2)\).
The coordinate and assembly maps define
\begin{equation}\label{eq:corank-factorization}
 \left.\JointFourRank_r\right|_{\Res_r(\bu_{-1})}
 :=
  \corank^{\times2}\circ
  \Assemble_{d_0,r,\bu_{-1}}\circ
  \operatorname{coord}_{\bu_{-1}}.
\end{equation}
Since the sign fibres partition \(\Res_r\), these composites give a single map
\begin{equation}\label{eq:corank-map}
 \JointFourRank_r:\Res_r\longrightarrow\BZ_{\ge0}^2.
\end{equation}
For every \(d\) in the family with \(\omega(d)=r\ge2\),
Theorem~\ref{thm:block} gives
\begin{equation}\label{eq:arithmetic-factorization}
 \bigl(r_4(-d),r_4(-d_0d)\bigr)
 =\JointFourRank_r\bigl(\bchi(d)\bigr).
\end{equation}
The cases \(r\le1\) have density zero and are omitted from this formalism.

We finally record how the construction behaves under relabelling and under
uniform measure.
The symmetric group \(\mathfrak S_r\) acts on \(\Res_r\) by relabelling the
variable primes; on single-prime coordinates,
\((\pi\cdot\bu_\lambda)_i=u_{\lambda,\pi^{-1}(i)}\).
For \(i<j\), quadratic reciprocity gives the pair-coordinate action
\begin{equation}\label{eq:perm-action}
 (\pi\cdot\bchi)_{ij}
 =\begin{cases}
   \chi_{\pi^{-1}(i),\pi^{-1}(j)},
     &\pi^{-1}(i)<\pi^{-1}(j),\\[2pt]
   \chi_{\pi^{-1}(j),\pi^{-1}(i)}
    +u_{-1,\pi^{-1}(i)}u_{-1,\pi^{-1}(j)},
     &\pi^{-1}(i)>\pi^{-1}(j).
  \end{cases}
\end{equation}
Thus
\(\pi\cdot\bchi(p_1,\dots,p_r)
 =\bchi(p_{\pi^{-1}(1)},\dots,p_{\pi^{-1}(r)})\), and every \(\pi\) acts by a
bijection.  The coordinate map itself uses the distinguished index \(r\) and
is not permutation invariant.

\begin{lemma}\label{lem:perm}
The map \(\JointFourRank_r\) is \(\mathfrak S_r\)-invariant.
\end{lemma}

\begin{proof}
Relabelling the variable indices conjugates the two unreduced matrices
determined by \(\bchi\) by the corresponding permutation matrix, with the
fixed indices unchanged.
Proposition~\ref{prop:trunc} makes the coranks of their reductions independent
of the deleted variable coordinate, proving the claim.
\end{proof}

For the probabilistic formulation, equip \(\Res_r\) and each
\(\MatData_r(\bu_{-1})\) with uniform probability measure.  Under
\(\Unif(\Res_r)\), the sign vector is uniform; conditional on its value
\(\bu_{-1}\), Lemma~\ref{lem:coordinate-bijection} identifies the uniform measures on
\(\Res_r(\bu_{-1})\) and \(\MatData_r(\bu_{-1})\).  The latter has independent
coordinates
\[
 A\sim\Unif\bigl(\Sym_n(\Omega)\bigr),\qquad
 \bu_2,\bu_{\Delta_1},\dots,\bu_{\Delta_s}\sim\Unif(\BF_2^n),\qquad
 \bc\sim\Unif(\BF_2^{\,1+s}).
\]
Let \(N_1,N_2\) be the resulting random matrices and put
\[
 \mathbf Z=(Z_1,Z_2)
 :=\corank^{\times2}(N_1,N_2).
\]
Then~\eqref{eq:corank-factorization} gives
\begin{equation}\label{eq:conditional-matrix-law}
 (\JointFourRank_r)_*\Unif\bigl(\Res_r(\bu_{-1})\bigr)
 =
 \Law_{\MatData_r(\bu_{-1})}(\mathbf Z).
\end{equation}

Averaging~\eqref{eq:conditional-matrix-law} over the sign vector gives, for
every \(r\ge2\),
\begin{equation}\label{eq:mixture}
 (\JointFourRank_r)_*\Unif(\Res_r)
 =2^{-r}\sum_{\bu_{-1}\in\BF_2^r}
 \Law_{\MatData_r(\bu_{-1})}(\mathbf Z).
\end{equation}

Permutation invariance also allows us to symmetrise a measure on \(\Res_r\)
without changing its joint \(4\)-rank pushforward.  For a probability measure
\(\mu\) on \(\Res_r\), put
\[
 \mu^{\mathrm{sym}}:=\frac1{r!}\sum_{\pi\in\mathfrak S_r}\pi_*\mu.
\]
Lemma~\ref{lem:perm} and contraction under pushforward give
\begin{equation}\label{eq:pushforward-symmetry}
 (\JointFourRank_r)_*\mu^{\mathrm{sym}}=(\JointFourRank_r)_*\mu,\qquad
 \|(\JointFourRank_r)_*\mu-(\JointFourRank_r)_*\nu\|_1
 \le\|\mu-\nu\|_1
\end{equation}
for all probability measures \(\mu,\nu\) on \(\Res_r\).

For later use, if \(0<\alpha<1/2\), Hoeffding's inequality gives
\begin{equation}\label{eq:sign-tail}
 2^{-r}\#\{\bu_{-1}\in\BF_2^r:|\bu_{-1}|<\alpha r\}
 \le e^{-2(1/2-\alpha)^2r}.
\end{equation}

\subsection{Counting for moments and truncated inversion}
\label{ssec:linalg}

The moment calculations in Section~\ref{sec:rmt} use three finite-field
ingredients: simultaneous constraints \(A\bx_i=\by_i\) on a symmetry-defect
fibre, Gaussian-binomial inversion, and incidence counts for the tuples
occurring in the moment expansion.  We record them in that order.

\subsubsection{Joint surjectivity on symmetry-defect fibres}

For \(\boldsymbol\xi\in\BF_2^n\), set \(h=|\boldsymbol\xi|\).  The matrix
\(\Omega_{\boldsymbol\xi}\) defined in Section~\ref{ssec:RR} has zero
diagonal, so it lies in \(\Alt_n\).  We call \(h=0\) the symmetric case and
\(h=n\) the all-ones case.

The linear map \(A\mapsto A^\top-A\) has kernel \(\Sym_n\) and image
\(\Alt_n\).  Hence, for every \(\Omega\in\Alt_n\), the fibre
\(\Sym_n(\Omega)\) is an affine translate of \(\Sym_n\) and has cardinality
\(2^{\binom{n+1}{2}}\).

To describe these fibres explicitly, let \(I_m\) and \(J_m\) denote the
\(m\times m\) identity and all-ones matrices, respectively.  Up to coordinate
permutation, the fibre attached to \(\boldsymbol\xi\) depends only on \(h\):
conjugation by a permutation matrix identifies the corresponding fibres and
preserves corank.  We may therefore order the coordinates so that
\(\xi_i=1\) exactly for \(i\le h\); then
\[
  \Sym_n(\Omega_{\boldsymbol\xi})
  =\left\{\begin{pmatrix}A_1&V\\V^{\top}&A_2\end{pmatrix}:
  A_1\in\Sym_h(J_h+I_h),\quad
  A_2\in\Sym_{n-h},\quad
  V\in\Mat_{h\times(n-h)}(\BF_2)\right\}.
\]
In these coordinates,
\begin{equation}\label{eq:sdelta}
  \rank\Omega_{\boldsymbol\xi}\in\{h-1,h\}.
\end{equation}
Thus the symmetry defect has rank $0$ in the symmetric case and rank at least
$\alpha n-1$ whenever $|\boldsymbol\xi|\ge\alpha n$.  For $n\le1$,
$J_n+I_n=0$ and the fibre is simply $\Sym_n$.

For the moment calculations, the required input from these fibres is the
joint distribution of \(A\bx_1,\dots,A\bx_k\) when
\(\bx_1,\dots,\bx_k\) are linearly independent.
\begin{lemma}\label{lem:joint-surj}
Let $\Sym_n(\Omega)$ be as in Definition~\textup{\ref{def:coset}}, with
$\Omega\in\Alt_n$, and let $1\le k\le n$.  Suppose that
$\bx_1,\dots,\bx_k\in\BF_2^{\,n}$ are linearly independent, and let
$\by_1,\dots,\by_k\in\BF_2^{\,n}$ be arbitrary. Then
\[
  \Prob_{A\sim\Unif(\Sym_n(\Omega))}\bigl(A\bx_i=\by_i\ \ \forall i\le k\bigr)
  =\ind_{\{\by_\bullet\in\Lambda\}}\cdot
   2^{-\bigl(kn-\binom k2\bigr)},
\]
where the compatibility set is
\[
  \Lambda:=\bigl\{\by_\bullet:\ \bx_i^{\top}\by_j+\bx_j^{\top}\by_i
  =\bx_i^{\top}\Omega\bx_j\ \ (i<j)\bigr\}.
\]
In particular, if \(k=1\), then
\(\Prob(A\bx=\by)=2^{-n}\) for every \(\bx\ne\bzero\) and every \(\by\).
For general \(k\),
\(\Prob(A\bx_i=\by_i\,\forall i)\le2^{-(kn-\binom k2)}\) for every \(\Omega\).
\end{lemma}

\begin{proof}
Fix $A_0\in\Sym_n(\Omega)$.  Since
\(\Sym_n(\Omega)=A_0+\Sym_n\),
$S:=A-A_0$ is uniform on $\Sym_n$, and
$A\bx_i=\by_i$ reads $S\bx_i=\bz_i$ with $\bz_i:=\by_i-A_0\bx_i$.  Extend
$\bx_1,\dots,\bx_k$ to a basis of $\BF_2^{\,n}$ and consider the $\BF_2$-linear map
\[
  \Phi:\Sym_n\to(\BF_2^{\,n})^{k},\qquad S\mapsto(S\bx_1,\dots,S\bx_k).
\]
In the chosen basis, $S\in\ker\Phi$ if and only if the first $k$ rows, and hence
by symmetry the first $k$ columns, of the matrix of the bilinear form $S$
vanish; thus $\dim\ker\Phi=\binom{n-k+1}2$ and
\[
  \dim\im\Phi=\binom{n+1}2-\binom{n-k+1}2=kn-\binom k2 .
\]
Since $S$ is symmetric, $\bx_i^{\top}S\bx_j=\bx_j^{\top}S\bx_i$, so $\im\Phi$
is contained in
\[
 \Lambda_0
 :=\{\bz_\bullet:\bx_i^{\top}\bz_j=\bx_j^{\top}\bz_i\ (i<j)\}.
\]
The \(\binom k2\) defining functionals are linearly independent because
\(\bx_1,\dots,\bx_k\) are, so
\(\dim\Lambda_0=kn-\binom k2=\dim\im\Phi\).  Hence
\(\im\Phi=\Lambda_0\).

The system is therefore solvable exactly when \(\bz_\bullet\in\Lambda_0\).
After substituting \(\bz_i=\by_i-A_0\bx_i\) and using
\(A_0^\top-A_0=\Omega\), this condition becomes
\(\by_\bullet\in\Lambda\).  When it holds, the solution set is a coset of
\(\ker\Phi\), and hence has probability
\[
 \frac{|\ker\Phi|}{|\Sym_n|}
 =2^{-\left(kn-\binom k2\right)}.
\]
\end{proof}

For a dependent tuple, apply Lemma~\ref{lem:joint-surj} to a basis of its
span.  The remaining equations are then either forced or inconsistent.
Consequently, the probability exponent depends only on
\(\dim\spn(\bx_\bullet)\), while consistency is expressed by pairwise bilinear
conditions.  This is the form needed in the moment calculation.

\subsubsection{Gaussian-binomial moments and inversion}\label{ssec:qbin}

Gaussian-binomial moments convert the preceding probabilities into information
about corank distributions.  We state the required identities for an arbitrary
prime power \(q\), although only \(q=2\) is used later.  With the notation
\(\eta_k(q)\) fixed in the introduction, write
\[
  \gaussq nk:=\prod_{i=0}^{k-1}\frac{q^{\,n-i}-1}{q^{\,k-i}-1}
  =q^{k(n-k)}\frac{\eta_n(q)}{\eta_k(q)\eta_{n-k}(q)}\qquad(0\le k\le n),
\]
with $\gaussq n0:=1$ and $\gaussq nk:=0$ for $k<0$ or $k>n$.  Thus
$\gaussq nk$ is the number of
$k$-dimensional subspaces of an $n$-dimensional vector space over $\BF_q$.  We
also write
\[
  |\GL_u(q)|=q^{\binom u2}\prod_{i=1}^{u}(q^i-1)
  =q^{u^2}\eta_u(q),
  \qquad |\GL_u|:=|\GL_u(2)|.
\]

For a nonnegative integer-valued random variable $X$, its \(u\)-th
Gaussian-binomial moment is
\[
  \Exp\gtwo Xu.
\]
When $X$ is the corank of a matrix, this is the expected number of
$u$-dimensional subspaces of its kernel.  The surjection moments used in
random-group and random-cokernel problems satisfy
\[
  \Exp\#\operatorname{Sur}(\BF_2^X,\BF_2^{\,u})
  =|\GL_u|\,\Exp\gtwo Xu.
\]
Thus the two kinds of moments differ by the factor \( |\GL_u| \); see
\cite{WoodRandomMatrices,WangWood} for systematic treatments.  We use
Gaussian-binomial moments because the \(q\)-binomial transform below acts
directly on \(\Exp\gtwo Xu\).

For completeness, we record these standard identities and their proofs, so
that the truncated-inversion argument used later is self-contained.

\begin{proposition}\label{prop:qid}
Let $q$ be a prime power.
\begin{enumerate}[\upshape(i)]
\item \emph{(Subspace-counting bounds)} For $0\le k\le n$,
  $q^{k(n-k)}\le\gaussq nk\le\eta_\infty(q)^{-1}q^{k(n-k)}$.
\item \emph{($q$-Pascal)} For $0\le k\le n$,
  $\gaussq nk=\gaussq{n-1}k+q^{\,n-k}\gaussq{n-1}{k-1}$.
\item \emph{($q$-chain)} For $0\le a\le u\le x$,
  $\gaussq xu\gaussq ua=\gaussq xa\gaussq{x-a}{u-a}$.
\item \emph{($q$-binomial inversion)} For $m\ge0$,
  $\sum_{j=0}^{m}(-1)^jq^{\binom j2}\gaussq mj=\ind_{\{m=0\}}$.
  Equivalently, the infinite lower-triangular matrices
  \begin{equation}\label{eq:Gdef}
    (G_q)_{x,u}:=\gaussq xu,
    \qquad
    (G_q^{-1})_{u,a}:=(-1)^{u-a}q^{\binom{u-a}2}\gaussq ua
    \quad(0\le a\le u\le x)
  \end{equation}
  satisfy, for all $x,a\ge0$,
  \[
    \sum_{u\ge0}(G_q)_{x,u}(G_q^{-1})_{u,a}
    =\sum_{u\ge0}(G_q^{-1})_{x,u}(G_q)_{u,a}
    =\ind_{\{x=a\}}.
  \]
\item \emph{($q$-binomial remainder)} For $N\ge0$ and $0\le m\le N$,
  \begin{equation}\label{eq:qrem}
    \sum_{j=0}^{m}(-1)^jq^{\binom j2}\gaussq Nj
    =(-1)^mq^{\binom{m+1}2}\gaussq{N-1}m .
  \end{equation}
\end{enumerate}
\end{proposition}

\begin{proof}
For (i), write
$\gaussq nk=q^{k(n-k)}\prod_{i=1}^{k}
(1-q^{-(n-k+i)})/(1-q^{-i})$.  Every factor lies between $1$ and
$(1-q^{-i})^{-1}$.

For (ii), classify the $k$-dimensional subspaces $W$ of an $n$-dimensional
space by whether $W$ is contained in a fixed hyperplane.  Formula (iii) follows
by counting flags $W\subseteq U$ with $\dim W=a$ and $\dim U=u$ in two orders.

For (iv), substitute $t=-1$ in the $q$-binomial theorem
$\prod_{i=0}^{m-1}(1+q^it)=\sum_{j=0}^{m}q^{\binom j2}\gaussq mjt^j$.  Using
(iii),
\[
  (G_qG_q^{-1})_{x,a}
  =\sum_{u=a}^{x}\gaussq xu(-1)^{u-a}q^{\binom{u-a}2}\gaussq ua
  =\gaussq xa\sum_{j=0}^{x-a}(-1)^jq^{\binom j2}\gaussq{x-a}j
  =\ind_{\{x=a\}}.
\]
The same calculation gives \(G_q^{-1}G_q=I\).

Finally, (v) follows by induction on $m$.  Expanding the last term with (ii)
and using $\binom{m+1}{2}=\binom m2+m$ cancels the two contributions containing
$\gaussq{N-1}{m-1}$ and leaves the displayed expression.
\end{proof}

\subsubsection{Subspace-incidence estimates}\label{ssec:counting}

The moment calculations reduce to incidence counts for ordered independent
tuples relative to fixed subspaces and alternating forms.  We record the three
estimates used later.

\begin{lemma}\label{lem:rank-deficit}
Let $K\subseteq\BF_q^{\,r}$ have codimension $\gamma$.  For
$0\le t\le u\le\gamma$, the number of ordered linearly independent $u$-tuples
$(\bx_1,\dots,\bx_u)$ satisfying
$\dim(\spn(\bx_1,\dots,\bx_u)\cap K)=t$ equals
\[
  \gaussq ut q^{\,ru-t\gamma}
  \prod_{i=0}^{u-t-1}(1-q^{\,i-\gamma})
  \prod_{i=0}^{t-1}(1-q^{\,i-(r-\gamma)})
  \le \gaussq ut q^{\,ru-t\gamma}.
\]
In particular, when \(t=0\), this number is
\[
  q^{\,ru}\prod_{i=0}^{u-1}(1-q^{\,i-\gamma})
  =q^{\,ru}(1+O(q^{\,u-\gamma})),
\]
and the tuples with $t\ge1$ form an $O(q^{\,u-\gamma})$ proportion of all
ordered independent $u$-tuples.
\end{lemma}

\begin{proof}
Put \(E=\BF_q^{\,u}\), \(V=\BF_q^{\,r}\), and
\(m=\dim K=r-\gamma\).  Regard the tuple as the linear map
\(f:E\to V\) defined by \(f(e_i)=\bx_i\).  The tuple is independent exactly
when \(f\) is injective, and its span meets \(K\) in dimension \(t\) exactly
when \(T:=f^{-1}(K)\) has dimension \(t\).

Choose \(T\) in \(\gaussq ut\) ways.  The restriction \(f|_T:T\to K\) and
the induced map \(\overline f:E/T\to V/K\) must be injective, giving
\(\prod_{i=0}^{t-1}(q^m-q^i)\) and
\(\prod_{i=0}^{u-t-1}(q^\gamma-q^i)\) choices, respectively.  Once they are
fixed, lift \(\overline f\) on a complement of \(T\).  Each of its \(u-t\)
basis vectors may be changed by an arbitrary element of \(K\), giving
\(q^{m(u-t)}\) maps \(f\), all injective with \(f^{-1}(K)=T\).  The required
number is therefore
\[
  \gaussq ut q^{m(u-t)}
  \prod_{i=0}^{t-1}(q^m-q^i)
  \prod_{i=0}^{u-t-1}(q^\gamma-q^i).
\]
Factoring out the powers of \(q\) and using \(m=r-\gamma\) gives the stated
formula and its upper bound.

For \(t=0\), the same formula gives the displayed asymptotic.  Finally, the
total number of ordered independent \(u\)-tuples is
\(q^{ru}\prod_{i=0}^{u-1}(1-q^{i-r})\), with the product bounded below by
\(\eta_\infty(q)\).  Using
\(\gaussq ut\le\eta_\infty(q)^{-1}q^{t(u-t)}\), the proportion with
\(t\ge1\) is at most
\[
  \eta_\infty(q)^{-2}\sum_{t=1}^{u}q^{t(u-t-\gamma)}
  =O(q^{u-\gamma}),
\]
as required.
\end{proof}

\begin{lemma}\label{lem:fallin}
Let $W\subseteq\BF_q^{\,r}$ have dimension $m$, and let $G\subseteq W$ have
dimension $g$.  Let $0\le u\le m$.  For an ordered linearly independent
$u$-tuple $(\by_1,\dots,\by_u)\in W^u$, say that $\by_j$ \emph{falls in} if
$\by_j\in G+\spn(\by_1,\dots,\by_{j-1})$.  The number of such tuples with
exactly $f$ fall-ins, where $0\le f\le u$, is
\[
  |\GL_u(q)|\,q^{(u-f)(g-f)}\gaussq gf\gaussq{m-g}{u-f}
  \le \binom uf q^{\,m(u-f)}q^{(g+u)f}.
\]
\end{lemma}

\begin{proof}
A fall-in increases
$\dim(G\cap\spn(\by_1,\dots,\by_j))$ by one, whereas a non-fall-in leaves it
unchanged.  Hence the tuples counted are exactly the ordered bases of the
$u$-dimensional subspaces $U\subseteq W$ with $\dim(U\cap G)=f$.  The standard
count of such subspaces gives the first formula.  For the upper bound, choose
the $f$ fall-in positions and bound the number of choices at a fall-in by
$q^{g+u}$ and at any other position by $q^m$.
\end{proof}

\begin{lemma}\label{lem:isotropic}
Let $B$ be an alternating bilinear form on $\BF_q^{\,r}$, let
$r_B:=\rank B$, and let $\rad(B):=\{x:B(x,y)=0\text{ for all }y\}$.  A
subspace $U$ is \emph{totally isotropic} if $B(x,y)=0$ for all $x,y\in U$.
The number $N_{\mathrm{iso}}(u)$ of $u$-dimensional totally isotropic
subspaces, for $0\le u\le r$, satisfies
\[
  N_{\mathrm{iso}}(u)q^{-(ru-\binom u2)}
  =\frac1{|\GL_u(q)|}
   \left(1+O(q^{\,2u-r_B})+O(q^{\binom{u+1}2-r_B})\right).
\]
If $B$ is nondegenerate, that is, $\rad(B)=0$, the exact formula is
\[
  N_{\mathrm{iso}}(u)q^{-(ru-\binom u2)}
  =\frac1{|\GL_u(q)|}\prod_{i=0}^{u-1}(1-q^{\,2i-r}).
\]
\end{lemma}

\begin{proof}
Suppose first that $B$ is nondegenerate.  After choosing
$\bx_1,\dots,\bx_i$, the next vector lies in their orthogonal complement but
not in their span, giving $q^{r-i}-q^i$ choices.  Dividing the product by the
number $|\GL_u(q)|$ of ordered bases gives the exact formula.

For general $B$, one has $\dim\rad(B)=r-r_B$.  Separate the isotropic subspaces
$U$ according as $U\cap\rad(B)$ is zero or nonzero.  In the first case,
projection to $\BF_q^r/\rad(B)$ is injective, and each isotropic image
has $q^{u(r-r_B)}$ lifts.  The nondegenerate formula in dimension $r_B$ contributes
\[
  \frac1{|\GL_u(q)|}\prod_{i=0}^{u-1}(1-q^{2i-r_B})
  =\frac1{|\GL_u(q)|}(1+O(q^{2u-r_B})).
\]
If $U\cap\rad(B)\ne0$, then $U$ contains a line in $\rad(B)$; the number of such
subspaces is at most
$(q^{r-r_B}-1)\gaussq{r-1}{u-1}$.  After multiplication by
\(q^{-(ru-\binom u2)}\), this contributes
$|\GL_u(q)|^{-1}O(q^{\binom{u+1}2-r_B})$.
\end{proof}

\section{From arithmetic boxes to random matrix pairs}\label{sec:box}

For \(d\) with \(r=\omega(d)\ge2\), the joint \(4\)-rank is
\(\JointFourRank_r(\bchi(d))\), where \(\bchi(d)\in\Res_r\) records the
relevant residue symbols; see \eqref{eq:arithmetic-factorization}.  The cases
\(r\le1\) have density zero.  If instead \(\bchi\) is uniform on \(\Res_r\),
then \eqref{eq:mixture} identifies the distribution of
\(\JointFourRank_r(\bchi)\) with the corank-pair distribution of the random
matrix model.  Thus
the arithmetic problem is reduced to comparing the assignments \(\bchi(d)\)
arising from primes with a uniform assignment \(\bchi\).  Since
\(\JointFourRank_r\) is invariant under relabelling and total variation
decreases under pushforward, \eqref{eq:pushforward-symmetry} shows that a bound
for the permutation-averaged assignment distribution gives the corresponding
bound for the joint \(4\)-ranks.

Smith's box method provides this assignment comparison
\cite[Theorem~5.4, Propositions~6.9--6.10 and Theorem~6.4]{Smith2017}.  His
positive covering reduces the count over \(\mathcal F_{d_0}(X)\) to
prime-coordinate boxes, in which the varying primes range independently over
disjoint intervals.  On each box satisfying the required size and spacing
conditions, the residue-symbol theorem compares the permutation-averaged
distribution of \(\bchi(d)\) with the uniform measure on \(\Res_r\).  Pushing
this estimate forward by \(\JointFourRank_r\) gives the boxwise comparison
with the random matrix model, and summing the box estimates returns the desired
comparison over \(\mathcal F_{d_0}(X)\).

\subsection{Admissible boxes}
\label{ssec:boxdecomp}

Smith's box construction is motivated by the standard model in which,
conditional on their number, the \(\log\log p\)-coordinates of the prime
factors of a typical integer resemble the order statistics of uniform points
on \([0,\log\log X]\); see \cite[\S5]{Smith2017} and
\cite{Granville2007}.  The
precise estimates needed here are summarized in
Proposition~\ref{prop:spacing}.

We now make the construction precise for the integer family.  Fix the
squarefree integer $d_0>1$ and recall $\mathcal F_{d_0}(X)$ from
\eqref{eq:F0}.  All varying prime factors belong to
\[
  \mathcal P_{d_0}=\{p\text{ prime}:p\nmid2d_0\}.
\]
Smith's covering propositions are stated for integers having no prime factor
below a fixed cutoff greater than \(3\).  To pass to the family above without changing
it, fix once and for all
\[
  D_*>\max\bigl(3,\max_{p\mid2d_0}p\bigr).
\]
Write \(d=am\), where \(a\) is supported on the finitely many primes
\(p<D_*\) in \(\mathcal P_{d_0}\).  On each fixed-\(a\) slice, \(m\) belongs,
in Smith's notation, to
\[
  S_{r-\omega(a)}(X/a,D_*).
\]
We apply the covering there and then reinsert the prime factors of \(a\) as
frozen singleton coordinates.  There are \(O_{d_0}(1)\) such slices, and the
changes \(X\mapsto X/a\) and \(r\mapsto r-\omega(a)\) are bounded in terms of
\(d_0\).  Thus all estimates below remain uniform after the slices are summed.
References to Smith's positive covering are understood in this slice-wise
sense.

For $r\ge0$, let
\[
  S_r(X;d_0)=\{d\in\mathcal F_{d_0}(X):\omega(d)=r\}.
\]
The Sathe--Selberg and Erd\H{o}s--Kac theorems place almost all the mass
in the range $r=\log\log X+O((\log\log X)^{2/3})$
\cite{SelbergSathe,ErdosKac}.  For such an $r$,
the box construction uses the following estimates on the sizes and the spacing
of the prime factors $p_1<\cdots<p_r$ of $d$.

Following Smith \cite{Smith2017}, $d$ is called
\emph{comfortably spaced above $D_1$} if
\[
  2D_1<p_i<\frac{p_{i+1}}2
  \qquad\text{whenever }i<r\text{ and }p_i>D_1,
\]
and \emph{$\eta$-regular} if
\[
  |\log\log p_i-i|
  <\eta^{1/5}\max\{i,\eta\}^{4/5}
  \qquad(1\le i<r/3).
\]
Comfortable spacing supplies disjoint intervals for the active coordinates;
regularity supplies the size estimates required by the residue-symbol
argument.  The choices of $D_1$ and $\eta$ used in the covering are fixed as
functions of $X$ in Definition~\ref{def:admissible-box}.

\begin{proposition}\label{prop:spacing}
Suppose that
\[
  |r-\log\log X|<(\log\log X)^{2/3}.
\]
The following estimates hold uniformly in this range.
\begin{enumerate}[\upshape(i)]
\item
\[
  \#S_r(X;d_0)\asymp_{d_0}
  \frac X{\log X}\frac{(\log\log X)^{r-1}}{(r-1)!},
\]
and
\[
  \sum_{\substack{r\ge0\\
    |r-\log\log X|\ge(\log\log X)^{2/3}}}
  \#S_r(X;d_0)
  \ll
  \exp\bigl(-c(\log\log X)^{1/3}\bigr)
  \#\mathcal F_{d_0}(X).
\]
\item For $D_1>3$ and $\varepsilon>0$, the proportion of elements of
\(S_r(X;d_0)\) not comfortably spaced above $D_1$ is
\[
  O_{d_0,\varepsilon}\bigl((\log D_1)^{-1}
       +(\log X)^{-1/2+\varepsilon}\bigr).
\]
\item For $\eta>1$, the proportion of elements of \(S_r(X;d_0)\) that are not
\(\eta\)-regular is
\[
  O_{d_0}\bigl(e^{-c\eta}
       +e^{-c(\log\log X)^{1/3}}\bigr).
\]
\end{enumerate}
\end{proposition}

\begin{proof}
Part~\textup{(i)} is the Sathe--Selberg estimate together with the usual
Erd\H{o}s--Kac tail.  Parts~\textup{(ii)} and~\textup{(iii)} are
\cite[Theorem~5.4(1)--(2)]{Smith2017}, applied on each fixed-\(a\) slice
described above.  The bounded shifts of \(X\) and \(r\) do not change the
displayed ranges or errors, and summing the finitely many slices preserves
uniformity.  The restriction \(p\nmid2d_0\) changes the relevant prime
harmonic sum only by a bounded term:
\[
  \sum_{\substack{p\le z\\p\in\mathcal P_{d_0}}}\frac1p
  =\log\log z+O_{d_0}(1).
\]
\end{proof}

\begin{definition}\label{def:box}
Let $D_1>3$, $r\ge1$ and $0\le m_0<r$.  A \emph{prime-coordinate box above
$D_1$} is
\[
  \bX=X_1\times\cdots\times X_r,
\]
where $X_i=\{p_i\}$ with $p_i<D_1$ for $i\le m_0$, while for $i>m_0$ the set
$X_i$ consists of primes in $\mathcal P_{d_0}$ lying in an interval
$(t_i,t_i')$, with
\[
  t_i'=\left(1+\frac1{e^{\,i-m_0}\log D_1}\right)t_i .
\]
For \(i\le m_0\), choose auxiliary endpoints
\(t_i<p_i<t_i'<D_1\) so that \(p_i\) is the only prime in
\((t_i,t_i')\) and the full endpoint sequence remains strictly increasing.
The active intervals lie above $D_1$, are pairwise disjoint and are arranged
in increasing order.  The first $m_0$ coordinates are
\emph{frozen} and the remaining coordinates are \emph{active}.  The integer
image of $\bX$ is
\[
  B_{\bX}
  =\left\{x_1\cdots x_r:x_i\in X_i,\ x_1\cdots x_r\le X\right\}
  \subseteq S_r(X;d_0).
\]
These are the coordinate widths and frozen coordinates in Smith's construction
preceding \cite[Proposition~6.9]{Smith2017}.  The coordinate primes here range
over $\mathcal P_{d_0}$, and the multiplication map is injective.  We restrict
to the boxes in the positive covering for which every tuple has product at
most $X$; for these boxes
\begin{equation}\label{eq:box-image-count}
  \#B_{\bX}=\prod_{i=1}^r\#X_i.
\end{equation}
\end{definition}

Thus $\bX$ is a space of ordered prime coordinates, whereas $B_{\bX}$ is a set
of integers; equation~\eqref{eq:box-image-count} identifies their counting
measures.  Definition~\ref{def:box} is parametric in the cutoff $D_1$; the
arithmetic covering uses the $X$-dependent choice in
Definition~\ref{def:admissible-box}.

For $d=p_1\cdots p_r$, put
\[
  \omega_{D_1}(d)=\#\{p\mid d:p<D_1\}.
\]
The construction freezes these \(\omega_{D_1}(d)\) small-prime coordinates.

We must also exclude boxes containing exceptional real characters.  Put
\[
  E=\BQ\bigl(\zeta_8,\sqrt{q_1^*},\dots,\sqrt{q_s^*}\bigr),
  \qquad d_0=2^eq_1\cdots q_s.
\]
Let \(\chi_n\) be the quadratic character associated with
\(\BQ(\sqrt n)\).  Fix the absolute constant $c_{\mathrm{zf}}$, and let
$\mathcal D_{\mathrm{exc}}$ be the set of squarefree integers $n$ for which
$L(s,\chi_n)$ has a real zero in
\[
  1-\frac{c_{\mathrm{zf}}}{\log(|n|+4)}\le s\le1.
\]
The box $\bX$ is called \emph{$E$-Siegel-free above $D_1$} if
\begin{equation}\label{eq:E-siegel-free}
  u\prod_{i\in I}x_i\notin\mathcal D_{\mathrm{exc}}
\end{equation}
for every $(x_1,\dots,x_r)\in\bX$, every $I\subseteq\{1,\dots,r\}$, every
squarefree \(u\) whose squareclass is generated by
\(-1,2,q_1,\dots,q_s\), whenever the absolute value in
\eqref{eq:E-siegel-free} exceeds $D_1$.  This is
\cite[Definition~6.2]{Smith2017} for
$P=\{-1,2,q_1,\dots,q_s\}$.

\begin{proposition}\label{prop:box-exceptions}
Let \(D_1>3\).  The following estimates hold, the second uniformly for \(r\)
in the range of Proposition~\ref{prop:spacing}.
\begin{enumerate}[\upshape(i)]
\item
\[
  \frac{\#\{d\in\mathcal F_{d_0}(X):
       \omega_{D_1}(d)>2\log\log D_1\}}
       {\#\mathcal F_{d_0}(X)}
  \ll_{d_0}(\log D_1)^{-(2\log2-1)}.
\]
\item
\[
  \#\{d\in S_r(X;d_0):d\text{ occurs in a non-$E$-Siegel-free box}\}
  \ll_{d_0}\frac{\#S_r(X;d_0)}{\log D_1}.
\]
\end{enumerate}
\end{proposition}

\begin{proof}
For the first estimate, observe that
\[
  2^{\omega_{D_1}(d)}
  =\sum_{\substack{a\mid d\\p\mid a\,\Longrightarrow\,p<D_1}}\mu^2(a).
\]
Since $\#\mathcal F_{d_0}(X)\asymp_{d_0}X$ and, uniformly for squarefree
$a$ supported on primes below $D_1$ with $(a,2d_0)=1$,
\[
  \#\{d\in\mathcal F_{d_0}(X):a\mid d\}\ll_{d_0}\frac Xa,
\]
interchanging the two sums and applying Mertens' theorem gives
\[
  \frac1{\#\mathcal F_{d_0}(X)}
  \sum_{d\in\mathcal F_{d_0}(X)}2^{\omega_{D_1}(d)}
  \ll_{d_0}
  \prod_{\substack{p<D_1\\p\nmid2d_0}}\left(1+\frac1p\right)
  \ll_{d_0}\log D_1.
\]
Markov's inequality gives part~\textup{(i)}.  For the second, apply Landau's
separation theorem and \cite[Proposition~6.10]{Smith2017} to the finitely many
squareclasses generated by \(P\).
\end{proof}

Smith's residue-symbol theorem uses auxiliary constants
\(c_1,\dots,c_{12}\).  Fix, once and for all,
\[
\begin{gathered}
 c_{\mathrm{cut}}=\frac1{200},\quad
 c_1=1000,\quad c_2=\frac1{1000},\quad c_3=2,\quad
 c_4=9,\quad c_5=4,\quad c_6=\frac1{2000},\\
 c_7=\frac1{100},\quad c_8=\frac1{1000},\quad
 c_9=c_{10}=\frac1{50},\quad
 c_{11}=\frac1{100},\quad c_{12}=\frac1{200}.
\end{gathered}
\]
These constants satisfy the hypotheses of
\cite[Theorem~6.4]{Smith2017} and the parameter inequalities used in
\cite[Corollary~6.11]{Smith2017}.  Choose
\(0<c_{\mathrm{reg}}<c_{\mathrm{cut}}\) sufficiently small for the regularity
estimates in that corollary.

\begin{definition}\label{def:admissible-box}
Fix the squarefree integer $d_0>1$ and the constants chosen above.
For all sufficiently large $X$, put
\begin{equation}\label{eq:admissible-scales}
  D_1=D_1(X)
  =\exp\left(\left(\frac12\log\log X\right)^{c_{\mathrm{cut}}}\right),
  \qquad
  \eta=\eta(X)=c_{\mathrm{reg}}\log\log\log X.
\end{equation}

Let $\bX=X_1\times\cdots\times X_r$ be a prime-coordinate box above $D_1(X)$,
and write $m_0$ for its number of frozen coordinates.  The box $\bX$ is
\emph{admissible at height $X$} if the following conditions hold.
\begin{enumerate}[\upshape(i)]
\item It belongs to Smith's positive covering and contains a comfortably
spaced, $\eta(X)$-regular integer, with
\[
  |r-\log\log X|<(\log\log X)^{2/3}.
\]
\item Its number of frozen coordinates satisfies
\[
  m_0\le2\log\log D_1.
\]
\item It is $E$-Siegel-free above $D_1$.
\end{enumerate}
\end{definition}

The scales $D_1$ and $\eta$ depend on $X$, while $r$ and $m_0$ depend on the
box.  In particular,
\[
  m_0\ll\log\log\log X,\qquad r\asymp\log\log X.
\]
Propositions~\ref{prop:spacing} and \ref{prop:box-exceptions} show that the
losses from atypical prime-factor geometry, excessive freezing and exceptional
real characters have total proportion
\[
  O_{d_0}\bigl((\log\log X)^{-c}\bigr)
\]
for some $c>0$.

Smith's positive-covering argument transfers a uniform estimate on admissible
boxes to the original family.  Smith performs the weighted integration in
\cite[Proposition~6.9]{Smith2017}; the same calculation is written out in
\cite[Theorem~5.13]{CKMP}.

\begin{proposition}\label{prop:box-decomposition}
Fix a squarefree integer $d_0>1$ and the constants
$c_{\mathrm{cut}},c_{\mathrm{reg}}$ chosen above.
For each sufficiently large $X$, let
$D_1(X)$ and $\eta(X)$ be the scales in \eqref{eq:admissible-scales}.  There is
\[
  c_{\mathrm{box}}
  =c_{\mathrm{box}}(d_0,c_{\mathrm{cut}},c_{\mathrm{reg}})>0
\]
with the following property.  Let
$\mathcal A\subseteq\mathcal F_{d_0}(X)$ and $0\le\theta\le1$.  If, for some
$\varepsilon_{\mathrm{box}}\ge0$,
\begin{equation}\label{eq:box-density-hyp}
  \left|
    \frac{\#(\mathcal A\cap B_{\bX})}{\#B_{\bX}}-\theta
  \right|
  \le\varepsilon_{\mathrm{box}}
\end{equation}
for every admissible box $\bX$ at height $X$, then
\begin{equation}\label{eq:box-density-transfer}
  \left|
    \frac{\#\mathcal A}{\#\mathcal F_{d_0}(X)}-\theta
  \right|
  \ll_{d_0,c_{\mathrm{cut}},c_{\mathrm{reg}}}
  (\log\log X)^{-c_{\mathrm{box}}}
  +\varepsilon_{\mathrm{box}}.
\end{equation}

Consequently, if $Y:\mathcal F_{d_0}(X)\to\mathcal Z$ is any statistic with
countable state space and $\rho$ is a probability measure on $\mathcal Z$, then
\begin{align}
 &\sum_{\omega\in\mathcal Z}
 \left|
   \frac{\#\{d\in\mathcal F_{d_0}(X):Y(d)=\omega\}}
        {\#\mathcal F_{d_0}(X)}
   -\rho(\omega)
 \right|\notag\\
 &\qquad\ll_{d_0,c_{\mathrm{cut}},c_{\mathrm{reg}}}
 (\log\log X)^{-c_{\mathrm{box}}}
 +\sup_{\bX}
 \sum_{\omega\in\mathcal Z}
 \left|
   \frac{\#\{d\in B_{\bX}:Y(d)=\omega\}}{\#B_{\bX}}
   -\rho(\omega)
 \right|,
\label{eq:box-transfer}
\end{align}
where the supremum is over the admissible boxes.
\end{proposition}

\begin{proof}
Let \(\mathcal G_X\) consist of the integers \(d\in\mathcal F_{d_0}(X)\)
for which \(r=\omega(d)\) lies in the range of
Proposition~\ref{prop:spacing}, \(d\) is comfortably spaced and
\(\eta(X)\)-regular, \(\omega_{D_1}(d)\le2\log\log D_1\), and \(d\) occurs
only in \(E\)-Siegel-free boxes.  Proposition~\ref{prop:spacing} and
Proposition~\ref{prop:box-exceptions} give
\[
 \#\bigl(\mathcal F_{d_0}(X)\setminus\mathcal G_X\bigr)
 \ll_{d_0,c_{\mathrm{cut}},c_{\mathrm{reg}}}
 (\log\log X)^{-c}\#\mathcal F_{d_0}(X)
\]
for some \(c>0\).  Every box in Smith's positive covering that meets
\(\mathcal G_X\) is admissible.

On each fixed-\(r\), fixed-\(a\), fixed-\(m_0\) slice, restrict Smith's
box-parameter measure to boxes meeting \(\mathcal G_X\) and rescale it by the
reciprocal membership factor used in the proof of
\cite[Proposition~6.9]{Smith2017}.  Sum these measures over the slices.
For the resulting measure \(\nu\), put
\[
  w(d):=\int\ind_{\{d\in B_{\bX}\}}\,d\nu(\bX).
\]
The calculation in that proof gives \(0\le w(d)\le1\), with \(w(d)=1\) for
\(d\in\mathcal G_X\) outside the terminal range
\[
 X(1-C/\log D_1)\le d\le X
\]
for an absolute \(C>0\).  The Sathe--Selberg estimate used there gives
\[
 \#\{d\in\mathcal F_{d_0}(X):X(1-C/\log D_1)\le d\le X\}
 \ll_{d_0}\frac{\#\mathcal F_{d_0}(X)}{\log D_1}.
\]
Consequently,
\[
 \sum_{d\in\mathcal F_{d_0}(X)}|1-w(d)|
 \ll_{d_0,c_{\mathrm{cut}},c_{\mathrm{reg}}}
 (\log\log X)^{-c}\#\mathcal F_{d_0}(X).
\]
Integrating \eqref{eq:box-density-hyp} over these admissible boxes gives
\[
 \left|
  \sum_{d\in\mathcal A}w(d)
  -\theta\sum_{d\in\mathcal F_{d_0}(X)}w(d)
 \right|
 \le\varepsilon_{\mathrm{box}}\#\mathcal F_{d_0}(X).
\]
Replacing the weights by \(1\) proves
\eqref{eq:box-density-transfer}.

For \eqref{eq:box-transfer}, let the supremum on its right-hand side be
\(\delta\).  For every \(\mathcal E\subseteq\mathcal Z\), the local error for
\(\mathcal A=Y^{-1}(\mathcal E)\) and \(\theta=\rho(\mathcal E)\) is at most
\(\delta/2\).  Apply the first assertion and then use the identity between
total variation and the supremum over events.
\end{proof}

By Proposition~\ref{prop:box-decomposition}, it remains to prove a uniform
boxwise estimate for the residue assignment.

\subsection{Equidistribution of residue assignments}\label{ssec:assign}

Fix an admissible box $\bX=X_1\times\cdots\times X_r$, and use the injective
coordinate parametrisation in \eqref{eq:box-image-count} to write
$d=p_1\cdots p_r$ with $p_i\in X_i$.  Let
\[
 \mu_{\bX}
 :=\frac1{\#B_{\bX}}\sum_{d\in B_{\bX}}\delta_{\bchi(d)}
\]
be the empirical distribution on \(\Res_r\) of the residue assignment
\eqref{eq:residue-data}, and let \(\mu_{\bX}^{\mathrm{sym}}\) be its
permutation average, defined as in Section~\ref{ssec:finite-model}.  Thus
\(\mu_{\bX}\) is the actual symbol distribution on the box, while Smith's
theorem controls its symmetrisation.

\begin{proposition}
\label{prop:box-assign}
Fix $d_0>1$.  There is $c_{\mathrm{sym}}=c_{\mathrm{sym}}(d_0)>0$ such that,
for every admissible box $\bX$ at height $X$ and all sufficiently large $X$,
\begin{equation}\label{eq:box-assign}
 \bigl\|\mu_{\bX}^{\mathrm{sym}}-\Unif(\Res_r)\bigr\|_1
 \ll_{d_0}(\log\log X)^{-c_{\mathrm{sym}}}.
\end{equation}
\end{proposition}

\begin{proof}
Smith's assignment theorem records the symbols between pairs of box
coordinates and against a fixed finite set \(P\).  In
\cite[Definition~6.1]{Smith2017}, take
\[
\begin{gathered}
 P_0=\{2,q_1,\dots,q_s\},\qquad P=\{-1\}\cup P_0,\\
 \mathsf M=\Bigl\{\{i,j\}:1\le i<j\le r\Bigr\},\qquad
 \mathsf M_P=[r]\times P.
\end{gathered}
\]
For these maximal choices, the assignment space has cardinality
\(\#\Res_r\), so the uniform main term in Smith's theorem is
\(1/\#\Res_r\).

Smith records \(\asm{p_i}{p_j}\) and \(\asm{q_k}{p_i}\), while
\eqref{eq:residue-data} uses prime discriminants.  The two assignments are
exchanged by the involution \(\Theta:\Res_r\to\Res_r\) given by
\[
 \Theta(\bchi)_{ij}=\chi_{ij}+u_{-1,i}u_{-1,j}\quad(i<j),
 \qquad
 \Theta(\bu_{\Delta_k})_i
 =u_{\Delta_k,i}+\asm{-1}{q_k}u_{-1,i}\quad(1\le k\le s),
\]
with \(\bu_{-1}\) and \(\bu_2\) fixed.  Quadratic reciprocity shows that
\(\Theta\) converts the raw symbols to those in \eqref{eq:residue-data}.
It is action-equivariant and preserves \(\Unif(\Res_r)\), so it suffices to
prove \eqref{eq:box-assign} for Smith's raw assignment.

The field-theoretic independence hypothesis also holds.  Every prime ramified
in \(E\) divides \(2d_0\), whereas every nonempty product of the \(p_i^*\) is
ramified at some \(p_i\nmid2d_0\).  Thus the varying squareclasses are
independent over \(E\), and \eqref{eq:E-siegel-free} is precisely Smith's
Siegel-less condition for the set \(P\) above.

We apply \cite[Theorem~6.4]{Smith2017} with the constants fixed above and
the parameter choice of \cite[Corollary~6.11]{Smith2017}.  Namely, take
\[
 k_0=m_0,\qquad k_2=r,\qquad t=D_1,\qquad t'=t'_{m_0+1},
\]
so Smith's permutation group is \(\mathfrak S_r\) and his normalization is
\(r!\).  Let \(k_1\) be the least index at least \(k_0\) for which
\[
 t'_{k_1+1}>\exp(D_1^{c_6}).
\]
The required scale comparison is obtained by putting
\(L=\log\log X\) and \(R=\log L\).  Admissibility gives
\[
 \log\log D_1=c_{\mathrm{cut}}R+O(1),\qquad
 r=L+O(L^{2/3}),\qquad
 m_0\le 2c_{\mathrm{cut}}R+O(1).
\]
Thus \(c_{\mathrm{cut}}\) plays the role of the cutoff constant in the proof
of that corollary.  The numerical inequalities among the constants fixed
above, together with a sufficiently small choice of \(c_{\mathrm{reg}}\),
give conditions \textup{(1)--(6)} of the theorem by the verification in that
corollary.
The fixed set \(P\) is absorbed in the constant terms,
\cite[Proposition~6.7]{Smith2017} handles the frozen singleton coordinates,
and the bound
\(m_0\le2\log\log D_1\) is absorbed by the chosen margins.  Moreover,
regularity at the first coordinate and the choice of \(c_{\mathrm{reg}}\)
give \(m_0\ge1\) for all sufficiently large \(X\); hence the endpoint convention
\(t'_{m_0}<D_1<t'_{m_0+1}\) covers the boundary case \(k_1=k_0\).
The remaining Siegel-less hypothesis is condition~\textup{(iii)} of
Definition~\ref{def:admissible-box}.

With $k_2=r$, the conclusion of the theorem reads
\[
 \sum_{\bchi_0\in\Res_r}
 \left|
   \frac1{r!}\sum_{\pi\in\mathfrak S_r}
   \frac{\#\{d\in B_{\bX}:\pi\cdot\bchi(d)=\bchi_0\}}{\#B_{\bX}}
   -\frac1{\#\Res_r}
 \right|
 \le r^{-c_{12}}+(t')^{-c_8},
\]
where we used \eqref{eq:box-image-count} to identify $\#B_{\bX}$ with the
number of coordinate tuples.  Finally $t'=t'_{m_0+1}>D_1(X)$ and
$r\asymp\log\log X$, so the right-hand side is
$O_{d_0}\bigl((\log\log X)^{-c_{\mathrm{sym}}}\bigr)$ for any
$c_{\mathrm{sym}}<c_{12}$.
\end{proof}

\begin{corollary}
\label{cor:box-corank}
For all sufficiently large \(X\), let \(\bX\) be an admissible box at height
\(X\).  Then
\begin{align*}
 &\sum_{\mathbf j\in\BZ_{\ge0}^2}
 \left|
   \frac{\#\{d\in B_{\bX}:
     (r_4(-d),r_4(-d_0d))=\mathbf j\}}{\#B_{\bX}}
   -\bigl((\JointFourRank_r)_*
     \Unif(\Res_r)\bigr)(\mathbf j)
 \right|\\
 &\hspace{45mm}
 \ll_{d_0}(\log\log X)^{-c_{\mathrm{sym}}}.
\end{align*}
\end{corollary}

\begin{proof}
By \eqref{eq:arithmetic-factorization}, the empirical distribution in the
statement is \((\JointFourRank_r)_*\mu_{\bX}\).  Therefore
\[
\begin{aligned}
 &\bigl\|(\JointFourRank_r)_*\mu_{\bX}
   -(\JointFourRank_r)_*\Unif(\Res_r)\bigr\|_1\\
 &\quad=
   \bigl\|(\JointFourRank_r)_*\mu_{\bX}^{\mathrm{sym}}
   -(\JointFourRank_r)_*\Unif(\Res_r)\bigr\|_1\\
 &\quad\le
   \bigl\|\mu_{\bX}^{\mathrm{sym}}-\Unif(\Res_r)\bigr\|_1.
\end{aligned}
\]
The equality is \eqref{eq:pushforward-symmetry}; the inequality is contraction
under pushforward.  Proposition~\ref{prop:box-assign} proves the result.
\end{proof}

\subsection{Family transfer}\label{ssec:tofamily}

Combining the boxwise estimate with the positive covering gives the required
transfer from the arithmetic family to the matrix model.

\begin{theorem}
\label{thm:box0}
Fix $d_0>1$ and $0<\alpha<1/2$.  There exists
$c_{\mathrm{tr}}=c_{\mathrm{tr}}(d_0,\alpha)>0$ such that, with
$\mu_X^{(d_0)}$ as in \eqref{eq:empirical-distribution}, for every probability measure
$\rho$ on $\BZ_{\ge0}^2$ and all sufficiently large $X$,
\begin{equation}\label{eq:family-transfer-bd}
 \bigl\|\mu_X^{(d_0)}-\rho\bigr\|_1
 \ll_{d_0,\alpha}(\log\log X)^{-c_{\mathrm{tr}}}
 {}+\sup_{\substack{|r-\log\log X|<(\log\log X)^{2/3}\\
                  |\bu_{-1}|\ge\alpha r}}
 \bigl\|\Law_{\MatData_r(\bu_{-1})}(\mathbf Z)-\rho\bigr\|_1 ,
\end{equation}
where the supremum is over $r$ in the displayed range and
$\bu_{-1}\in\BF_2^r$ of weight at least $\alpha r$.
\end{theorem}

\begin{proof}
Put $Y(d)=\bigl(r_4(-d),r_4(-d_0d)\bigr)$, and let $\nu_{\bX}$ be the distribution of
$Y(d)$ for uniform $d\in B_{\bX}$.  By Corollary~\ref{cor:box-corank},
\[
 \bigl\|\nu_{\bX}-(\JointFourRank_r)_*\Unif(\Res_r)\bigr\|_1
 \ll_{d_0}(\log\log X)^{-c_{\mathrm{sym}}}
\]
uniformly over the admissible boxes at height $X$; here $r$ is the number of
coordinates of $\bX$.  Apply \eqref{eq:box-transfer} to $Y$ and to the given
$\rho$, and bound each box term by the triangle inequality through
 $(\JointFourRank_r)_*\Unif(\Res_r)$.  Since the admissible boxes at height
$X$ have $|r-\log\log X|<(\log\log X)^{2/3}$, this yields the intermediate
estimate
\[
 \bigl\|\mu_X^{(d_0)}-\rho\bigr\|_1
 \ll_{d_0}(\log\log X)^{-c_{\mathrm{tr}}}
 {}+\sup_{|r-\log\log X|<(\log\log X)^{2/3}}
 \bigl\|(\JointFourRank_r)_*\Unif(\Res_r)-\rho\bigr\|_1 .
\]

Finally, \eqref{eq:mixture} writes each finite pushforward as the uniform
mixture, over \(\bu_{-1}\in\BF_2^r\), of
\[
 \Law_{\MatData_r(\bu_{-1})}(\mathbf Z).
\]
Convexity of the \(\ell^1\) norm bounds the last norm by
\[
 \max_{|\bu_{-1}|\ge\alpha r}
 \bigl\|\Law_{\MatData_r(\bu_{-1})}(\mathbf Z)-\rho\bigr\|_1
 +2\cdot2^{-r}\#\{\bu_{-1}:|\bu_{-1}|<\alpha r\},
\]
since the $\ell^1$ distance between two probability measures is at most $2$.
The second term is at most $2e^{-2(1/2-\alpha)^2r}$ by
\eqref{eq:sign-tail}, hence $O_\alpha((\log X)^{-c})$ for some
$c=c(\alpha)>0$ in the displayed range of $r$.  Shrinking
$c_{\mathrm{tr}}$ proves \eqref{eq:family-transfer-bd}.
\end{proof}

\begin{remark}
The same box construction can be adapted when every prime factor is required
to lie in a fixed Frobenius class of a multiquadratic field $F$.  The typical
value of \(r\) is then $[F:\BQ]^{-1}\log\log X$, and
$\Unif(\Res_r)$ is replaced by the conditional uniform measure on the affine coset
determined by the characters common to $E$ and $F$.  This conditioning can
change the induced distribution of the matrix pair: for $F=\BQ(i)$ it
prescribes the sign vector $\bu_{-1}$,
and for $F=\BQ(\zeta_8)$ it also fixes the symbols at $2$.  In particular,
Theorem~\ref{thm:box0} concerns the full space $\Res_r$.  Under a
fixed-sign condition its target must be replaced by the corresponding
conditional distribution of the residue assignments; its matrix pushforward
requires a separate analysis.
\end{remark}

\section{Joint coranks of bordered matrices}\label{sec:rmt}

The family transfer \eqref{eq:family-transfer-bd} reduces the problem to a
uniform estimate for the corank pair on the uniform probability space
 \(\MatData_r(\bu_{-1})\) introduced in Section~\ref{ssec:finite-model}.  Fix \(d_0\),
 \(r\) and \(\bu_{-1}\);
the \(2\)-adic stratum is determined by \(d_0\) and \(|\bu_{-1}|\).  We prove an
\(\ell^1\) error bound \(C\eta^r\), with \(\eta\in(0,1)\), uniformly for
 \(|\bu_{-1}|\ge\alpha r\).  The complementary sign vectors have exponentially
small mass by \eqref{eq:sign-tail}.

This estimate is assembled in three steps.  We
first prove a quantitative truncated inversion of the one-dimensional
Gaussian-binomial moment transform and apply it to a single core block.  We
then pass to the two-dimensional transform, whose inverse is the tensor power
of the one-dimensional inverse but whose truncation residual needs its own
bivariate estimate, and extract from it a criterion that converts conditional
mixed moments on a good event into a total-variation bound.  We finally verify
that criterion for the bordered pair by estimating its conditional mixed
moments.  The last step does not follow from the first: the two fixed border
widths need not agree, and the diagonal perturbation \(\bu_{d_0}\) and the
upper border \(\mathsf U\) are determined by the lower-border data and the sign
vector rather than sampled independently.

\subsection{One-dimensional truncated inversion for corank distributions}\label{ssec:framework}

The normalization here differs from the ordinary size moments occurring, for
example, in the work of Fouvry--Kl\"uners and Smith
\cite{FK4rank,SmithTwistsII}.  As recalled in Section~\ref{ssec:qbin}, the
Gaussian-binomial moment of a corank variable \(X\) is a normalized surjection
moment:
\[
 \Exp\#\operatorname{Sur}(\BF_2^X,\BF_2^{\,u})
   =|\GL_u|\,\Exp\gtwo Xu.
\]
Thus it is the surjection-moment normalization used in random finite-group and
cokernel problems; see \cite{WoodRandomMatrices,WangWood}.  Equivalently,
\(\gtwo Xu\) is a polynomial of degree \(u\) in \(2^X\), so these moments form
a triangular change of basis from ordinary moments of \(2^X\).  This basis
counts \(u\)-dimensional kernel subspaces, and
Lemma~\ref{lem:joint-surj} makes it accessible one order at a time.
Fixed-order convergence of these moments, together with a uniqueness argument,
is sufficient to identify a limiting distribution.

Fixed-order convergence and uniqueness do not by themselves retain a rate.
To obtain a total-variation bound \(C\eta^n\) with \(\eta\in(0,1)\), we retain
moments only through an order \(D\) that grows with the matrix size, apply the
bounded inverse to this truncated moment sequence, and estimate the residual
explicitly.  This is the quantitative truncated inversion used below.

We first develop it in one variable.  The tensor-product inverse and the
multivariable residual are treated in Section~\ref{ssec:criterion}.

\subsubsection{Gaussian-binomial inversion and truncation}\label{sssec:1d-inversion}

At $q=2$, set
\[
  \gamma_{u,a}:=(G_2^{-1})_{u,a}
  =(-1)^{u-a}2^{\binom{u-a}2}\gtwo ua\qquad(0\le a\le u).
\]
For a finite signed measure $\sigma$ on $\BZ_{\ge0}$ and every $u$ for which
$m_u(|\sigma|)<\infty$, set
\begin{equation}\label{eq:moment-transform}
  m_u(\sigma):=\sum_{x\ge u}\gtwo xu\,\sigma(x).
\end{equation}
When these absolute moments are finite for every $u$, write
$(\mathcal G\sigma)_u=m_u(\sigma)$ and
$\momentseq{\sigma}=\mathcal G\sigma$ for the full Gaussian-binomial moment
transform.  For a finitely supported measure the transform is triangular, and
finite $q$-binomial inversion reads
\[
  \sigma(a)=\sum_{u\ge a}\gamma_{u,a}m_u(\sigma).
\]

Set
\begin{equation}\label{eq:weight}
  \theta_u:=2^{-u(u+1)/2},
\end{equation}
and define
\[
  \|z\|_{\infty,\theta}
   :=\sup_{u\ge0}\frac{|z_{u}|}{\theta_{u}},
  \qquad
  \momentSpace{1}
   :=\{z:\|z\|_{\infty,\theta}<\infty\}.
\]
For \(z\in\momentSpace{1}\), define
\[
 \bigl(\mathcal G^{-1}z\bigr)(a)
 :=\sum_{u\ge a}\gamma_{u,a}z_u.
\]
Theorem~\ref{thm:embed-1d} shows that this series defines an
\(\ell^1\)-sequence.  If
\(\momentseq{|\sigma|}\in\momentSpace{1}\), the residual estimate below shows
that \(\mathcal G^{-1}\mathcal G\sigma=\sigma\); no inversion is asserted here
under the weaker assumption that all moments are merely finite.

Thus \(\mathcal G\) maps
\[
 \{\sigma\in\ell^1(\BZ_{\ge0}):
       \momentseq{|\sigma|}\in\momentSpace{1}\}
\]
to \(\momentSpace{1}\), with formal inverse bounded from
\(\momentSpace{1}\) to \(\ell^1(\BZ_{\ge0})\).

\begin{theorem}
\label{thm:embed-1d}
The operator $\mathcal G^{-1}$ maps $\momentSpace{1}$ continuously into
$\ell^1(\BZ_{\ge0})$, and
\[
  \bigl\|\mathcal G^{-1}z\bigr\|_1
  \le \Cinv\|z\|_{\infty,\theta},
  \qquad
  \Cinv:=\sum_{u\ge0}\theta_u\sum_{a=0}^{u}|\gamma_{u,a}|<\infty.
\]
\end{theorem}

\begin{proof}
Put $t=u-a$.  Proposition~\ref{prop:qid}(i) gives
\[
  \theta_u|\gamma_{u,a}|
  \le\eta_\infty(2)^{-1}
     2^{-u(u+1)/2+\binom t2+at}
  =\eta_\infty(2)^{-1}2^{-a(a+1)/2-t}
  \le\eta_\infty(2)^{-1}2^{-u}.
\]
Thus $\Cinv=\sum_u\theta_u\sum_{a\le u}|\gamma_{u,a}|<\infty$, and
\[
  \bigl\|\mathcal G^{-1}z\bigr\|_1
  \le\sum_{a\ge0}\sum_{u\ge a}|\gamma_{u,a}|\,|z_u|
  \le\|z\|_{\infty,\theta}\sum_{u\ge0}\theta_u\sum_{a\le u}|\gamma_{u,a}|
\]
by Tonelli's theorem.
\end{proof}

Let $P_D$ denote truncation of a sequence: $(P_Dz)_{u}=z_{u}$ when $u\le D$,
and zero otherwise.  Suppose first that
$m_u(|\sigma|)<\infty$ for every $u$, so that the whole sequence
$\momentseq{\sigma}$ is defined, and set
\begin{equation}\label{eq:trunc-op}
  Q_D:=\mathcal G^{-1}P_D\mathcal G,
\end{equation}
and define the residual
\begin{equation}\label{eq:residual-definition}
  r_D(\sigma)
  :=(\operatorname{id}-Q_D)\sigma
  =\sigma-\mathcal G^{-1}P_D\momentseq{\sigma}.
\end{equation}
\(Q_D\) takes a finite signed measure to its moment sequence, retains the orders
\(0,\dots,D\), and applies the bounded inverse.  Since the composite
$\mathcal G^{-1}P_D\mathcal G$ reads no higher moment, the same formulas define
$Q_D\sigma$ and $r_D(\sigma)$ whenever
$m_u(|\sigma|)<\infty$ only for $0\le u\le D$, even if the upper-right corner
of the full transform is unavailable.

Truncation need not return $\sigma$: the decomposition
\[
  \sigma=Q_D\sigma+r_D(\sigma)
\]
isolates in $r_D(\sigma)$ exactly the part of $\sigma$ that the retained
moments do not control.  The two terms are estimated separately:
Theorem~\ref{thm:embed-1d} bounds $\|Q_D\sigma\|_1$ in terms of the retained
moments, and Lemma~\ref{lem:residual} bounds $\|r_D(\sigma)\|_1$ by a single
higher moment.  These are the two terms of the truncated inversion inequality
below.

\begin{lemma}\label{lem:residual}
There is an absolute constant $C$ such that, for every finite signed measure $\sigma$ with $m_{D+1}(|\sigma|)<\infty$,
\[
  \|r_D(\sigma)\|_1
  \le C(D+1)2^{D(D+1)/2}m_{D+1}(|\sigma|).
\]
\end{lemma}

\begin{proof}
Apply the truncated inverse first to a point mass at $x$.  Finite
$q$-binomial inversion gives
\[
  \ind_{\{a=x\}}
  =\sum_{u=a}^{x}\gamma_{u,a}\gtwo xu.
\]
Consequently the residual has the pointwise kernel representation
\[
  r_D(\sigma)(a)=\sum_x\sigma(x)K_D(x,a),
  \qquad
  K_D(x,a):=\sum_{u=\max(a,D+1)}^x\gamma_{u,a}\gtwo xu.
\]
Every inner sum is finite, so this identity uses only the assumed absolute
$(D+1)$-st moment.
For $x>D$ and $a\le D$, the $q$-chain identity and~\eqref{eq:qrem} give
\[
  |K_D(x,a)|=\gtwo xa\,2^{\binom{D-a+1}2}\gtwo{x-a-1}{D-a}.
\]
Summing over $a\le D$ and using Proposition~\ref{prop:qid}(i),(iii) gives
\[
  \sum_{a\le D}|K_D(x,a)|
  \le C(D+1)2^{D(D+1)/2}\gtwo x{D+1}.
\]
For $a>D$, finite inversion gives $K_D(x,a)=\ind_{\{a=x\}}$.
Consequently
\[
  \sum_{a\ge0}|K_D(x,a)|
  \le 1+C(D+1)2^{D(D+1)/2}\gtwo x{D+1},
\]
and the first term is absorbed by the second when $x>D$.  The result follows
after summing over $x$; for $x\le D$ the residual is identically zero.
\end{proof}

If \(\momentseq{|\sigma|}\in\momentSpace{1}\), then
\[
 \|r_D(\sigma)\|_1
 \ll (D+1)2^{D(D+1)/2}\theta_{D+1}
 =(D+1)2^{-(D+1)}\longrightarrow0.
\]
At the same time, the tail estimate in the proof of
Theorem~\ref{thm:embed-1d} shows that
\(\mathcal G^{-1}P_D\mathcal G\sigma\) converges in \(\ell^1\) to
\(\mathcal G^{-1}\mathcal G\sigma\).  Hence
\(\mathcal G^{-1}\mathcal G\sigma=\sigma\) on the domain displayed above.

\begin{theorem}
\label{thm:trunc-inv-1d}
Let $\sigma$ be a finite signed measure on $\BZ_{\ge0}$ with
$m_{D+1}(|\sigma|)<\infty$.  Then
\begin{equation}\label{eq:trunc-inv-1d}
  \|\sigma\|_1
  \le \Cinv\|P_D\momentseq{\sigma}\|_{\infty,\theta}
   +C(D+1)2^{D(D+1)/2}m_{D+1}(|\sigma|),
\end{equation}
where $C$ is the constant of Lemma~\ref{lem:residual}.
\end{theorem}

\begin{proof}
By definition,
$\sigma=\mathcal G^{-1}P_D\momentseq{\sigma}+r_D(\sigma)$.
Theorem~\ref{thm:embed-1d} bounds the first term, and
Lemma~\ref{lem:residual} bounds the second.
\end{proof}

The first term on the right of~\eqref{eq:trunc-inv-1d} is
$\Cinv\max_{0\le u\le D}|m_u(\sigma)|/\theta_u$ and uses only the moments
through order $D$; the second controls everything that was discarded.  The two
terms move in opposite directions as $D$ grows.  Balancing them gives the
following criterion.

\begin{corollary}
\label{thm:rate-1d}
Let $\sigma_n$ be finite signed measures on $\BZ_{\ge0}$.  Suppose that
there are constants $C<\infty$, $C_1>1$, $\qerr\in(0,1)$ and $\beta_0>0$
such that
\begin{enumerate}[\upshape(i)]
\item for every $n$ and every $u\le\beta_0n$,
  \[
    |m_u(\sigma_n)|\le CC_1^{u}\theta_u\qerr^n;
  \]
\item the absolute moment sequences are uniformly bounded at the reference
  scale:
  \[
    K:=\sup_n\|\momentseq{|\sigma_n|}\|_{\infty,\theta}<\infty.
  \]
\end{enumerate}
Then there are $C'<\infty$ and $\eta\in(0,1)$ such that
$\|\sigma_n\|_1\le C'\eta^n$.  More precisely, for every $D\le\beta_0n$,
\[
  \|\sigma_n\|_1
  \le C\Cinv C_1^D\qerr^n+C''(D+1)2^{-(D+1)}K.
\]
Choosing $D=\lfloor\beta n\rfloor$ with $0<\beta<\beta_0$ sufficiently small gives
$\eta=\max\{2^{-\beta/2},C_1^\beta \qerr\}<1$.
At $u=0$, part~\textup{(i)} also controls the difference of
the total masses, so no separate normalization is required.
\end{corollary}

\begin{proof}
Apply Theorem~\ref{thm:trunc-inv-1d}.  Part~\textup{(i)} bounds the truncated
term by $C\Cinv C_1^D\qerr^n$.  By part~\textup{(ii)},
$m_{D+1}(|\sigma_n|)\le K\theta_{D+1}$, and
$2^{D(D+1)/2}\theta_{D+1}=2^{-(D+1)}$, so the residual is at most
$C(D+1)2^{-(D+1)}K$.  The stated choice of $D$ balances the two bounds; its
polynomial factor is absorbed by decreasing $\beta$.
\end{proof}

\subsubsection{Application to a single core block}\label{ssec:1d}

For one core block, the zero- and high-rank-defect limits are \(\pisym\) and
\(\piCL\), respectively.

\begin{lemma}\label{lem:1d-moment}
Let $A\sim\Unif(\Sym_n(\Omega))$, put $X:=\corank A$ and
$\rho_n:=\Law(X)$.  Then
\begin{equation}\label{eq:1d-uniform}
  m_u(\rho_n)=\Exp\gtwo Xu\le\eta_\infty(2)^{-1}\theta_u
  \qquad(n,u\ge0).
\end{equation}
Define the target moment
\[
  m_u(\Omega)
  :=\begin{cases}
    \dfrac{\theta_u}{\eta_u(2)},&\Omega=0,\\[2ex]
    \dfrac{2^{-u^2}}{\eta_u(2)}=\dfrac1{|\GL_u|},&\Omega\ne0.
  \end{cases}
\]
For the zero defect this is the limit as $n\to\infty$; for a sequence of
nonzero defects it is the limit whenever $\rank\Omega\to\infty$.  More
precisely,
\[
  \frac{m_u(\rho_n)}{m_u(\Omega)}
  =\begin{cases}
    \displaystyle\prod_{i=0}^{u-1}(1-2^{i-n})
      =1+O(2^{u-n}),&\Omega=0,\\[2ex]
    1+O(2^{2u-\rank\Omega})
      +O\bigl(2^{\binom{u+1}2-\rank\Omega}\bigr),&\Omega\ne0.
  \end{cases}
\]
Thus $m_u(\Omega)\asymp\theta_u$ for $\Omega=0$, whereas
$m_u(\Omega)\asymp2^{-u^2}$ for $\Omega\ne0$.
\end{lemma}

\begin{proof}
The first-moment identity is
\[
  m_u(\rho_n)=\sum_{\dim W=u}\Prob(W\subseteq\ker A).
\]
For a basis $x_1,\dots,x_u$ of $W$, Lemma~\ref{lem:joint-surj} gives
$\Prob(W\subseteq\ker A)=2^{-(nu-\binom u2)}$ if $W$ is totally isotropic for
the alternating form $B_\Omega(x,y)=x^\top\Omega y$, and zero otherwise.
Consequently,
\begin{equation}\label{eq:1d-iso}
  m_u(\rho_n)=N_{\mathrm{iso}}(u)2^{-(nu-\binom u2)}.
\end{equation}
If $\Omega=0$, every subspace is isotropic, and
\begin{equation}\label{eq:1d-sym-moment}
  m_u(\rho_n)
  =\gtwo nu2^{-(nu-\binom u2)}
  =\frac{\theta_u}{\eta_u(2)}\frac{\eta_n(2)}{\eta_{n-u}(2)}.
\end{equation}
For $\Omega\ne0$, Lemma~\ref{lem:isotropic} gives
\begin{equation}\label{eq:1d-excess-moment}
  m_u(\rho_n)=m_u(\Omega)
  \left(1+O(2^{2u-\rank\Omega})
  +O\bigl(2^{\binom{u+1}2-\rank\Omega}\bigr)\right),
  \qquad m_u(\Omega)=\frac1{|\GL_u|}.
\end{equation}
The uniform bound follows from
$N_{\mathrm{iso}}(u)\le\gtwo nu$ and Proposition~\ref{prop:qid}(i).
\end{proof}

The formula~\eqref{eq:pisym} is normalized so that its
Gaussian-binomial moments are exactly the zero-defect target moments.
Indeed, writing $x=u+j$ and using the product formula for
$\gtwo{x}{u}$ gives
\[
 \sum_{x\ge u}\gtwo{x}{u}\pisym(x)
 =\frac{\eta_\infty(2)}{\eta_\infty(4)}
   \frac{\theta_u}{\eta_u(2)}
   \sum_{j\ge0}\frac{\theta_j}{\eta_j(2)}
 =\frac{\theta_u}{\eta_u(2)}.
\]
Here the last equality is Euler's identity
\[
 \sum_{j\ge0}\frac{2^{-j(j+1)/2}}{\eta_j(2)}
 =\frac{\eta_\infty(4)}{\eta_\infty(2)}.
\]
Thus bounded Gaussian-binomial inversion identifies the inverse of
$(m_u(0))_{u\ge0}$ with $\pisym$.

\begin{theorem}\label{thm:1d}
For each $n$, let $\Omega_n\in\Alt_n$, take
$A_n\sim\Unif(\Sym_n(\Omega_n))$, and put $X_n:=\corank A_n$.
\begin{enumerate}[\upshape(1)]
\item If $\Omega_n=0$ for every $n$, then
  \[
    \|\Law(X_n)-\pisym\|_1\le C\eta^n
  \]
  for some $\eta\in(0,1)$.
\item If $\Omega_n\ne0$, then, for every $\epsilon\in(0,1)$,
  \[
    \|\Law(X_n)-\piCL\|_1
      \le C_\epsilon 2^{-(1-\epsilon)\rank\Omega_n}.
  \]
  In particular $\Law(X_n)\to\piCL$ whenever
  $\rank\Omega_n\to\infty$; if
  $\rank\Omega_n\ge\alpha n$, the displayed error is at most
  \(C_\epsilon(2^{-(1-\epsilon)\alpha})^n\).
\end{enumerate}
\end{theorem}

\begin{proof}
Write $\rho_n:=\Law(X_n)$ and set
\[
  \mu_{\Omega_n}:=\mathcal G^{-1}(m_u(\Omega_n))_{u\ge0},
  \qquad \sigma_n:=\rho_n-\mu_{\Omega_n}.
\]
Equation~\eqref{eq:1d-uniform} and the corresponding target-moment bound
show that all required absolute moments are finite.  Converting the ratio
estimates of Lemma~\ref{lem:1d-moment} to the \(\theta\)-weighted norm gives
\[
  \frac{|m_u(\sigma_n)|}{\theta_u}
  \ll
  \begin{cases}
    2^u2^{-n},&\Omega_n=0,\\
    2^u2^{-\rank\Omega_n},&\Omega_n\ne0.
  \end{cases}
\]
In the second line the factor
$m_u(\Omega_n)/\theta_u=2^{-\binom u2}/\eta_u(2)$ cancels the quadratic
growth in the relative isotropic-subspace error.

If $\Omega_n=0$, then
$m_u(\Omega_n)=\theta_u/\eta_u(2)$ and
$\mu_{\Omega_n}=\pisym$.  Corollary~\ref{thm:rate-1d}, with
$D=\lfloor n/2\rfloor$, gives
$\|\sigma_n\|_1\ll n2^{-n/2}$.

Suppose now that $s_n:=\rank\Omega_n>0$.  Since
$m_u(\Omega_n)=1/|\GL_u|$, Gaussian-binomial inversion and Euler's product
identity give
\[
\begin{aligned}
  \mu_{\Omega_n}(a)
  &=\sum_{u\ge a}\gamma_{u,a}m_u(\Omega_n)\\
  &=\frac1{|\GL_a|}
    \sum_{j\ge0}\frac{(1/2)^{\binom j2}}{\eta_j(2)}
       \bigl(-2^{-(a+1)}\bigr)^j
   =\frac{2^{-a^2}\eta_\infty(2)}{\eta_a(2)^2}.
\end{aligned}
\]
Thus $\mu_{\Omega_n}=\piCL$.  Suppose that
\[
  \binom{D+2}{2}\le s_n.
\]
Then Lemma~\ref{lem:1d-moment} gives
\[
  m_{D+1}(|\sigma_n|)\ll2^{-(D+1)^2}.
\]
Theorem~\ref{thm:trunc-inv-1d} therefore bounds the truncated term by
$O(2^D2^{-s_n})$ and the residual by
$O((D+1)2^{-(D+1)(D+2)/2})$.  Taking
\[
  D=\left\lfloor\sqrt{2(1-\epsilon)s_n}\right\rfloor
\]
and absorbing the finitely many small values of $s_n$ into the constant gives
$\|\sigma_n\|_1\ll_\epsilon2^{-(1-\epsilon)s_n}$.
\end{proof}

\begin{remark}[Comparison with the one-matrix theory]
\label{rem:XZ-boundary}
For a nonzero defect, the loss of $\epsilon$ in Theorem~\ref{thm:1d} comes
from the error term
\[
  O\bigl(2^{\binom{u+1}{2}-\rank\Omega}\bigr)
\]
in Lemma~\ref{lem:isotropic}.  This term makes the weighted moment error grow
with $u$, so the bounded inverse must be combined with truncation.  The
spectral method of \cite{XZ} gives the sharper exponent
$2^{-\rank\Omega}$ for the one-matrix ensembles considered there; the estimate
above is sufficient for the joint theorem.

If $\rank\Omega$ remains bounded, the second assertion is vacuous and the limiting
distribution need not be $\piCL$.  For the arithmetic defect
\(\Omega=\bu_{-1}\bu_{-1}^{\top}+\diag(\bu_{-1})\),
$\rank\Omega=|\bu_{-1}|+O(1)$.  Thus only the near-symmetric range
$|\bu_{-1}|=o(r)$ lies outside the range in which the theorem gives uniform
convergence to $\piCL$.

\end{remark}

\subsection{Multivariable truncated inversion and joint convergence}\label{ssec:criterion}

The full multivariable moment transform and its inverse are tensor powers of
their one-dimensional counterparts.  Truncated inversion additionally
requires a residual estimate obtained by expanding over nonempty sets of
coordinates.  The mixed-moment estimates are proved separately for the
coupled matrices.

\subsubsection{Tensor-product inversion}

Fix $m\ge1$.  For a finite signed measure $\sigma$ on $\BZ_{\ge0}^m$ and a
multi-index $\mathbf k=(k_1,\dots,k_m)$ for which
$m_{\mathbf k}(|\sigma|)<\infty$, set
\begin{equation}\label{eq:moment-transform-multi}
  m_{\mathbf k}(\sigma)
  :=\sum_{\mathbf x\ge\mathbf k}
    \prod_{i=1}^{m}\gtwo{x_i}{k_i}\,\sigma(\mathbf x).
\end{equation}
When these absolute moments are finite for every $\mathbf k$, write
$(\mathcal G\sigma)_{\mathbf k}=m_{\mathbf k}(\sigma)$ and
$\momentseq{\sigma}=\mathcal G\sigma$.
This is the $m$-fold tensor power of the one-dimensional
transform~\eqref{eq:moment-transform}; its inverse is accordingly the tensor
power $(\mathcal G^{-1})^{\otimes m}$ of the one-dimensional inverse,
\[
  \bigl((\mathcal G^{-1})^{\otimes m}z\bigr)(\mathbf a)
  :=\sum_{\mathbf k\ge\mathbf a}
      \prod_{i=1}^{m}\gamma_{k_i,a_i}\,z_{\mathbf k}.
\]
Extend the weight~\eqref{eq:weight} multiplicatively,
\[
  \theta_{\mathbf k}:=\prod_{i=1}^{m}\theta_{k_i},
\]
and define
\[
  \|z\|_{\infty,\theta}
   :=\sup_{\mathbf k}\frac{|z_{\mathbf k}|}{\theta_{\mathbf k}},
  \qquad
  \momentSpace{m}
   :=\{z:\|z\|_{\infty,\theta}<\infty\}.
\]
For measures we continue to use the ordinary total-variation norm
$\|\sigma\|_1=\sum_{\mathbf x}|\sigma(\mathbf x)|$.  When $m=1$, this notation
reduces to that of Section~\ref{sssec:1d-inversion}.

\begin{corollary}\label{thm:embed}
For every $m\ge1$, the operator $(\mathcal G^{-1})^{\otimes m}$ maps
$\momentSpace{m}$ continuously into $\ell^1(\BZ_{\ge0}^m)$, and
\[
  \bigl\|(\mathcal G^{-1})^{\otimes m}z\bigr\|_1
  \le \Cinv^{\,m}\|z\|_{\infty,\theta},
  \qquad
  \Cinv=\sum_{u\ge0}\theta_u\sum_{a=0}^{u}|\gamma_{u,a}|.
\]
\end{corollary}

\begin{proof}
For $m=1$ this is Theorem~\ref{thm:embed-1d}, which also shows
$\Cinv<\infty$.  For $m>1$ the row sums of the tensor product factor
coordinatewise, so the one-dimensional bound applies in each coordinate;
Tonelli's theorem then gives the stated bound.
\end{proof}

Let $P_D$ now denote coordinatewise truncation:
$(P_Dz)_{\mathbf k}=z_{\mathbf k}$ when $\|\mathbf k\|_\infty\le D$, and zero
otherwise.  The residual of a finite signed measure $\sigma$ on
$\BZ_{\ge0}^m$ for which
$m_{\mathbf k}(|\sigma|)<\infty$ whenever
$\|\mathbf k\|_\infty\le D$ is defined as in~\eqref{eq:residual-definition},
\[
  r_D(\sigma)
  :=(\operatorname{id}^{\otimes m}-Q_D^{\otimes m})\sigma
  =\sigma-(\mathcal G^{-1})^{\otimes m}P_D\momentseq{\sigma},
\]
with $Q_D$ the one-dimensional truncated reconstruction~\eqref{eq:trunc-op};
the second equality holds because both the moment transform and the truncation
act coordinatewise.  Here $P_D\momentseq{\sigma}$ denotes only the displayed
finite array of retained moments; the full moment sequence need not exist.

\begin{lemma}\label{lem:residual-multi}
Let $\sigma$ be a finite signed measure, and assume that
$m_{(D+1)\mathbf 1_S}(|\sigma|)$ is finite for every nonempty
$S\subseteq[m]$.  Then
\begin{equation}\label{eq:residual-multi}
  \|r_D(\sigma)\|_1
  \le\sum_{\emptyset\ne S\subseteq[m]}
  \bigl[C(D+1)2^{D(D+1)/2}\bigr]^{|S|}
  m_{(D+1)\mathbf 1_S}(|\sigma|).
\end{equation}
\end{lemma}

\begin{proof}
Recall the one-dimensional operator $Q_D$ of~\eqref{eq:trunc-op} and set
$R_D:=\operatorname{id}-Q_D$, so that $Q_D+R_D=\operatorname{id}$ in one
coordinate.  The proof of Lemma~\ref{lem:residual} exhibits $K_D$ as the
kernel of $R_D$ and gives the row-sum estimate
\[
  \sum_a |R_D(x,a)|
  \le C(D+1)2^{D(D+1)/2}\gtwo{x}{D+1}.
\]
By the definition of the tensor residual above,
$r_D(\sigma)=(\operatorname{id}^{\otimes m}-Q_D^{\otimes m})\sigma$.
Expanding $Q_D=\operatorname{id}-R_D$ gives
\[
  \operatorname{id}^{\otimes m}-Q_D^{\otimes m}
  =\sum_{\emptyset\ne S\subseteq[m]}
     (-1)^{|S|+1}R_D^{\otimes S}\otimes
     \operatorname{id}^{\otimes S^c}.
\]
Apply the one-dimensional row-sum estimate in the coordinates belonging to
$S$, then sum against $|\sigma|$.  Tonelli's theorem gives precisely the term
indexed by $S$ in~\eqref{eq:residual-multi}.
\end{proof}

If $\momentseq{|\sigma|}\in\momentSpace{m}$ and
$K:=\|\momentseq{|\sigma|}\|_{\infty,\theta}$, the definition of the weighted
norm gives
\[
  m_{(D+1)\mathbf 1_S}(|\sigma|)
  \le K\theta_{D+1}^{|S|}
  \qquad(\emptyset\ne S\subseteq[m]).
\]
Consequently, the term in~\eqref{eq:residual-multi} indexed by a set of size
$k$ is at most
$C^k(D+1)^k2^{-k(D+1)}K$, since
$2^{D(D+1)/2}\theta_{D+1}=2^{-(D+1)}$.

\begin{theorem}\label{thm:trunc-inv}
Let $\sigma$ be a finite signed measure, and suppose that all residual absolute
moments in~\eqref{eq:residual-multi} are finite.  Then
\begin{equation}\label{eq:trunc-inv}
  \|\sigma\|_1
  \le \Cinv^{\,m}\|P_D\momentseq{\sigma}\|_{\infty,\theta}
  +\sum_{\emptyset\ne S\subseteq[m]}
   \bigl[C(D+1)2^{D(D+1)/2}\bigr]^{|S|}
   m_{(D+1)\mathbf 1_S}(|\sigma|).
\end{equation}
For $m=2$,
\begin{equation}\label{eq:m2-trunc}
\begin{aligned}
  \|\sigma\|_1\le{}&\Cinv^2\|P_D\momentseq{\sigma}\|_{\infty,\theta}
  +C(D+1)2^{D(D+1)/2}
       \bigl(m_{(D+1,0)}(|\sigma|)+m_{(0,D+1)}(|\sigma|)\bigr)\\
  &+\bigl[C(D+1)2^{D(D+1)/2}\bigr]^2
       m_{(D+1,D+1)}(|\sigma|).
\end{aligned}
\end{equation}
\end{theorem}

\begin{proof}
By definition,
$\sigma=(\mathcal G^{-1})^{\otimes m}P_D\momentseq{\sigma}+r_D(\sigma)$.
Corollary~\ref{thm:embed} bounds the first term, and
Lemma~\ref{lem:residual-multi} bounds the second.  For $m=1$ this is
Theorem~\ref{thm:trunc-inv-1d}.
\end{proof}

Combining the tensor-product inverse bound with the multivariable residual
estimate gives the following convergence criterion for every fixed $m$.

\begin{corollary}\label{thm:rate}
Let $\sigma_n$ be finite signed measures on $\BZ_{\ge0}^m$.  Suppose that
there are constants $C<\infty$, $C_1>1$, $\qerr\in(0,1)$ and $\beta_0>0$
such that
\begin{enumerate}[\upshape(i)]
\item for every $n$ and every multi-index $\mathbf k$ with
  $\|\mathbf k\|_\infty\le\beta_0n$,
  \[
    |m_{\mathbf k}(\sigma_n)|
    \le CC_1^{\|\mathbf k\|_\infty}\theta_{\mathbf k}\qerr^n;
  \]
\item the absolute moment sequences are uniformly bounded in the
  \(\theta\)-weighted norm:
  \[
    K:=\sup_n\|\momentseq{|\sigma_n|}\|_{\infty,\theta}<\infty.
  \]
\end{enumerate}
Then there are $C'<\infty$ and $\eta\in(0,1)$ such that
$\|\sigma_n\|_1\le C'\eta^n$.  More precisely,
\[
  \|\sigma_n\|_1
  \le C\Cinv^m C_1^D\qerr^n+C''(D+1)^m2^{-(D+1)}K.
\]
Choosing $D=\lfloor\beta n\rfloor$ with $0<\beta<\beta_0$ sufficiently small gives
$\eta=\max\{2^{-\beta/2},C_1^\beta \qerr\}<1$.
At $\mathbf k=\mathbf0$, part~\textup{(i)} also controls the difference of
the total masses, so no separate normalization is required.
\end{corollary}

\begin{proof}
Apply Theorem~\ref{thm:trunc-inv}.  Part~\textup{(i)} bounds the
truncated term, while part~\textup{(ii)} and the weighted-norm estimate
preceding Theorem~\ref{thm:trunc-inv} bound the residual.  The terms with
$|S|\ge2$ are
dominated by those with $|S|=1$.  The stated choice of $D$ balances the two
bounds; its polynomial factor is absorbed by decreasing $\beta$.  For $m=1$
this is Corollary~\ref{thm:rate-1d}.
\end{proof}

In particular, convergence of the joint distribution to
$\mu_1\otimes\mu_2$ gives asymptotic independence at the same rate.  Indeed, if
$\|\Law(X_n,Y_n)-\mu_1\otimes\mu_2\|_1\le C\eta^n$, then marginalization and
the triangle inequality give
\[
  \|\Law(X_n,Y_n)-\Law(X_n)\otimes\Law(Y_n)\|_1\le3C\eta^n.
\]

\subsubsection{Truncated inversion on a good event}

For $m=2$, the criterion below has two cases.  Both restrict the probability
measure to a good event, compare its conditional mixed moments with a target
array, and apply \eqref{eq:m2-trunc}.  The cases differ in the control of the three residual
moments in that inequality and in the resulting truncation depth.  In the
applications below, the target array is a
product
$m_{u,v}=m_u^{(1)}m_v^{(2)}$ of two one-matrix target moment sequences from
Lemma~\ref{lem:1d-moment}; the corresponding one-dimensional inverses were
computed in the proof of Theorem~\ref{thm:1d} and are $\pisym$ for a zero
symmetry defect and $\piCL$ for a defect of high rank.

The required conditional mixed-moment estimate is converted to total
variation by the following criterion.

\begin{theorem}\label{thm:criterion}
For each $n$, let $(X_n,Y_n)$ be a pair of nonnegative integer-valued random
variables whose joint distribution is $\rho_n$, together with an auxiliary random element
\(\bbeta_n\) on the same probability space.  Let $\mu$ be a probability
measure on $\BZ_{\ge0}^2$ with Gaussian-binomial moments \(m_{u,v}\).
Suppose that one of the following two sets of hypotheses holds.
\begin{enumerate}[\upshape(I)]
\item \emph{Exponential loss in \(D\).}
The target array \((m_{u,v})\) belongs to \(\momentSpace{2}\).  Set
\(M_\theta:=\|(m_{u,v})\|_{\infty,\theta}\).
There are constants $C<\infty$, $C_0>1$, $\qerr\in(0,1)$, $\delta>0$ and
$\beta_0>0$ such that, for every integer $0\le D\le\beta_0n$, there are an
\(\bbeta_n\)-measurable event $\mathcal E_{n,D}$ and a nonnegative
\(\bbeta_n\)-measurable error weight $\mathrm{Err}_{n,D}$ with the following
properties:
\begin{enumerate}[\upshape(a)]
\item for almost every realization $\bbeta\in\mathcal E_{n,D}$ and all
$u,v\ge0$,
  \[
    \Exp\left[\gtwo{X_n}u\mid\bbeta\right]\le C\theta_u,
    \qquad
    \Exp\left[\gtwo{Y_n}v\mid\bbeta\right]\le C\theta_v;
  \]
\item for $0\le u,v\le D+1$ and every
  $\bbeta\in\mathcal E_{n,D}$,
  \[
    \left|\Exp\left[\gtwo{X_n}u\gtwo{Y_n}v\mid\bbeta\right]-m_{u,v}\right|
    \le C_0^D m_{u,v}\mathrm{Err}_{n,D}(\bbeta),
  \]
  and $\Exp[\mathrm{Err}_{n,D}\ind_{\mathcal E_{n,D}}]\le \qerr^n$;
\item $\Prob(\mathcal E_{n,D}^{c})\le e^{-\delta n}$.
\end{enumerate}

\item \emph{Exponential loss in \(D^2\).}
The target moments factor as $m_{u,v}=m_u^{(1)}m_v^{(2)}$, where
$m_0^{(1)}=m_0^{(2)}=1$ and, for some $M<\infty$,
\[
  m_u^{(i)}\le M2^{-u^2}\qquad(i=1,2;\ u\ge0).
\]
There are constants $c_{\mathrm{loss}},\beta_0,C_{\mathrm{err}},\delta>0$ and
$\qerr\in(0,1)$ such
that, for every integer $0\le D\le\beta_0\sqrt n$, there are an
\(\bbeta_n\)-measurable event $\mathcal E_{n,D}$ and a nonnegative
\(\bbeta_n\)-measurable error weight $\mathrm{Err}_{n,D}$ satisfying
\begin{enumerate}[\upshape(a)]
\item for $0\le u,v\le D+1$ and almost every realization
$\bbeta\in\mathcal E_{n,D}$,
 \[
 \left|\Exp\left[\gtwo{X_n}u\gtwo{Y_n}v\mid\bbeta\right]-m_{u,v}\right|
 \le 2^{c_{\mathrm{loss}}(D+1)^2}
       \theta_u\theta_v\mathrm{Err}_{n,D}(\bbeta);
 \]
\item $\Exp[\mathrm{Err}_{n,D}\ind_{\mathcal E_{n,D}}]
  \le C_{\mathrm{err}}\qerr^n$;
\item $\Prob(\mathcal E_{n,D}^{c})\le e^{-\delta n}$.
\end{enumerate}
\end{enumerate}
Then there are $C'<\infty$ and $\eta\in(0,1)$ such that
\[
  \|\rho_n-\mu\|_1\le C'\eta^n.
\]
In case~\textup{(I)}, if $m_{u,v}=m_u^{(1)}m_v^{(2)}$ and the inverse
one-dimensional moment sequences are probability measures $\mu_1,\mu_2$,
then $\mu=\mu_1\otimes\mu_2$.  The same conclusion, and hence asymptotic
independence, holds automatically in case~\textup{(II)}.
\end{theorem}

\begin{proof}
Both cases restrict $\rho_n$ to the good event, compare the conditional mixed
moments of the restriction with the target moments, and apply the two-dimensional
truncated inversion~\eqref{eq:m2-trunc} of Theorem~\ref{thm:trunc-inv}.  We
carry out these steps once and then specialize the residual estimates and the
truncation depth.

Let $D\ge0$ be an integer in the range allowed by the case at hand; it is
chosen at the end of that case.  Write
$\mathcal E=\mathcal E_{n,D}$, $p=\Prob(\mathcal E)$, and let
$\rho_n^{\mathcal E}$ be the restriction of $\rho_n$ to $\mathcal E$.  For the
signed measure
\[
  \sigma_n:=\rho_n^{\mathcal E}-p\mu
\]
we have
\begin{equation}\label{eq:peeling}
  \|\rho_n-\mu\|_1
  \le \|\sigma_n\|_1+2\Prob(\mathcal E^c).
\end{equation}
Because $\mathcal E$ is $\bbeta_n$-measurable, the moments of $\sigma_n$ are
the integrals over $\mathcal E$ of the conditional mixed-moment errors.  In
case~\textup{(I)}, the bound \(M_\theta<\infty\) turns the
relative bound~\textup{(I.b)} into the absolute bound
\[
 \left|\Exp\left[\gtwo{X_n}u\gtwo{Y_n}v\mid\bbeta\right]-m_{u,v}\right|
 \le L_D\,\theta_u\theta_v\mathrm{Err}_{n,D}(\bbeta),
 \qquad
 L_D:=C_0^D M_\theta,
\]
valid for $0\le u,v\le D+1$ and $\bbeta\in\mathcal E$; in case~\textup{(II)},
hypothesis~\textup{(II.a)} is the same inequality with
$L_D:=2^{c_{\mathrm{loss}}(D+1)^2}$.  Integrating over $\mathcal E$ and
using $\Exp[\mathrm{Err}_{n,D}\ind_{\mathcal E}]
\le C_{\mathrm{err}}\qerr^n$, with $C_{\mathrm{err}}=1$ in
case~\textup{(I)}, gives
\begin{equation}\label{eq:crit-signed}
  |m_{u,v}(\sigma_n)|
  \le C_{\mathrm{err}}L_D\,\theta_u\theta_v\qerr^n
  \qquad(0\le u,v\le D+1),
\end{equation}
so that the truncated term of~\eqref{eq:m2-trunc} is at most
$\Cinv^2C_{\mathrm{err}}L_D\qerr^n$.  The three residual absolute moments
in~\eqref{eq:m2-trunc} are estimated from
$|\sigma_n|\le\rho_n^{\mathcal E}+p\mu$, and their prefactors are handled
using $2^{D(D+1)/2}\theta_{D+1}=2^{-(D+1)}$.  Only these residual moments and
the choice of $D$ depend on the case.

\emph{Case \textup{(I)}.}  Choose $\beta\in(0,\beta_0)$ so small
that $C_0^\beta \qerr<1$, and put $D=\lfloor\beta n\rfloor$.  The
one-coordinate absolute moments satisfy, by~\textup{(I.a)},
\[
\begin{aligned}
  m_{D+1,0}(|\sigma_n|)
  &\le (C+M_\theta)\theta_{D+1},\\
  m_{0,D+1}(|\sigma_n|)
  &\le (C+M_\theta)\theta_{D+1},
\end{aligned}
\]
while condition~\textup{(I.b)} at $(D+1,D+1)$ gives
\[
  m_{D+1,D+1}(|\sigma_n|)
  \le \bigl(2+C_0^D\qerr^n\bigr)
       M_\theta\theta_{D+1}^2.
\]
Inserting these bounds and~\eqref{eq:crit-signed}
into~\eqref{eq:m2-trunc} yields
\[
  \|\sigma_n\|_1
  \ll C_0^D\qerr^n
     +(D+1)2^{-(D+1)}
     +(D+1)^2 2^{-2(D+1)}
       \bigl(1+C_0^D\qerr^n\bigr).
\]
For $D=\lfloor\beta n\rfloor$, the right-hand side is at most
$C'\eta^n$ for some $\eta\in(0,1)$.  Combining this estimate
with~\eqref{eq:peeling} and~\textup{(I.c)}
proves case~\textup{(I)}.  If the target moments factor as in the final
assertion of the theorem, the tensor-product identity for the inverse
transform gives
\(\mu=\mu_1\otimes\mu_2\).

\emph{Case \textup{(II)}.}  Unlike case~\textup{(I)}, this
case has no conditional marginal bound and allows a loss of $2^{O(D^2)}$ in
the mixed moments.  The faster decay \(2^{-u^2}\) of the target moments
compensates for this loss at the shorter truncation depth \(D\asymp\sqrt n\).
Choose $\beta\in(0,\beta_0)$ so small that
$2^{2c_{\mathrm{loss}}\beta^2}\qerr<1$ and put
$D=\lfloor\beta\sqrt n\rfloor$.  Here~\eqref{eq:crit-signed} reads
\begin{equation}\label{eq:strong-signed}
 |m_{u,v}(\sigma_n)|
 \le C_{\mathrm{err}}2^{c_{\mathrm{loss}}(D+1)^2}
       \theta_u\theta_v\qerr^n.
\end{equation}
Applying part~\textup{(II.a)} at $(D+1,0)$, $(0,D+1)$ and $(D+1,D+1)$, and
using $m_u^{(i)}\le M2^{-u^2}$ for the target moments, we obtain
\[
\begin{aligned}
 m_{D+1,0}(|\sigma_n|)+m_{0,D+1}(|\sigma_n|)
 &\ll 2^{-(D+1)^2}
   +2^{c_{\mathrm{loss}}(D+1)^2}\theta_{D+1}\qerr^n,\\
 m_{D+1,D+1}(|\sigma_n|)
 &\ll 2^{-2(D+1)^2}
   +2^{c_{\mathrm{loss}}(D+1)^2}\theta_{D+1}^2\qerr^n.
\end{aligned}
\]
After multiplication by the residual prefactors in~\eqref{eq:m2-trunc}, the
terms containing \(\qerr^n\) in these two lines are bounded by the truncated
signed-moment term and are absorbed into it.  Inserting the remaining terms
and~\eqref{eq:strong-signed} into~\eqref{eq:m2-trunc} gives
\[
 \|\sigma_n\|_1
 \ll 2^{c_{\mathrm{loss}}(D+1)^2}\qerr^n
    +(D+1)2^{-D^2/2+O(D)}
    +(D+1)^2 2^{-D^2+O(D)}.
\]
For $D=\lfloor\beta\sqrt n\rfloor$, the sum of the three terms is at most
$C'\eta^n$ for some $\eta\in(0,1)$.  Part~\textup{(II.c)}
and~\eqref{eq:peeling} prove convergence in
$\ell^1$.  Let \(\mu_1,\mu_2\) be the marginals of \(\mu\).  Their moment
sequences are \(m_u^{(1)}\) and \(m_v^{(2)}\), respectively, whereas
\(\mu_1\otimes\mu_2\) has mixed moments
\(m_u^{(1)}m_v^{(2)}=m_{u,v}\).  These arrays belong to
\(\momentSpace{2}\), and bounded tensor-product inversion is injective on this
weighted domain.  Hence \(\mu=\mu_1\otimes\mu_2\), proving asymptotic
independence.
\end{proof}

The good events in both cases may depend on the truncation depth, as the
support and cross-rank thresholds grow linearly with $D$.  Only the moments of
the good submeasure enter the residual;
no high-moment bound is required on the exponentially small bad event.

Case~\textup{(II)} of Theorem~\ref{thm:criterion} is applied to the bordered
matrix pair on the probability space \(\MatData_r(\bu_{-1})\) of
Section~\ref{ssec:finite-model}, whose good event is constructed in
Section~\ref{ssec:rmt-input} and whose conditional mixed moments are computed
with a $2^{O(D^2)}$ loss in Section~\ref{ssec:mix-proof};
case~\textup{(I)} suffices for the symmetric unbordered diagonal perturbations
of Section~\ref{ssec:unbordered}.  In both cases the target array is the product
of the two one-matrix target moment sequences, so the limit produced is the
corresponding product of one-matrix limiting distributions.

\subsection{The bordered-pair theorem and the good event}\label{ssec:rmt-input}

The estimate required in \eqref{eq:family-transfer-bd} is uniform in the
\(2\)-adic stratum and the sign vector.

\begin{theorem}
\label{thm:rmt0}
Let $d_0>1$ be squarefree and $\alpha\in(0,\tfrac12)$.
There exist $C_{\mathrm{rmt}}=C_{\mathrm{rmt}}(d_0,\alpha)<\infty$ and
$\eta=\eta(d_0,\alpha)\in(0,1)$ such that, for every $r\ge2$ and every sign
vector $\bu_{-1}\in\BF_2^r$ satisfying $|\bu_{-1}|\ge\alpha r$, on the probability space
$\MatData_r(\bu_{-1})$ of Section~\ref{ssec:finite-model}, the corank pair
satisfies
\[
  \bigl\|\Law(Z_1,Z_2)-\piCL\otimes\piCL\bigr\|_1
  \le C_{\mathrm{rmt}}\,\eta^{\,r}.
\]
In particular, the limit is independent of the stratum and of the sign vector
in the indicated range.
\end{theorem}

The proof applies case~\textup{(II)} of Theorem~\ref{thm:criterion} with the
border data as the auxiliary random element.  The present subsection fixes that
conditioning and constructs the good event $\mathcal E_{n,D}$ together with its
probability bound~\textup{(II.c)}.  The remaining input, hypothesis
\textup{(II.a)}, is the conditional mixed-moment estimate of
Proposition~\ref{thm:mix}; the deduction of Theorem~\ref{thm:rmt0} is therefore
completed in Section~\ref{ssec:mix-proof}, after that proposition has been
proved.

Fix $d_0$, $r\ge2$ and a sign vector
$\bu_{-1}$ with $|\bu_{-1}|\ge\alpha r$, and take all random variables on
$\MatData_r(\bu_{-1})$ with the uniform measure of
Section~\ref{ssec:finite-model}. We write
\[
  n:=r-1,\qquad
  \Omega:=\bu_{-1}^\circ(\bu_{-1}^\circ)^{\top}
          +\diag(\bu_{-1}^\circ),\qquad
  \Gamma:=\diag(\bu_{d_0})+\Omega,\qquad
  h:=|\supp\bu_{d_0}|.
\]
We call
\(\bbeta:=(\bu_2,\bu_{\Delta_1},\dots,\bu_{\Delta_s},\bc)\) the
\emph{border data}; it determines $\bu_{d_0}$, the two border blocks and the corner
block, and the core block $A\sim\Unif(\Sym_n(\Omega))$ is independent of it.
The matrices \(\mathsf U,\mathsf L\) have size \(n\times b_2\), and
\[
  N_2=\begin{pmatrix}A+\diag(\bu_{d_0})&\mathsf U\\
  \mathsf L^{\top}&C\end{pmatrix},
\]
while $N_1$ is the matrix of Theorem~\ref{thm:block}, of border width
$b_1\le1$.  Since \(d_0>1\), one has \(b_2\ge1\), and consequently
\(b_1\le b_2\).

The character vector \(\bu_{d_0}\) is uniform on \(\BF_2^n\),
independently of the sign vector.  Write $d_0=2^eq_1\cdots q_s$, as in
Section~\ref{ssec:block}, and use
\[
  \asm{q_k}{p}=\asm{q_k^*}{p}+\asm{-1}{q_k}\asm{-1}{p}.
\]
For each \(i\le n\), the definition of \(\bu_{d_0}\) gives
\[
  u_{d_0,i}=e\,u_{2,i}+\sum_{k\le s}
       \bigl(u_{\Delta_k,i}+\asm{-1}{q_k}u_{-1,i}\bigr).
\]
This is a nonconstant affine function of
\((u_{2,i},u_{\Delta_1,i},\dots,u_{\Delta_s,i})\), because \(d_0>1\) forces
\(e=1\) or \(s\ge1\).  These coordinate tuples are independent and uniform as \(i\)
varies, so the coordinates \(u_{d_0,i}\) are independent and uniform.

\begin{lemma}\label{lem:bad-event}
Let $\alpha\in(0,\tfrac12)$ and $|\bu_{-1}|\ge\alpha r$.  For
$H\le n/4$, define
\begin{equation}\label{eq:good-event}
\begin{aligned}
  \mathcal E_{\mathrm{good}}(H):={}&
  \{h\ge H\}\cap\{\rank\Gamma\ge H\}
  \cap\{\rank[\,\mathsf L_1^{\top}\ \mathsf C_1\,]=b_1\}
  \cap\{\rank[\,\mathsf L^{\top}\ C\,]=b_2\}\\
  &\cap\{\ker\mathsf U=\bzero\}
   \cap\{b_1=0\text{ or }\bu_2\ne\bzero\}.
\end{aligned}
\end{equation}
There is $\delta=\delta(\alpha,d_0)>0$ such that, uniformly for
$H\le n/4$,
\[
  \Prob\bigl(\mathcal E_{\mathrm{good}}(H)^c\bigr)
  \le e^{-\delta n}
\]
for all sufficiently large $n$.
\end{lemma}

\begin{proof}
By the preceding calculation, $h=|\supp\bu_{d_0}|$ is binomial with parameters
$(n,\tfrac12)$, so Hoeffding's inequality gives
$\Prob(h<n/4)\le e^{-n/8}$.  Also
\[
  \Gamma=\diag(\bu_{d_0}+\bu_{-1}^\circ)
          +\bu_{-1}^\circ(\bu_{-1}^\circ)^{\top},
\]
and therefore
\[
  \rank\Gamma\ge
  \#\{i\le n:u_{d_0,i}\ne u_{-1,i}\}-1.
\]
The count on the right is again binomial with parameters $(n,\tfrac12)$,
uniformly in the fixed sign vector.  Hence, for every $H\le n/4$,
\[
 \Prob(\rank\Gamma<H)
 \le\Prob\bigl(\operatorname{Bin}(n,\tfrac12)<H+1\bigr)
 \le e^{-c n}.
\]
Together with the estimate for \(h\), this controls the first two conditions,
after changing the constants for finitely many \(n\).

For the border rows of $N_2$, the vectors
$\bu_{\Delta_1},\dots,\bu_{\Delta_s}$ are independent and uniform
in $\BF_2^n$.  The
possible row indexed by $0$ has $\BF_2^n$-part
$\bu_{\Delta_0}\in
\{\bu_{-1}^\circ,\bu_2,\bu_{-1}^\circ+\bu_2\}$.  In the last two
cases it contains the independent uniform vector $\bu_2$; in the first it is
the fixed nonzero vector $\bu_{-1}^\circ$.  Thus every nontrivial linear
combination of the $b_2$ rows has a uniform, or fixed nonzero, $\BF_2^n$-part.
A union bound gives
\[
  \Prob\bigl(\rank[\,\mathsf L^{\top}\ C\,]<b_2\bigr)\le 2^{b_2-n}.
\]
If $b_1=1$, the \(\BF_2^n\)-part of the border row of $N_1$ is
$(\bu_{-1}^\circ)^{\top}$, which is nonzero for all sufficiently large $r$
because $|\bu_{-1}|\ge\alpha r$.  Thus
\(\rank[\,\mathsf L_1^{\top}\ \mathsf C_1\,]=b_1\); the assertion is
vacuous when \(b_1=0\).

The columns of $\mathsf U$ are \(\bu_{q_k}\).  For \(k\ge1\),
\(\bu_{q_k}=\bu_{\Delta_k}+\asm{-1}{q_k}\bu_{-1}^\circ\), and, when present,
\(\bu_{q_0}=\bu_2\).  Every nontrivial linear combination of these columns is
uniform in $\BF_2^n$, up to a fixed translation.  Hence
\[
  \Prob(\ker\mathsf U\ne\bzero)\le 2^{b_2-n}.
\]
When \(b_1=1\), also \(\Prob(\bu_2=\bzero)=2^{-n}\).
Adding the estimates proves the lemma.
\end{proof}

The bounds for \(h\) and \(\rank\Gamma\) supply the codimensions of the
overlap and cross-radical bad subspaces in Lemma~\ref{lem:weighted-sum}; the
two full-row-rank conditions give \(\dim H_1=\dim H_2=n\) in
Proposition~\ref{prop:mb-reduce}; and \(\ker\mathsf U=0\), together with
\(\bu_2\ne0\) when \(b_1=1\), makes the two coordinate projections injective
on every tuple span of nonzero weight.

There are two distinct exceptional sets.  The event
$\mathcal E_{\mathrm{good}}(H)$ constrains the border data alone; its complement
has probability at most $e^{-\delta n}$ and is discarded once, in the peeling
step~\eqref{eq:peeling}.  After a realization
$\bbeta\in\mathcal E_{\mathrm{good}}(H)$ has been fixed, exceptional kernel
tuples still occur inside the conditional expectation over the core $A$.
Section~\ref{ssec:mix-proof} controls them by the codimensions and ranks
guaranteed by~\eqref{eq:good-event}, rather than by another probability estimate
for the border data.

\subsection{Conditional mixed moments of bordered pairs}\label{ssec:mix-proof}

Proposition~\ref{thm:mix} is the conditional mixed-moment estimate required by
hypothesis~\textup{(II.a)} of Theorem~\ref{thm:criterion}: for border data in
the good event, the conditional Gaussian-binomial mixed moments of the corank
pair agree with the product $m_u(\Omega)m_v(\Omega)$ of two one-matrix target
moments, up to a factor $2^{O(D^2)}$ times the averaged error variable that
supplies~\textup{(II.b)}.

\begin{proposition}\label{thm:mix}
Let $d_0>1$ be squarefree and $\alpha\in(0,\tfrac12)$.  There is a constant
$c_{\mathrm{mix}}>0$, depending only on $d_0$ and $\alpha$, with the following
property.
Fix an integer $r\ge2$ and a sign vector $\bu_{-1}$ with
$|\bu_{-1}|\ge\alpha r$.  Put
\[
  b=b_1+b_2,\qquad H_D=4D+6b_2+b+12.
\]
For every $D\ge0$ such that
\begin{equation}\label{eq:mix-rank-range}
  \rank\Omega\ge4D+8b_2+10,
\end{equation}
all $0\le u,v\le D+1$, and every realization $\bbeta$ of the border data in
$\mathcal E_{\mathrm{good}}(H_D)$, the conditional expectation under
\(\MatData_r(\bu_{-1})\) satisfies
\begin{equation}\label{eq:mix}
\begin{aligned}
  &\Bigl|\Exp_A\bigl[\gtwo{Z_1}u\gtwo{Z_2}v\bigm|\bbeta\bigr]
    -m_u(\Omega)m_v(\Omega)\Bigr|\\
  &\qquad\le
  2^{c_{\mathrm{mix}}(D+b_2+b+1)^2}\theta_u\theta_v
  \Bigl(2^{-h}+2^{-\rank\Gamma}+2^{-n}+2^{-\rank\Omega}\Bigr).
\end{aligned}
\end{equation}
\end{proposition}

The conditional moment calculation separates into two steps.  First the
bordered kernel equations are converted exactly into a weighted sum over two
families of tuples.  That sum is then estimated by separating pairs with
disjoint projected spans from the exceptional pairs.  We introduce the common
terminology before carrying out the first step.

Let \(V_1,V_2\) be \(n\)-dimensional \(\BF_2\)-spaces, let
\(\psi_i:V_i\to\BF_2^n\) be linear maps and \(\Omega_i\) an alternating
form on \(V_i\), for \(i=1,2\), and let
\(B_{12}:V_1\times V_2\to\BF_2\) be bilinear.  For ordered linearly
independent tuples
\(\bw_\bullet=(\bw_1,\dots,\bw_u)\) and
\(\bw'_\bullet=(\bw'_1,\dots,\bw'_v)\), put
\[
  \mathcal X:=\psi_1(\spn\bw_\bullet),\qquad
  \mathcal Y:=\psi_2(\spn\bw'_\bullet),\qquad
  k:=\dim(\mathcal X+\mathcal Y).
\]
Given a predicate \(\mathcal R_{u,v}(\bw_\bullet,\bw'_\bullet)\), call the
tuple pair \emph{compatible} if
\begin{enumerate}[\upshape(i)]
\item \(\psi_1,\psi_2\) are injective on
  \(\spn\bw_\bullet,\spn\bw'_\bullet\), respectively;
\item the two spans are totally isotropic for \(\Omega_1\) and
  \(\Omega_2\), respectively;
\item \(B_{12}(\bw_i,\bw'_j)=0\) for every \(i,j\);
\item \(\mathcal R_{u,v}(\bw_\bullet,\bw'_\bullet)\) holds.
\end{enumerate}
The associated compatible-pair sum is
\begin{equation}\label{eq:compatible-pair-sum}
  \mathcal W_{u,v}:=\frac1{|\GL_u||\GL_v|}
  \sum_{(\bw_\bullet,\bw'_\bullet)\ \mathrm{compatible}}
  2^{-\left(nk-\binom k2\right)}.
\end{equation}
Thus compatibility records exactly the algebraic conditions on a tuple pair;
once they hold, the weight depends only on the dimension \(k\) of the joint
projected span.

For \(S\subseteq\{1,\dots,n\}\), let
\(\pi_S:\BF_2^n\to\BF_2^S\) be the coordinate projection and put
\(V_{S^c}:=\ker\pi_S\).

\begin{proposition}\label{prop:mb-reduce}
Let $H\ge0$ and $\bbeta\in\mathcal E_{\mathrm{good}}(H)$, and use
\(\mathsf U_1,\mathsf L_1,\mathsf C_1\) from the assembly map
\eqref{eq:assembly-map}.  Thus
\[
 N_1=\begin{pmatrix}A&\mathsf U_1\\\mathsf L_1^{\top}&\mathsf C_1\end{pmatrix},
 \qquad
 N_2=\begin{pmatrix}A+\diag(\bu_{d_0})&\mathsf U\\
 \mathsf L^{\top}&C\end{pmatrix}.
\]
Put
\[
 H_1:=\ker[\,\mathsf L_1^{\top}\ \mathsf C_1\,]\subseteq\BF_2^{n+b_1},
 \qquad
 H_2:=\ker[\,\mathsf L^{\top}\ C\,]\subseteq\BF_2^{n+b_2},
\]
so both spaces have dimension $n$, and let \(\psi_1,\psi_2\) drop the last
$b_1$, respectively $b_2$, coordinates.  Define the restricted alternating
and cross forms directly
by
\[
 \Omega_1:=
 \left.
 \begin{pmatrix}\Omega&\mathsf U_1\\\mathsf U_1^{\top}&0_{b_1}\end{pmatrix}
 \right|_{H_1},
 \qquad
 \Omega_2:=
 \left.
 \begin{pmatrix}\Omega&\mathsf U\\\mathsf U^{\top}&0_{b_2}\end{pmatrix}
 \right|_{H_2},
 \qquad
 B_{12}:=
 \left.
 \begin{pmatrix}\Gamma&\mathsf U\\
 \mathsf U_1^{\top}&0_{b_1\times b_2}\end{pmatrix}
 \right|_{H_1\times H_2}.
\]
Finally, with $S=\supp\bu_{d_0}$, put
\[
 \Sigma:=\pi_S\{\mathsf U_1\mathbf a+\mathsf U\mathbf b:
   \mathbf a\in\BF_2^{b_1},\ \mathbf b\in\BF_2^{b_2}\},
 \qquad \dim\Sigma\le b:=b_1+b_2.
\]
Set \(W_\Sigma:=\pi_S^{-1}(\Sigma)\).
For tuple pairs $\bw_i=(\bx_i,\mathbf a_i)\in H_1$ and
$\bw'_j=(\by_j,\mathbf b_j)\in H_2$, let
$\mathcal R_{u,v}$ be the following target-consistency condition: whenever
\[
  \bz=\sum_i \xi_i\bx_i=\sum_j \zeta_j\by_j,
\]
one has
\begin{equation}\label{eq:target-consistency}
 \mathsf U_1\sum_i \xi_i\mathbf a_i
 =\diag(\bu_{d_0})\bz+\mathsf U\sum_j \zeta_j\mathbf b_j.
\end{equation}
In the compatible-pair sum~\eqref{eq:compatible-pair-sum}, take \(V_i=H_i\)
and use the maps, alternating forms, and predicate just defined.  Then, for
all $u,v\ge0$,
\[
  \Exp_A\bigl[\gtwo{Z_1}u\gtwo{Z_2}v\bigm|\bbeta\bigr]
  =\mathcal W_{u,v},
\]
and the predicate has the two properties
\[
 \mathcal R_{u,v}\Longrightarrow\mathcal X\cap\mathcal Y\subseteq W_\Sigma,
\]
\[
 \mathcal X\cap\mathcal Y=\{\bzero\}
 \ \text{and }\psi_1,\psi_2\text{ injective on the two tuple spans}
 \Longrightarrow\mathcal R_{u,v}.
\]
Finally,
\[
 \dim\ker\psi_1,\ \dim\ker\psi_2\le b_2,\qquad
 \operatorname{codim}_{H_1}\psi_1^{-1}(W_\Sigma)\ge h-b_2-b,
\]
\[
 |\rank\Omega_1-\rank\Omega|,
 |\rank\Omega_2-\rank\Omega|\le4b_2,
\]
and
\[
  \rank B_{12}\ge\rank\Gamma-2(b_1+b_2).
\]
\end{proposition}

\begin{proof}
\emph{Kernel tuples.}
The kernel conditions are
\[
\begin{aligned}
 N_1(\bx,\mathbf a)^{\top}=0
 &\iff A\bx=\mathsf U_1\mathbf a,
       \quad(\bx,\mathbf a)\in H_1,\\
 N_2(\by,\mathbf b)^{\top}=0
 &\iff A\by=\diag(\bu_{d_0})\by+\mathsf U\mathbf b,
       \quad(\by,\mathbf b)\in H_2.
\end{aligned}
\]
The two full-row-rank conditions in the good event give
\(\dim H_1=\dim H_2=n\).

Expand each Gaussian binomial coefficient by ordered bases:
\[
 \gtwo{Z_1}u\gtwo{Z_2}v
 =\frac1{|\GL_u||\GL_v|}
  \sum_{\substack{\bw_\bullet\ \mathrm{independent\ in}\ \ker N_1\\
                  \bw'_\bullet\ \mathrm{independent\ in}\ \ker N_2}}1.
\]
For a first-family tuple occurring in this sum, a relation
$\sum_i \xi_i\bx_i=0$ gives
$\mathsf U_1\sum_i \xi_i\mathbf a_i=0$.  The map $\mathsf U_1$ is
injective on the good event.  This is vacuous when $b_1=0$ and follows from
$\bu_2\ne0$ when $b_1=1$.  Thus the corresponding relation among the full
vectors is zero.  Likewise, a relation among the projected second-family
vectors gives
$\mathsf U\sum_j\zeta_j\mathbf b_j=0$, and $\ker\mathsf U=0$ on the
good event.  Therefore $\psi_1$ and $\psi_2$ are injective on every span
contributing nonzero weight.

\emph{Compatibility with the common core.}
For fixed kernel tuples, the equations imposed on \(A\) have source vectors
\(\bx_i,\by_j\).  The two families have target vectors
\[
  \mathsf U_1\mathbf a_i,
  \qquad
  \diag(\bu_{d_0})\by_j+\mathsf U\mathbf b_j,
\]
respectively.  If the combined projected family is dependent, these targets
define a linear map on \(\mathcal X+\mathcal Y\) precisely when
\eqref{eq:target-consistency} holds.  Thus condition~\textup{(iv)} in the
definition of compatibility is exactly the consistency of the target
assignment, rather than an additional estimate.

Once the target assignment is well defined, the compatibility equations in
Lemma~\ref{lem:joint-surj} have three types.  Two first-family vectors give
  \(\bw_i^{\top}\Omega_1\bw_j=0\), and two second-family vectors give
  \((\bw'_i)^{\top}\Omega_2\bw'_j=0\).  A first-family vector against a
second-family vector gives
\[
  \bx_i^{\top}\Gamma\by_j
  +\bx_i^{\top}\mathsf U\mathbf b_j
  +\mathbf a_i^{\top}\mathsf U_1^{\top}\by_j=0,
\]
which is the cross condition for \(B_{12}\).  Restriction to $H_1$ and
$H_2$ gives compatibility conditions~\textup{(ii)} and~\textup{(iii)} for
\(\Omega_1,\Omega_2\), and \(B_{12}\).

Apply Lemma~\ref{lem:joint-surj} to a basis of the joint projected span.
If any consistency or compatibility condition fails, the probability is zero;
otherwise it is
$2^{-(nk-\binom k2)}$, where $k$ is the dimension of the joint projected span;
the preceding ordered-basis expansion now gives
\eqref{eq:compatible-pair-sum}.

\emph{Overlap and rank bounds.}
When the two projections are injective on the corresponding tuple spans, the
predicate $\mathcal R_{u,v}$ is automatic if the two projected spans have zero
intersection: the only common projected relation is then the zero relation on
both sides.  If $\mathcal R_{u,v}$ holds and \(z\in\mathcal X\cap\mathcal Y\),
then
\eqref{eq:target-consistency} gives
\[
  \diag(\bu_{d_0})z
  =\mathsf U_1\mathbf a_0+\mathsf U\mathbf b_0.
\]
On $S=\supp\bu_{d_0}$ the left side restricts to $\pi_S(z)$, and hence
$\pi_S(z)\in\Sigma$.  This proves the two asserted properties of
\(\mathcal R_{u,v}\).

The projection kernels have
dimensions at most $b_1$ and $b_2$.  Restricting the coordinate projection to
$H_1$ can lose at most $b_1$ ranks, and restricting it to $H_2$ can lose at
most $b_2$ ranks.  Passing from $V_{S^c}$ to $W_\Sigma$ loses at most
$\dim\Sigma\le b$ further ranks.  Thus
\[
 \operatorname{codim}_{H_1}\psi_1^{-1}(W_\Sigma)\ge h-b_1-b\ge h-b_2-b.
\]
Bordering an alternating form by width at most $b_2$ changes its rank by at
most $2b_2$, and restricting it to a subspace of codimension at most $b_2$
changes the rank by at most another $2b_2$.  Hence
\[
 |\rank\Omega_1-\rank\Omega|\le4b_2,
 \qquad |\rank\Omega_2-\rank\Omega|\le4b_2.
\]
Finally, adjoining rows and columns cannot lower the rank of the top-left block
$\Gamma$, while restriction to $H_1\times H_2$ lowers rank by at most
$b_1+b_2$ on each side.  This gives the stated conservative bound
$\rank B_{12}\ge\rank\Gamma-2(b_1+b_2)$.
\end{proof}

In the bordered setting, the two projected kernel spaces can overlap.
The first property of \(\mathcal R_{u,v}\) in
Proposition~\ref{prop:mb-reduce} shows that a compatible pair with
\(\mathcal X\cap\mathcal Y\ne0\) must have
\(\mathcal X\cap W_\Sigma\ne0\), where
\(W_\Sigma=\pi_S^{-1}(\Sigma)\) has large codimension.  The following relative
count quantifies the resulting loss.  It is the input that controls
overlapping compatible pairs in Lemma~\ref{lem:weighted-sum}.

\begin{lemma}\label{lem:iso-deficit}
Let $B$ be an alternating form on $\BF_2^{\,n}$, with radical
$\rad(B)$, and let $K\subseteq\BF_2^{\,n}$ have codimension $\gamma$.
Call an ordered linearly independent $u$-tuple $\bx_\bullet$
\emph{radical-free isotropic} if its span $\mathcal X$ is totally isotropic
for \(B\) and $\mathcal X\cap\rad(B)=\bzero$.  Let $\mathcal N$ be the number
of such tuples and $\mathcal N_{\ge g}$ the number among them with
$\dim(\mathcal X\cap K)\ge g$. If $\rank B\ge2u+2$ and $n\ge2u+2$, then for
$0\le g\le u$
\begin{equation}\label{eq:iso-deficit-ratio}
  \mathcal N_{\ge g}\ \le\ C^{\,u}\,
  2^{\,g(u-g)+\binom g2-g\gamma}\;\mathcal N ,
\end{equation}
with $C$ absolute.
\end{lemma}

\begin{proof}
Put $R:=\rad(B)$.  Say a tuple is
\emph{$K$-adapted} if $\bx_1,\dots,\bx_g\in K$, and let $\mathcal A_g$ be the
number of radical-free isotropic $K$-adapted tuples.

\emph{Adapted bases.} Fix an
radical-free isotropic $\mathcal X$ with
$m:=\dim(\mathcal X\cap K)\ge g$. Among the
$|\GL_u|=\prod_{i<u}(2^{u}-2^{i})$ ordered bases of $\mathcal X$, those that are
$K$-adapted number $\prod_{i<g}(2^{m}-2^{i})\cdot\prod_{i=g}^{u-1}(2^{u}-2^{i})$,
so the $K$-adapted ones make up a fraction
$\prod_{i<g}\frac{2^{m}-2^{i}}{2^{u}-2^{i}}\ge\eta_\infty(2)\,2^{-g(u-m)}
\ge\eta_\infty(2)\,2^{-g(u-g)}$ of them, using $m\ge g$. Summing over the
radical-free isotropic $\mathcal X$ with $\dim(\mathcal X\cap K)\ge g$ gives
$\mathcal A_g\ge\eta_\infty(2)2^{-g(u-g)}\mathcal N_{\ge g}$.

\emph{Counting adapted tuples.} Build a
tuple one vector at a time. Having chosen $\bx_1,\dots,\bx_i$ spanning an
isotropic $U_i$ with $U_i\cap R=\bzero$, isotropy confines $\bx_{i+1}$ to
$U_i^{\perp_B}$, whose dimension is $n-\dim B(U_i)=n-i$ exactly, because
$U_i\cap R=\bzero$ makes $B$ injective on $U_i$. Hence an unconstrained position
offers at most $2^{\,n-i}$ choices, while a position required to lie in $K$
offers at most $|K|=2^{\,n-\gamma}$. Therefore
\[
  \mathcal A_g\ \le\ \prod_{i<g}2^{\,n-\gamma}
  \cdot\prod_{i=g}^{u-1}2^{\,n-i}
  \ =\ 2^{\,nu-\binom u2}\cdot2^{\,\binom g2-g\gamma},
\]
since $\sum_{i<g}(n-\gamma)+\sum_{i=g}^{u-1}(n-i)
=nu-\binom u2+\binom g2-g\gamma$. Conversely, a lower bound for $\mathcal N$: at
step $i$ one must avoid $U_i$ itself ($2^{i}$ vectors) and avoid creating
an intersection with the radical, i.e.\ avoid the set $U_i+R$ of at most
$2^{\,i+n-\rank B}$ vectors, so the number of choices is at least
$2^{\,n-i}-2^{i}-2^{\,i+n-\rank B}\ge\tfrac12\,2^{\,n-i}$ whenever
$i\le u-1$, $n\ge2u+2$ and $\rank B\ge2u+2$. Hence
$\mathcal N\ge2^{-u}2^{\,nu-\binom u2}$, which gives the required lower bound
for \(\mathcal N\).

Combining the two steps gives~\eqref{eq:iso-deficit-ratio}.
\end{proof}

The comparison is made against the isotropic count \(\mathcal N\) itself.
Both \(\mathcal N\) and \(\mathcal N_{\ge g}\) contain only tuples whose spans
avoid \(\rad(B)\).  Tuples meeting the radical are not covered by
\eqref{eq:iso-deficit-ratio}; they are estimated by a separate error term in
Lemma~\ref{lem:weighted-sum}.

\begin{lemma}\label{lem:weighted-sum}
Let \(\mathcal W_{u,v}\) be the compatible-pair sum
\eqref{eq:compatible-pair-sum}, formed from \(n\)-dimensional spaces
\(V_1,V_2\), maps \(\psi_1,\psi_2\), alternating forms
\(\Omega_1,\Omega_2\), a bilinear form \(B_{12}\), and predicates
\(\mathcal R_{u,v}\).  Fix a subset
$S\subseteq\{1,\dots,n\}$ with $h:=|S|$, a subspace
$\Sigma\subseteq\BF_2^{\,S}$ with $\dim\Sigma\le b$, and a nonzero matrix
$\Omega\in\Alt_n$.  Put
\[
  W_\Sigma:=\pi_S^{-1}(\Sigma),
  \qquad \operatorname{codim}W_\Sigma\ge h-b.
\]
Assume, for some fixed $T\ge0$,
\begin{enumerate}[\upshape(P1)]
\item
  \[
    \dim\ker\psi_1,\ \dim\ker\psi_2\le T,
    \qquad
    \operatorname{codim}_{V_1}\psi_1^{-1}(W_\Sigma)
       \ge h-T-b,
  \]
\item
  \[
    |\rank\Omega_1-\rank\Omega|\le4T,
    \qquad |\rank\Omega_2-\rank\Omega|\le4T.
  \]
\end{enumerate}
Suppose that the predicates have the two properties
\begin{enumerate}[\upshape(R1)]
\item $\mathcal R_{u,v}$ implies
  $\mathcal X\cap\mathcal Y\subseteq W_\Sigma$;
\item $\mathcal R_{u,v}$ holds whenever
  $\mathcal X\cap\mathcal Y=\{\bzero\}$ and $\psi_1,\psi_2$ are injective on
  $\spn\bw_\bullet,\spn\bw'_\bullet$, respectively.
\end{enumerate}
There is an absolute constant $c_*>0$ such that, for $0\le u,v\le D+1$, if
\[
 h\ge4D+2T+b+10,
 \qquad \rank B_{12}\ge4D+2T+10,
 \qquad \rank\Omega\ge4D+8T+10,
\]
then
\begin{equation}\label{eq:weighted-sum}
\begin{aligned}
 |\mathcal W_{u,v}-m_u(\Omega)m_v(\Omega)|
 \le{}&2^{c_*(D+T+b+1)^2}\theta_u\theta_v\\
 &\times\left(2^{-h}+2^{-\rank B_{12}}
       +2^{-n}+2^{-\rank\Omega}\right).
\end{aligned}
\end{equation}
\end{lemma}

\begin{proof}
Put the effective depth and combined error
\[
  D_*:=D+T+b+1,\qquad
  \varepsilon_*:=2^{-h}+2^{-\rank B_{12}}
                  +2^{-n}+2^{-\rank\Omega}.
\]
Throughout the proof, \(C\) denotes an absolute constant that may change from
line to line.  Since \(\Omega\ne0\), the target moments satisfy
\[
  m_j(\Omega)=\frac1{|\GL_j|}\le C\theta_j.
\]
Also put
\[
  K_\Sigma:=\psi_1^{-1}(W_\Sigma),
  \qquad
  K_{12}:=\{\bx\in V_1:B_{12}(\bx,\by)=0
                 \text{ for every }\by\in V_2\},
\]
so that
\[
 \operatorname{codim}K_\Sigma\ge h-T-b,
 \qquad
 \operatorname{codim}K_{12}=\rank B_{12}.
\]
For a compatible pair, property~\textup{(R1)} implies
\begin{equation}\label{eq:s-le-g}
  \dim(\mathcal X\cap\mathcal Y)
  \le \dim(\spn\bw_\bullet\cap K_\Sigma).
\end{equation}

There are four sources of loss.  The subspaces \(K_\Sigma\) and
\(K_{12}\) control, respectively, overlap of the projected spans and
rank loss in the cross equations.  The projection kernels have codimension at
least \(n-T\); they measure tuples omitted when compatibility condition
\textup{(i)} is imposed.  Finally, the radicals of the alternating forms are
controlled by
\[
 \rank\Omega_1,\ \rank\Omega_2\ge\rank\Omega-4T.
\]
The two subspace sources and the radical source contribute compatible pairs
outside the main term.  The projection kernels instead enter only when the
main term is compared with the full one-matrix counts.

\emph{Subspace-deficit estimate.}
Let \(K\subseteq V_1\) have codimension \(\gamma\).  Consider compatible pairs
for which the first tuple span avoids \(\rad(\Omega_1)\) but meets \(K\)
nontrivially.  Their contribution to \(\mathcal W_{u,v}\) is
\begin{equation}\label{eq:bad-subspace-first}
  O\!\left(2^{CD_*^2-\gamma}m_u(\Omega)m_v(\Omega)\right).
\end{equation}

To prove~\eqref{eq:bad-subspace-first}, put
\[
 t:=\dim(\spn\bw_\bullet\cap K)\ge1,\qquad
 g:=\dim(\spn\bw_\bullet\cap K_{12}),\qquad
 s:=\dim(\mathcal X\cap\mathcal Y).
\]
There are three factors.  First, Lemma~\ref{lem:iso-deficit} bounds the number
of first tuples, relative to the radical-free isotropic count, by
\[
  2^{-t\gamma+O(D_*^2)}.
\]
The hypotheses of that lemma follow from
\(\rank\Omega_1\ge\rank\Omega-4T\ge2u+2\) and
\(n\ge h\ge2u+2\); Lemma~\ref{lem:isotropic} then compares the
radical-free count with the baseline isotropic count within a factor
\(2^{O(D_*^2)}\).

Second, let
\[
 L:=\{\by\in V_2:B_{12}(\bw_i,\by)=0\text{ for every }i\}.
\]
The map \(\spn\bw_\bullet\to V_2^\vee\) induced by \(B_{12}\) has kernel
\(\spn\bw_\bullet\cap K_{12}\), of dimension \(g\); hence
\(\dim L=n-u+g\).  Put
\(G:=L\cap\psi_2^{-1}(\mathcal X)\), so that \(\dim G\le u+T\).
Compatibility makes \(\psi_2\) injective on the second tuple span, and
therefore
\[
 \dim(\spn\bw'_\bullet\cap G)
 =\dim(\mathcal X\cap\mathcal Y)=s.
\]
Lemma~\ref{lem:fallin}, applied in \(L\) with fall-in parameter \(s\), and
with the second-family isotropy condition discarded for an upper bound, gives
\[
  2^{\,gv-s(n-2u+g-T-v)+O(D_*^2)}
\]
relative to the isotropic second-tuple count in an
\((n-u)\)-dimensional space.  Here \(gv\) records the enlargement of \(L\);
the loss from counting unrestricted rather than isotropic tuples is
\(2^{O(v^2)}\), uniformly in the other parameters, and is absorbed in
\(2^{O(D_*^2)}\).  Indeed,
\[
 \rank(\Omega_2|_L)
 \ge\rank\Omega-4T-2u\ge2v+2,
\]
so Lemma~\ref{lem:isotropic} applies uniformly.

Third, the compatible-pair weight uses \(k=u+v-s\), rather than \(u+v\).
The resulting ratio of weights is exactly
\[
  2^{\,ns-s(u+v)+\binom{s+1}{2}}.
\]
Multiplying the three factors gives, apart from \(O(D_*^2)\),
\[
 -t\gamma+gv-s(n-2u+g-T-v)
   +ns-s(u+v)+\binom{s+1}{2}
 =-t\gamma+gv+s(u-g+T)+\binom{s+1}{2}.
\]
This cancellation of every term linear in \(n\) is the point of the
compatible-pair weight.  Since \(t\ge1\) and \(g,s,u,v\le D+1\), the remaining
terms contribute \(O(D_*^2)\), leaving \(2^{-\gamma+O(D_*^2)}\).  Restoring the
two baseline isotropic counts, the disjoint-span weight, and the
\(|\GL_u||\GL_v|\) normalization multiplies this relative bound by
\(m_u(\Omega)m_v(\Omega)\).  Summing over \(t,g,s\) proves
\eqref{eq:bad-subspace-first}.

\emph{The disjoint-span main term.}
Start with the \(\Omega_1\)-isotropic first tuples.  Lemma~\ref{lem:isotropic},
together with \textup{(P2)}, gives their target count and bounds the tuples
whose spans meet \(\rad(\Omega_1)\) by
\(2^{CD_*^2-\rank\Omega}\) on the relative scale.  From the remaining
radical-free tuples, discard those whose spans meet
\[
 K_\Sigma,\qquad K_{12},\qquad\ker\psi_1.
\]
Lemma~\ref{lem:iso-deficit}, summed over positive intersection dimensions,
bounds the three discarded proportions by
\[
  2^{CD_*^2-h},\qquad
  2^{CD_*^2-\rank B_{12}},\qquad
  2^{CD_*^2-n},
\]
respectively.  Here
\(\operatorname{codim}\ker\psi_1\ge n-T\), and the losses \(T+b\) are absorbed
in \(CD_*^2\).  Although \(\ker\psi_1\subseteq K_\Sigma\), we record its stronger
codimension separately because it contributes to the \(2^{-n}\) term.  Hence
the number of remaining first tuples is
\begin{equation}\label{eq:generic-first-count}
 |\GL_u|m_u(\Omega)2^{nu-\binom u2}
 \left(1+O(2^{CD_*^2}\varepsilon_*)\right).
\end{equation}

For these tuples the cross map has rank exactly $u$, so the second tuples lie in a
subspace $L\subseteq V_2$ of dimension $n-u$.  The restriction of
$\Omega_2$ to $L$ loses rank by at most $2u$.  More explicitly,
\[
\rank(\Omega_2|_L)
 \ge\rank\Omega-4T-2u
 \ge2v+2,
\]
where the last inequality follows from
$\rank\Omega\ge4D+8T+10$ and $u,v\le D+1$.  Thus
Lemma~\ref{lem:isotropic} applies to its isotropic $v$-tuples.

Discard the second tuples whose spans meet
\(\rad(\Omega_2|_L)\), \(\ker\psi_2\), or
\(L\cap\psi_2^{-1}(\mathcal X)\).  The first discarded proportion is
\(2^{CD_*^2-\rank\Omega}\) by Lemma~\ref{lem:isotropic}.  The latter two
subspaces have codimension in \(L\) at least \(n-u-T\) and \(n-2u-T\),
respectively, so Lemma~\ref{lem:iso-deficit} bounds both discarded
proportions by \(2^{CD_*^2-n}\).  Every remaining second tuple has injective
projection and \(\mathcal X\cap\mathcal Y=0\); property~\textup{(R2)} now
makes \(\mathcal R_{u,v}\) automatic.  The number of such tuples is
\begin{equation}\label{eq:generic-second-count}
 |\GL_v|m_v(\Omega)2^{(n-u)v-\binom v2}
 \left(1+O(2^{CD_*^2}\varepsilon_*)\right).
\end{equation}
The bounded rank losses guarantee that all restricted alternating forms are
nonzero, so their target moments equal $1/|\GL_j|$.

Every tuple pair retained above has disjoint projected spans, and therefore
\(k=u+v\).  The exponent identity
\[
 \left(nu-\binom u2\right)
 +\left((n-u)v-\binom v2\right)
 -\left(n(u+v)-\binom{u+v}2\right)=0
\]
shows that~\eqref{eq:generic-first-count} and
\eqref{eq:generic-second-count}, after division by
$|\GL_u||\GL_v|$, make the total contribution of these disjoint-span pairs
\begin{equation}\label{eq:gen-main}
 m_u(\Omega)m_v(\Omega)
 +O\!\left(2^{CD_*^2}\theta_u\theta_v\varepsilon_*\right).
\end{equation}
We used here
$m_j(\Omega)/\theta_j=2^{-\binom j2}/\eta_j(2)$, so the relative radical
error in Lemma~\ref{lem:isotropic} is still bounded by the displayed
quadratic factor on the $\theta_u\theta_v$ scale.

\emph{Compatible pairs outside the main term.}
By~\eqref{eq:bad-subspace-first}, compatible pairs whose first span meets
\(K_\Sigma\) or \(K_{12}\) contribute at most
\[
 O\!\left(2^{CD_*^2}\theta_u\theta_v
   (2^{-h}+2^{-\rank B_{12}})\right).
\]
Pairs whose first span meets \(\rad(\Omega_1)\), or whose second span meets the
radical of the relevant restriction of \(\Omega_2\), contribute
\[
 O\!\left(2^{CD_*^2-\rank\Omega}\theta_u\theta_v\right)
\]
by Lemma~\ref{lem:isotropic}; the possible rank loss \(2u+4T\) is absorbed in
\(CD_*^2\).

These loci exhaust the compatible pairs outside the disjoint-span contribution
in~\eqref{eq:gen-main}.  Indeed, compatibility already imposes
injectivity of both projections.  If a compatible pair avoids the radical
loci and \(K_{12}\), then the cross-solution space used above is
defined; if it also avoids \(K_\Sigma\), equation~\eqref{eq:s-le-g} gives
\(\mathcal X\cap\mathcal Y=0\), so the pair belongs to the disjoint-span main
term.  Thus the projection kernels create the \(2^{-n}\) error only through
the tuples omitted from the product count; they are not an additional class
of compatible pairs.

Adding the compatible exceptional contribution to~\eqref{eq:gen-main} proves
\eqref{eq:weighted-sum} after enlarging the absolute constant $c_*$.
\end{proof}

\begin{proof}[Proof of Proposition~\ref{thm:mix}]
Let $\bbeta$ lie in the good event with the threshold $H_D$ of the
proposition.  Proposition~\ref{prop:mb-reduce} identifies the conditional
mixed moment with the compatible-pair sum~\eqref{eq:compatible-pair-sum}.
We apply Lemma~\ref{lem:weighted-sum} with \(T=b_2\) and the cross form
\(B_{12}\) of Proposition~\ref{prop:mb-reduce}.  Indeed,
\[
  \rank B_{12}\ge\rank\Gamma-2b
  \ge4D+2b_2+10,
\]
because $b\le b_2+1$ and $b_2\ge1$.  The support hypothesis follows from
$h\ge H_D$, while the assumed lower bound for $\rank\Omega$ supplies the
nonzero-defect hypothesis.  Hence~\eqref{eq:weighted-sum} applies.  Since
$b_2$ and $b$ are bounded in terms of $d_0$, its quadratic
prefactor is at most
$2^{c_{\mathrm{mix}}(D+b_2+b+1)^2}$ after enlarging
$c_{\mathrm{mix}}$.  Finally,
$2^{-\rank B_{12}}\le2^{2b}2^{-\rank\Gamma}$, and the fixed factor is absorbed
in the same prefactor.  This is~\eqref{eq:mix}.
\end{proof}

We can now verify the three hypotheses of case~\textup{(II)} of
Theorem~\ref{thm:criterion} for the bordered pair, in the order in which that
case uses them.

\begin{proof}[Proof of Theorem~\ref{thm:rmt0}]
Fix a sign vector with $|\bu_{-1}|\ge\alpha r$, and put $n=r-1$.
By~\eqref{eq:sdelta},
\[
  \rank\Omega\ge |\bu_{-1}^\circ|-1\ge\alpha r-2.
\]
After enlarging the final constant, we may assume that $r$ exceeds a constant
depending only on $(\alpha,d_0)$.  Then $\Omega\ne0$ and
Lemma~\ref{lem:1d-moment}, together with the inversion computed in the proof
of Theorem~\ref{thm:1d}, gives
\[
  m_u(\Omega)=\frac1{|\GL_u|}\ll2^{-u^2},
  \qquad
  \mathcal G^{-1}(m_u(\Omega))_{u\ge0}=\piCL.
\]
Thus the target array $m_{u,v}=m_u(\Omega)m_v(\Omega)$ factors and decays as
required in case~\textup{(II)}, and its inverse is $\piCL\otimes\piCL$.

Let $b_{\max}=\omega(d_0)+1$.  Choose $\beta_0>0$ sufficiently small and,
for $0\le D\le\beta_0\sqrt n$, put
\[
  \mathcal E_{n,D}:=\mathcal E_{\mathrm{good}}(H_D),
  \qquad
  \mathrm{Err}_{n,D}:=
  2^{-h}+2^{-\rank\Gamma}+2^{-n}+2^{-\rank\Omega},
\]
with $H_D$ as in Proposition~\ref{thm:mix}.  Both are measurable with respect
to the border data $\bbeta$.  For all sufficiently large $n$, uniformly in
$b_2\le b_{\max}$,
\[
  H_D\le\alpha n/4,
  \qquad
  \rank\Omega\ge4D+8b_2+10.
\]
Proposition~\ref{thm:mix} therefore supplies
part~\textup{(II.a)} of Theorem~\ref{thm:criterion}, with a prefactor of the
required form $2^{c_{\mathrm{loss}}(D+1)^2}$, for a constant depending only
on \(d_0\) and \(\alpha\).

For part~\textup{(II.b)}, since \(h\sim\operatorname{Bin}(n,\tfrac12)\),
\[
  \Exp[2^{-h}]=\left(\frac34\right)^n.
\]
Moreover,
\[
  \rank\Gamma\ge
  \#\{i\le n:u_{d_0,i}\ne u_{-1,i}\}-1,
\]
where the count is binomial with parameters $(n,\tfrac12)$; hence
\[
  \Exp[2^{-\rank\Gamma}]
  \le2\left(\frac34\right)^n.
\]
Finally $2^{-n}\le(3/4)^n$ and
$2^{-\rank\Omega}\le4\,2^{-\alpha n}$.  Thus
\[
  \Exp[\mathrm{Err}_{n,D}\ind_{\mathcal E_{n,D}}]
  \le C\qerr^n,
  \qquad \qerr:=\max\{3/4,2^{-\alpha}\}<1,
\]
uniformly for $D\le\beta_0\sqrt n$.  Part~\textup{(II.c)} is the bound
$\Prob(\mathcal E_{\mathrm{good}}(H_D)^c)\le e^{-\delta n}$ of
Lemma~\ref{lem:bad-event}, which applies because $H_D\le n/4$.
Case~\textup{(II)} of Theorem~\ref{thm:criterion} now yields
\[
  \bigl\|\Law(Z_1,Z_2)-\piCL\otimes\piCL\bigr\|_1
  \le C\eta^n
\]
for some $\eta<1$.  The constants are uniform in the finitely many strata, in
$b_2\le b_{\max}$, and in the sign vector over $|\bu_{-1}|\ge\alpha r$.
Since $n=r-1$, this proves Theorem~\ref{thm:rmt0}.
\end{proof}

\subsection{Unbordered zero- and high-rank-defect cases}\label{ssec:unbordered}

The zero-defect case has the symmetric rather than the
Cohen--Lenstra--Gerth limiting distribution and admits a direct mixed-moment calculation.
A high-rank symmetry defect instead uses Lemma~\ref{lem:weighted-sum}.

\begin{lemma}
\label{lem:symmetric-diag-moment}
Let $A\sim\Unif(\Sym_n(0))$, fix $\bv\in\BF_2^n$, put
$S=\supp\bv$ and $h=|S|$, and set
\[
  X=\corank A,\qquad Y=\corank(A+\diag\bv).
\]
There is an absolute $C>1$ such that, for $0\le u,v\le D+1$ and
$h\ge4D+8$,
\begin{equation}\label{eq:symmetric-diag-moment}
 \left|\Exp_A\left[\gtwo Xu\gtwo Yv\mid\bv\right]
       -m_u(0)m_v(0)\right|
 \le C^D\theta_u\theta_v\left(2^{3D-h}+2^{3D-n}\right).
\end{equation}
\end{lemma}

\begin{proof}
Expand the mixed moment over ordered independent kernel tuples and apply
Lemma~\ref{lem:joint-surj} to a basis of their joint span.  If
$\mathcal X$ and $\mathcal Y$ are the two spans, compatibility says
\[
  x_S^{\top}y_S=0\quad(x\in\mathcal X,\ y\in\mathcal Y),
  \qquad
  \mathcal X\cap\mathcal Y\subseteq V_{S^c}.
\]
Put
$t=\dim(\mathcal X\cap V_{S^c})$ and
$s=\dim(\mathcal X\cap\mathcal Y)$, so $0\le s\le t$.

For fixed \((s,t)\), Lemma~\ref{lem:rank-deficit} gives at most
\[
 \gtwo ut\,2^{nu-th}
\]
ordered first-family tuples.  The cross equations cut the second ambient space
by \(u-t\) dimensions, leaving a space of dimension \(n-u+t\).  The common
span is contained in the \(t\)-dimensional space
\(\mathcal X\cap V_{S^c}\); choosing its \(s\)-dimensional position and applying
Lemma~\ref{lem:fallin} bounds the second-family tuples by
\[
 C^D\,2^{s(t-s)}\,2^{(n-u)v}\,2^{tv}\,
       2^{-s(n-u-v)}.
\]
Finally, the joint span has dimension \(u+v-s\), and
Lemma~\ref{lem:joint-surj} assigns the probability
\[
 2^{-n(u+v-s)+\binom{u+v-s}{2}}.
\]

For \(s=t=0\), the exact products in the first two counts give, before
division by \(|\GL_u||\GL_v|\),
\[
 2^{nu}\,2^{(n-u)v}\,
 2^{-n(u+v)+\binom{u+v}{2}}
 =2^{\binom u2+\binom v2}.
\]
This is \(m_u(0)m_v(0)\) after normalization.  The omitted factors in the
exact independent-tuple products contribute
\(O(C^D2^{3D-h})+O(C^D2^{3D-n})\) relative to this main term.

For \(t\ge1\), use
\(\gtwo ut\le\eta_\infty(2)^{-1}2^{t(u-t)}\) and
\[
 \binom{u+v-s}{2}-\binom{u+v}{2}
 =-s(u+v)+\binom{s+1}{2}.
\]
After cancellation of every term involving \(n\), the logarithm to base \(2\)
of the contribution relative to the case \(s=t=0\) is at most
\[
 -t(h-u-v)-t^2+s(t-s)+\binom{s+1}2
 \le-t(h-u-v).
\]
The last inequality uses \(s\le t\), since
\(s(t-s)+\binom{s+1}{2}\le\binom{t+1}{2}\le t^2\).
Summing over
$t\ge1$ and $0\le s\le t$, and using $u,v\le D+1$ and $h\ge4D+8$, gives the
first error in~\eqref{eq:symmetric-diag-moment}; the possible fall-ins caused
by the finite joint span give the second.  This proves the lemma.
\end{proof}

For a nonzero defect produced by quadratic reciprocity, the required rank
tails can be checked directly.  Combining those estimates with
Lemma~\ref{lem:weighted-sum} gives the following unconditional consequence.

\begin{theorem}
\label{thm:diag}
Let $\bu_{-1}\in\BF_2^n$, put
\[
  \Omega:=\bu_{-1}\bu_{-1}^{\top}+\diag(\bu_{-1}),
\]
and take $A\sim\Unif(\Sym_n(\Omega))$ and
$\bv\sim\Unif(\BF_2^n)$ independently.  Set
\[
  X:=\corank A,\qquad Y:=\corank\bigl(A+\diag(\bv)\bigr).
\]
Then the following estimates hold.
\begin{enumerate}[\upshape(1)]
\item If $\bu_{-1}=\bzero$, there are $C_0<\infty$ and $\eta_0\in(0,1)$ such
that
\[
  \bigl\|\Law(X,Y)-\pisym\otimes\pisym\bigr\|_1
  \le C_0\eta_0^n.
\]
\item For every $\alpha>0$, there are
$C_\alpha<\infty$ and $\eta_\alpha\in(0,1)$ such that, uniformly over
$\bu_{-1}$ satisfying $\rank\Omega\ge\alpha n$,
\[
  \bigl\|\Law(X,Y)-\piCL\otimes\piCL\bigr\|_1
  \le C_\alpha\eta_\alpha^n.
\]
\end{enumerate}
More generally, the same proof applies to a sequence
$\Omega_n\in\Alt_n$ for which either $\Omega_n=0$ for all $n$, or
$\rank\Omega_n\ge\alpha n$, provided that
$\bv_n\sim\Unif(\BF_2^n)$ remains independent of the core and there are
$c,\delta>0$ such that
\[
 \Prob\bigl(\rank(\diag(\bv_n)+\Omega_n)<cn\bigr)\le e^{-\delta n}.
\]
\end{theorem}

\begin{proof}
Condition on $\bv$, put $S=\supp\bv$, $h=|S|$, and
$\Gamma=\diag(\bv)+\Omega$.

Suppose first that $\bu_{-1}=0$.  For $0\le D\le\beta_0n$, with $\beta_0>0$
sufficiently small, take
\[
  \mathcal E_{n,D}=\{h\ge4D+8\}.
\]
Lemma~\ref{lem:symmetric-diag-moment} supplies
part~\textup{(I.b)} of Theorem~\ref{thm:criterion}, with
$\mathrm{Err}_{n,D}=2^{-h}+2^{-n}$.  The conditional marginal bounds follow from
Lemma~\ref{lem:1d-moment}, since diagonal translation preserves the symmetric
fibre.  Now $h\sim\operatorname{Bin}(n,\tfrac12)$,
$\Exp[2^{-h}]=(3/4)^n$, and Hoeffding's inequality gives the bad-event bound.
Case~\textup{(I)} of Theorem~\ref{thm:criterion} therefore gives
\[
  \bigl\|\Law(X,Y)-\pisym\otimes\pisym\bigr\|_1
  \le C_0\eta_0^n
\]
for some $C_0<\infty$ and $\eta_0\in(0,1)$.

Suppose next that $\rank\Omega\ge\alpha n$.  In
Lemma~\ref{lem:weighted-sum} take
\[
 V_1=V_2=\BF_2^n,\qquad \psi_1=\psi_2=\operatorname{id},\qquad
 \Sigma=0,\qquad T=b=0,
\]
with $\Omega_1=\Omega_2=\Omega$ and cross form
\[
 B_{12}(\bx,\by):=\bx^{\top}\Gamma\by,
\]
and consistency predicate
\(\mathcal R_{u,v}:\mathcal X\cap\mathcal Y\subseteq V_{S^c}\).
Indeed, subtracting the two kernel equations on a common vector gives
\(\diag(\bv)\bz=0\), which is precisely this condition.  Thus
\textup{(R1)} and \textup{(R2)} of the weighted-sum lemma hold.  For
$0\le D\le\beta_0\sqrt n$, take
\[
 \mathcal E_{n,D}=
 \{h\ge4D+10\}\cap\{\rank\Gamma\ge4D+10\},
\]
where $\beta_0$ is fixed and $n$ is sufficiently large.  The weighted-sum
lemma supplies part~\textup{(II.a)} of Theorem~\ref{thm:criterion}, with
\[
 \mathrm{Err}_{n,D}=2^{-h}+2^{-\rank\Gamma}+2^{-n}
             +2^{-\rank\Omega}.
\]
The variable $h$ is $\operatorname{Bin}(n,1/2)$.  After a simultaneous
coordinate permutation write $\bu_{-1}=(\bone_\kappa,\bzero)$.  If
$h_1=|\supp\bv_1|$ and $h_0=|\supp\bv_0|$, then
\[
 \rank\Gamma\ge(\kappa-h_1)+h_0-1,
\]
and the right side plus one has the distribution
$\operatorname{Bin}(n,1/2)$.  Hoeffding's inequality gives the bad-event bound,
and
\[
 \Exp[2^{-h}]=\left(\frac34\right)^n,
 \qquad
 \Exp[2^{-\rank\Gamma}]\le2\left(\frac34\right)^n.
\]
Together with $2^{-\rank\Omega}\le2^{-\alpha n}$, this verifies the
remaining hypotheses of the quadratic-loss criterion.  The target moments
are $1/|\GL_u|$, so the limit is $\piCL\otimes\piCL$.

For the final generalization, the assumed lower-tail bound gives
\[
 \Exp\bigl[2^{-\rank(\diag(\bv_n)+\Omega_n)}\bigr]
 \le e^{-\delta n}+2^{-cn}.
\]
Together with the binomial estimates for $|\supp\bv_n|$ and, in the nonzero
case, the assumed linear lower bound for $\rank\Omega_n$, this supplies
exactly the averaged error and bad-event bounds used above.
\end{proof}

\begin{remark}[A fixed identity shift]
\label{rem:identity-perturbation}
Let $A_n$ be uniform in $\Sym_n(0)$.  Taking
$\bv=\bone$ in Lemma~\ref{lem:symmetric-diag-moment} and applying
Case~\textup{(I)} of Theorem~\ref{thm:criterion} gives constants $C<\infty$ and
$\eta\in(0,1)$ such that
\[
 \left\|
 \Law\bigl(\corank A_n,\corank(A_n+I_n)\bigr)
 -\pisym\otimes\pisym
 \right\|_1
 \le C\eta^n.
\]
Thus the fixed shift \(I_n\) already makes the two coranks asymptotically
independent.
\end{remark}

\begin{remark}[Extremal sign vectors, prescribed characters and borders]
\label{rem:degenerate-rmt}
A condition $p\equiv1\bmod4$ forces $\bu_{-1}=0$, so
Theorem~\ref{thm:diag} determines the unbordered core pair.  More general
character restrictions can place the vectors
\((\aled2p,\aled{\Delta_1}p,\dots,\aled{\Delta_s}p)\) in proper affine subspaces, in
which case \(\bu_{d_0}\) need not be uniform and the diagonal
perturbation may have bounded support.

Theorem~\ref{thm:rmt0} includes the all-ones sign vector
$\bu_{-1}=\bone$: its defect is $J_n+I_n$, of rank at least $n-1$, and the
two bordered coranks are asymptotically independent with common marginal distribution
$\piCL$.  The symmetric case is
different: its unbordered limit is $\pisym\otimes\pisym$, but the arithmetic
matrices also contain the fixed borders of Theorem~\ref{thm:block}.  When the
symmetry defect of the core is zero, the augmented alternating forms created
by those borders have bounded rank, so tuples meeting their radicals need not
be negligible.  Even in width one, the corner can change one marginal away
from \(\pisym\).  For arbitrary fixed width the
limiting mixed moments have not been computed here.  Thus the
full squarefree family uses the range $|\bu_{-1}|\ge\alpha r$ of
Theorem~\ref{thm:box0}; bordered sign vectors with
\(|\bu_{-1}|/r\to0\), including \(\bu_{-1}=0\), require separate analysis.
\end{remark}

\section{Proofs of the main results}\label{sec:proof}

\subsection{Asymptotic independence of the two
  \texorpdfstring{$4$}{4}-ranks}\label{ssec:assembly}

\begin{proof}[Proof of Theorem~\ref{thm:B}]
Fix once and for all $\alpha=1/4$.  By
Theorem~\ref{thm:box0}, with $\rho=\piCL\otimes\piCL$, it suffices to estimate
the distribution of the bordered corank pair uniformly for
\[
 |r-\log\log X|<(\log\log X)^{2/3},
 \qquad |\bu_{-1}|\ge\alpha r.
\]
Theorem~\ref{thm:rmt0}, uniformly in $\bu_{-1}$ (and hence in the stratum it
determines),
gives constants $C_{d_0}<\infty$ and $\vartheta\in(0,1)$ such that
\begin{equation}\label{eq:assembly-parent-rmt}
  \bigl\|\Law_{\MatData_r(\bu_{-1})}(\mathbf Z)
  -\piCL\otimes\piCL\bigr\|_1
  \le C_{d_0}\vartheta^{r}
  \ll_{d_0}\vartheta^{(\log\log X)/2}
\end{equation}
for all sufficiently large $X$.

Substituting \eqref{eq:assembly-parent-rmt} into
\eqref{eq:family-transfer-bd} gives
\[
  \bigl\|\mu_X^{(d_0)}-\piCL\otimes\piCL\bigr\|_1
  \ll_{d_0}(\log\log X)^{-c_{\mathrm{tr}}}
       +\vartheta^{(\log\log X)/2}.
\]
Total variation is one half of the $\ell^1$ distance, so shrinking the
exponent gives \eqref{eq:main-TV} for some $c=c(d_0)>0$.
\end{proof}

\subsection{The biquadratic application}\label{ssec:biquad}

Let $K:=\BQ(\sqrt{d_0},\sqrt{-d})$, a biquadratic field containing the three
quadratic subfields
\[
  k_0=\BQ(\sqrt{d_0})\ (\text{real}),\qquad k_1=\BQ(\sqrt{-d}),\qquad
  k_2=\BQ(\sqrt{-d_0d}).
\]
The dependence of \(k_1,k_2\) and \(K\) on \(d\) is suppressed throughout
this subsection.
For $k_0$, and only there, $\Cl_{k_0}$ denotes the ordinary class group; write
$h_{k_0}$ for its order.

All density statements below are relative to the
family~\eqref{eq:F0}.  For quantities $A_d$ and $B_d$ on this family, write
\[
 A_d\peq B_d
 \quad\Longleftrightarrow\quad
 \Prob_{d\in\mathcal F_{d_0}(X)}(A_d=B_d)\longrightarrow1
 \quad(X\to\infty).
\]
The superscript $\mathbb P$ has the same meaning over any other relation; in
particular, $G_d\pcong H_d$ means that the displayed isomorphism exists with
probability tending to one.  Pointwise identities retain the ordinary equality
sign.
Since \(\#\mathcal F_{d_0}(X)\asymp_{d_0}X\), any exceptional set of
squarefree parameters of size \(o(X)\) remains negligible after restriction
to this family.  We use this observation for the external density-one inputs
below.

\subsubsection{Class-group maps and class-number formulas}
\label{ssec:kuroda}

Let
\[
  j:\prod_{i=0}^{2}\Cl_{k_i}\longrightarrow\Cl_K,
  \qquad
  N:\Cl_K\longrightarrow\prod_{i=0}^{2}\Cl_{k_i}
\]
be extension of ideals and the product of the three norm maps, respectively.
Extension and norm of ideals give
$N\circ j=2\operatorname{id}$ and $j\circ N=2\operatorname{id}$.
Kuroda's class number formula gives
\begin{equation}\label{eq:kuroda-class-number}
  h_K=\frac12q_Kh_{k_0}h_{k_1}h_{k_2},
\end{equation}
where
\[
 q_K=[\CO_K^\times:
       \CO_{k_0}^\times\CO_{k_1}^\times\CO_{k_2}^\times]
\]
is the Hasse unit index.
This is the $k=\BQ$, imaginary case printed immediately after
\cite[equation~(1.5)]{Lemmermeyer}.  We also use the ambiguous class
number formula for a cyclic extension $L/k$ with Galois group $G$:
\begin{equation}\label{eq:ambiguous-class-number}
  |\Cl_L^G|
  =\frac{|\Cl_k|\prod_{v\le\infty}e_v}
  {[L:k]\,[\CO_k^\times:
      \CO_k^\times\cap N_{L/k}(L^\times)]}.
\end{equation}

\begin{lemma}\label{lem:artin-divisors}
Let $M/k$ be a fixed abelian extension and let $m\ge1$.  For a density-one set
of squarefree integers $d$, every element of $\Gal(M/k)$ occurs as
$\operatorname{Art}_{M/k}(\mathfrak p)$ for at least $m$ distinct primes
$\mathfrak p\mid d$ of $k$.
\end{lemma}

This is \cite[Corollary~2.3]{KMS}.

\subsubsection{The Hasse unit index}
\label{ssec:unit-index}

The field $K$ is totally imaginary with unique real quadratic subfield $k_0$.
Write $E_F:=\units F$, let $\varepsilon_0$ be the fundamental unit of $k_0$,
and let $\mu_K$ be the group of roots of unity of $K$.

The totally imaginary signature makes the unit index particularly simple.
First discard the finitely many \(d\) for which \(\mu_K\ne\{\pm1\}\).
Indeed, extra roots of unity force \(K\) to contain
\(\BQ(i)\) or \(\BQ(\sqrt{-3})\); since the only imaginary quadratic
subfields of \(K\) are \(k_1\) and \(k_2\), this determines \(d\) up to
finitely many possibilities.  This omission does not affect any density-one
statement.  For the remaining \(d\), the unit group \(E_K\) has
\(\BZ\)-rank one and
\[
 E_{k_1}=E_{k_2}=\mu_K=\{\pm1\},
 \qquad
 E_{k_0}E_{k_1}E_{k_2}=\langle\mu_K,\varepsilon_0\rangle.
\]
For every
\(\varepsilon\in E_K\),
\[
  \prod_{i=0}^{2}N_{K/k_i}(\varepsilon)
  =N_{K/\BQ}(\varepsilon)\varepsilon^2,
\]
and therefore
$E_K^{2}\subseteq E_{k_0}E_{k_1}E_{k_2}$.  Thus
\[
  E_K=\langle\mu_K,\varepsilon_0\rangle
  \quad\text{or}\quad
  E_K=\langle\mu_K,\sqrt{u\varepsilon_0}\rangle
  \quad\text{for some }u\in\mu_K,
\]
so $q_K\in\{1,2\}$, with $q_K=2$ precisely in the second case.

\begin{theorem}\label{thm:qK}
Let $d_0>1$ be squarefree.  Then
\[
  q_K\peq1.
\]
\end{theorem}

\begin{proof}
Suppose \(q_K=2\).  Then, for some \(u\in\mu_K\),
\(\eta:=\sqrt{u\varepsilon_0}\) belongs to \(E_K\setminus k_0\).  Hence
$K=k_0\bigl(\sqrt{u\varepsilon_0}\bigr)$, and since
$K=k_0(\sqrt{-d})$ we get $-d\equiv u\varepsilon_0$ modulo $(k_0^{\times})^{2}$.
Complex conjugation sends \(u\) to \(u^{-1}\), and hence
\[
  N_{K/k_0}(\eta)^2=N_{K/k_0}(u\varepsilon_0)=\varepsilon_0^2.
\]
Therefore \(N_{K/k_0}(\eta)=\pm\varepsilon_0\), so either
\(\varepsilon_0\) or \(-\varepsilon_0\) is a norm from \(K\).

Now $K/k_0$ is the quadratic extension cut out by $-d$, so a prime
$\mathfrak p$ of $k_0$ with $\mathfrak p\mid d$ is ramified in $K/k_0$, and for
an element of $k_0^{\times}$ that is a local norm everywhere one has, at every
$\mathfrak p\nmid2\infty$ unramified in $k_0(\sqrt{\pm\varepsilon_0})/k_0$,
\[
  \Bigl(\frac{\pm\varepsilon_0}{\mathfrak p}\Bigr)=1,
  \qquad\text{i.e.}\qquad
  \mathrm{Art}_{k_0(\sqrt{\pm\varepsilon_0})/k_0}(\mathfrak p)=1.
\]
Apply Lemma~\ref{lem:artin-divisors} with $k=k_0$, with the fixed abelian
extension
$M=k_0\bigl(\sqrt{\varepsilon_0}\bigr)
\cdot k_0\bigl(\sqrt{-\varepsilon_0}\bigr)$, which is independent of $d$,
and with \(m=1\).  For a density-one set of \(d\), every element of
\(\Gal(M/k_0)\) occurs at a prime \(\mathfrak p\mid d\).  Choosing one that is
nontrivial on both quadratic subextensions shows that neither
\(\varepsilon_0\) nor \(-\varepsilon_0\) is a norm from \(K\), contradicting
\(q_K=2\).
\end{proof}

\subsubsection{The direct-sum defect}
\label{ssec:2rank}

Since $j\circ N=2\cdot\mathrm{id}$ on $\Cl_K$, one has
$2\Cl_K\subseteq\operatorname{im}j$.  Multiplying once more by $2$ and passing
to primary parts gives
\[
  4\Cl_K[2^\infty]
  \subseteq j\Bigl(\prod_i2\Cl_{k_i}[2^\infty]\Bigr),
  \qquad
  \bigl|4\Cl_K[2^\infty]\bigr|
  \le\prod_i\bigl|2\Cl_{k_i}[2^\infty]\bigr|.
\]
Indeed, let
$x\in4\Cl_K[2^\infty]$ and write $x=j(y)$ with $y\in\prod_i2\Cl_{k_i}$. Split
$y=y_2+y'$ into its $2$-primary and odd parts. Then $j(y_2)$ is $2$-primary and
$j(y')$ has odd order, and $x=j(y_2)+j(y')$ is $2$-primary, so $j(y')=0$ and
$x=j(y_2)$ with $y_2\in\prod_i2\Cl_{k_i}[2^\infty]$, the splitting into primary
parts being compatible with multiplication by $2$.

\begin{definition}\label{def:lambdaK}
The preceding inclusion allows us to define the \emph{direct-sum defect} by
\[
  \lambda_K:=\log_2\frac{\prod_{i=0}^{2}\bigl|2\Cl_{k_i}[2^\infty]\bigr|}
                        {\bigl|4\Cl_K[2^\infty]\bigr|}
  \ =\ \sum_{i=0}^{2}\sum_{n\ge2}\rk_{2^{n}}\Cl_{k_i}
       -\sum_{n\ge3}\rk_{2^{n}}\Cl_K\ \ge\ 0 .
\]
\end{definition}

Abbreviate $S:=\prod_{i=0}^{2}2\Cl_{k_i}[2^\infty]$ and $Q:=4\Cl_K[2^\infty]$,
both finite.  We have $Q\subseteq j(S)$ and trivially
$|j(S)|\le|S|$, so
\[
  |Q|\ \le\ |j(S)|\ \le\ |S| .
\]
Thus $\lambda_K=\log_2(|S|/|Q|)\ge0$, and $\lambda_K=0$ holds if and only if
both inequalities are equalities.  Equality on the right is equivalent to the
injectivity of $j|_S$, while equality on the left, given $Q\subseteq j(S)$,
is equivalent to $Q=j(S)$.  Hence $\lambda_K=0$ if and only if
$j$ restricts to an isomorphism $j|_S:S\to Q$.

It remains to compute $\lambda_K$.  The ambiguous class number formula gives
the \(2\)-rank of $\Cl_K$; combining it with the \(2\)-part of Kuroda's formula
relates $\lambda_K$ to the \(4\)-rank of $\Cl_K$.  The density-one formula of
Koymans--Morgan--Smit evaluates that \(4\)-rank independently and leaves only
dyadic data.

Let $\Delta_{d_0}:=\Delta(d_0)$ be the discriminant of $k_0$.  Thus, for
$p\nmid2d_0$, the additive symbol $\asm{\Delta_{d_0}}{p}$ vanishes exactly
when $p$ splits in $k_0/\BQ$.

\begin{proposition}\label{prop:rk2}
Let $d_0>1$ be squarefree with $h_{k_0}$ odd, and let
\[
  \omega_{+}(d):=\#\{p\mid d:\ \asm{\Delta_{d_0}}{p}=0\},\qquad
  \omega_{-}(d):=\#\{p\mid d:\ \asm{\Delta_{d_0}}{p}=1\}
\]
be the numbers of prime divisors of $d$ that split, respectively remain inert, in
$k_0/\BQ$, and let $\delta=\delta(d)\in\{0,1,2\}$ be the number of primes of $k_0$
above $2$ that ramify in $K/k_0$.  Then
\begin{equation}\label{eq:rk2}
  \rk_2\Cl_K\ \peq\ 2\omega_{+}(d)+\omega_{-}(d)+\delta-1 .
\end{equation}
\end{proposition}

\begin{proof}
Write $G=\Gal(K/k_0)=\langle\sigma\rangle$ and
$A=\Cl_K[2^\infty]$.  For the quadratic extension $K/k_0$ one has
$\iota_{k_0\to K}\circ N_{K/k_0}=1+\sigma$ on $\Cl_K$, where $\iota$ denotes
extension of ideals.  Hence
$(1+\sigma)\Cl_K\subseteq\iota(\Cl_{k_0})$.  The latter group has odd order,
whereas $A$ is a $2$-group, and therefore
\[
  (1+\sigma)A=0,\qquad
  \bigl(\Cl_K^{G}\bigr)[2^\infty]=A^{G}=A[2]=\Cl_K[2].
\]
In particular,
$|(\Cl_K^{G})[2^\infty]|=2^{\rk_2\Cl_K}$.

For the unit norm index, the four classes in
$E_{k_0}/E_{k_0}^{2}$ are represented by
$1,-1,\varepsilon_0,-\varepsilon_0$.  Thus
\[
 [E_{k_0}:E_{k_0}\cap N_{K/k_0}(K^\times)]\le4,
\]
with strict inequality only if one of the three nontrivial representatives
$u_0$ is a norm.  In that case $u_0$ is a local norm at every ramified prime
$\mathfrak p\mid d$, and hence
\[
 \mathrm{Art}_{k_0(\sqrt{u_0})/k_0}(\mathfrak p)=1
 \qquad(\mathfrak p\mid d).
\]
Apply Lemma~\ref{lem:artin-divisors} to the fixed abelian extension
\[
 M=k_0(\sqrt{-1})\,k_0(\sqrt{\varepsilon_0})\,
     k_0(\sqrt{-\varepsilon_0}).
\]
For each \(u_0\in\{-1,\varepsilon_0,-\varepsilon_0\}\), choose
\(\sigma_{u_0}\in\Gal(M/k_0)\) whose restriction to
\(k_0(\sqrt{u_0})\) is nontrivial.  For a density-one set of \(d\), the lemma
supplies a prime \(\mathfrak p_{u_0}\mid d\), possibly depending on \(u_0\),
with Artin symbol \(\sigma_{u_0}\).  Thus no nontrivial unit class is a norm,
and consequently
\begin{equation}\label{eq:unit-norm-index}
 [E_{k_0}:E_{k_0}\cap N_{K/k_0}(K^\times)]\peq4.
\end{equation}

Apply~\eqref{eq:ambiguous-class-number} to $K/k_0$ and count the ramified
places.

\emph{Infinite places.} $k_0$ is real quadratic, so it has two real places; $K$
is totally imaginary, so both become complex, and each contributes $e_v=2$: a
factor $2^{2}$.

\emph{Primes dividing $d$.} Let $p\mid d$. Since $\gcd(d,2d_0)=1$, $p$ is
unramified in $k_0/\BQ$, and it splits or is inert according as
$\asm{\Delta_{d_0}}{p}$ is $0$ or $1$; the number of primes of $k_0$ above $p$ is
therefore $2$ or $1$. Each such prime $\mathfrak p$ satisfies
$v_{\mathfrak p}(-d)=1$ because $d$ is squarefree and $p$ is unramified in $k_0$,
so $\mathfrak p$ ramifies in $K=k_0(\sqrt{-d})$. Together these contribute
$2^{\,2\omega_{+}(d)+\omega_{-}(d)}$.

\emph{Primes above $2$.} By definition these contribute $2^{\delta}$; that
$\delta$ depends only on the \(2\)-adic stratum follows from the local
square-class computation in Lemma~\ref{lem:dyadic} below.

No other place ramifies. Hence $\prod_{v}e_v
=2^{\,2\omega_{+}+\omega_{-}+\delta+2}$, and
\eqref{eq:ambiguous-class-number}, \eqref{eq:unit-norm-index}, and
$[K:k_0]=2$ give, on a density-one subset,
\[
  \bigl|\Cl_K^{G}\bigr|
  =\frac{h_{k_0}\cdot2^{\,2\omega_{+}+\omega_{-}+\delta+2}}{2\cdot4}
  =h_{k_0}\cdot2^{\,2\omega_{+}+\omega_{-}+\delta-1}.
\]
Since $h_{k_0}$ is odd, the $2$-primary part of $\Cl_K^{G}$ has order
$2^{\,2\omega_{+}+\omega_{-}+\delta-1}$.  The first paragraph of the proof
identifies this order with $2^{\rk_2\Cl_K}$.
\end{proof}

Taking $2$-parts in~\eqref{eq:kuroda-class-number} gives the identity on which
the comparison of ranks turns:
\begin{equation}\label{eq:kuroda-2part}
  \sum_{n\ge1}\rk_{2^{n}}\Cl_K
  \ =\ \log_2q_K-1+\sum_{i=0}^{2}\sum_{n\ge1}\rk_{2^{n}}\Cl_{k_i}.
\end{equation}

We need a negative-twist correction to the real-base formula of
Koymans--Morgan--Smit.  Recall that \(\Delta(n)\) is the discriminant of
\(\BQ(\sqrt n)\), and let $\Delta_F$ be the discriminant of $F$.  For an
integer $m$, write $\omega_{\mathrm{in}}(m)$ for the number of its distinct
prime divisors that are inert in $F/\BQ$.

\begin{lemma}\label{lem:kms-negative}
Let \(F/\BQ\) be real quadratic.  For a density-one set of negative squarefree
integers \(n\),
\begin{equation}\label{eq:kms}
  \rk_4\Cl_{F(\sqrt n)}
  =\omega_{\mathrm{in}}(\Delta(n))+\omega(\Delta_F)
   +\rk_2\Cl_F-2.
\end{equation}
\end{lemma}

\begin{proof}
Let \(\chi_n\) be the quadratic character of \(F(\sqrt n)/F\), and use the
Selmer groups of \cite[\S3.1]{KMS}.  The kernel of restriction from
\(H^1(G_F,\BZ/2\BZ)\) to \(H^1(G_{F(\sqrt n)},\BZ/2\BZ)\) is
\(\{0,\chi_n\}\).  At every real place the local condition defining
\(\operatorname{Sel}_{\chi_n}(G_F,\BZ/2\BZ)\) is zero.  Since \(n<0\), the character
\(\chi_n\) is nontrivial at both real places, and hence
\[
 \operatorname{Sel}_{\chi_n}(G_F,\BZ/2\BZ)\cap\{0,\chi_n\}=\{0\}.
\]
Thus the last sentence in the proof of \cite[Theorem~3.4]{KMS}, which inserts
the one-dimensional kernel \(\{0,\chi_n\}\), must be omitted in this sign
range.  On the same generic set used in their proof of Theorem~6.1, the
dimension identity is consequently
\[
\begin{aligned}
 \dim_{\BF_2}2\operatorname{Sel}(G_{F(\sqrt n)},\BZ/4\BZ)
  ={}&\dim_{\BF_2}\operatorname{Sel}_{\chi_n}(G_F,\BZ/2\BZ)\\
    &+\dim_{\BF_2}\operatorname{Sel}(G_F,\BZ/2\BZ),
\end{aligned}
\]
without the final \(-1\) printed there.  Proposition~3.10 of \cite{KMS}
is unchanged and, for real \(F\), gives
\[
 \dim_{\BF_2}\operatorname{Sel}_{\chi_n}(G_F,\BZ/2\BZ)
 =\omega_{\mathrm{in}}(\Delta(n))+\omega(\Delta_F)-2.
\]
Their Lemma~3.1 identifies the other two dimensions with
\(\rk_4\Cl_{F(\sqrt n)}\) and \(\rk_2\Cl_F\), proving~\eqref{eq:kms}.

Finally, the exceptional sets in the inputs to their proof have size \(o(X)\)
among squarefree \(|n|\le X\).  Their intersection with \(n<0\), and hence with
the fixed coprimality subfamily used here, is still \(o(X)\).
\end{proof}

Write $\iota_2=\ind_{\{2\text{ is inert in }k_0/\BQ\}}$ and
$w=\omega(\Delta_{d_0})-s\in\{0,1\}$.  Since the odd prime divisors of
$\Delta(-d)$ are those of $d$,
\[
  \omega_{\mathrm{in}}(\Delta(-d))
  =\omega_-(d)+b_1\iota_2.
\]

\begin{proposition}\label{prop:lambda-value}
Let $d_0>1$ be squarefree with $h_{k_0}$ odd.  Then
\begin{equation}\label{eq:lambda-value}
  \lambda_K\peq b_1(\iota_2-1)+w-(b_2-s)+\delta,
\end{equation}
where $b_1,b_2$ are the border widths from Section~\ref{ssec:block} and
$\delta$ is as in Proposition~\textup{\ref{prop:rk2}}.
\end{proposition}

\begin{proof}
The hypothesis that $h_{k_0}$ is odd gives
$\rk_2\Cl_{k_0}=0$ identically.  On the density-one set supplied by
Theorem~\ref{thm:qK} and Proposition~\ref{prop:rk2}, one also has $q_K=1$.
Split the two sides of \eqref{eq:kuroda-2part} into
their $2$-, $4$- and higher-rank contributions, and cancel the higher tails
by the definition of $\lambda_K$.  This gives
\[
  \rk_4\Cl_K=\sum_{i=0}^{2}\rk_2\Cl_{k_i}-\rk_2\Cl_K-1+\lambda_K.
\]
Genus theory gives $\rk_2\Cl_{k_1}=r+b_1-1$ and
$\rk_2\Cl_{k_2}=r+b_2-1$.  Substituting \eqref{eq:rk2} and
$r=\omega_+(d)+\omega_-(d)$ yields
\begin{equation}\label{eq:rk4-id}
  \rk_4\Cl_K
  \peq\omega_{-}(d)+b_1+b_2-2-\delta+\lambda_K.
\end{equation}
On the common density-one set, apply~\eqref{eq:kms} with $F=k_0$ and $n=-d$,
use $\rk_2\Cl_{k_0}=0$ and $\omega(\Delta_{d_0})=s+w$, and compare the result
with~\eqref{eq:rk4-id}.  This gives~\eqref{eq:lambda-value}.
\end{proof}

The right side of~\eqref{eq:lambda-value} is determined by the following local
calculation at \(2\).

\begin{lemma}\label{lem:dyadic}
Let $M/\BQ_2$ be a ramified quadratic extension and let $u\in\BZ_2^{\times}$.
\begin{enumerate}[\upshape(i)]
\item If $M=\BQ_2(\sqrt v)$ with $v\in\BZ_2^{\times}$ (equivalently,
  \(v\equiv3\bmod4\)), then $M(\sqrt u)/M$ is unramified for every
  $u\in\BZ_2^{\times}$.
\item If $M=\BQ_2(\sqrt{2v})$ with $v\in\BZ_2^{\times}$, then $M(\sqrt u)/M$ is
  ramified precisely when $u\equiv3\bmod4$.
\end{enumerate}
\end{lemma}

\begin{proof}
The unramified quadratic extension of $\BQ_2$ is $\BQ_2(\sqrt5)$, so the
unramified quadratic extension of $M$ is $M(\sqrt5)$; hence for a unit $u$ the
extension $M(\sqrt u)/M$ is unramified (possibly trivial) if and only if $u$
lies in $(M^{\times})^{2}\cup5\,(M^{\times})^{2}$.  Now a unit $u$ is a square in
$M$ if and only if $\BQ_2(\sqrt u)\subseteq M$, i.e.\ if and only if
$\BQ_2(\sqrt u)$ is $\BQ_2$ or $M$.

In case (i), $M=\BQ_2(\sqrt v)$ with $v$ a unit, so the units that are squares in
$M$ are those with $u\equiv1$ or $u\equiv v$ modulo $(\BZ_2^{\times})^{2}$, that
is, the classes $\{1,v\}$ in $\BZ_2^{\times}/(\BZ_2^{\times})^{2}=\{1,3,5,7\}
\bmod8$. Adjoining the classes $5\cdot\{1,v\}=\{5,5v\}$ and using
$v\in\{3,7\}\bmod 8$, one gets $\{1,3,5,7\}$ in either case: $\{1,3\}\cup
\{5,7\}$ for $v\equiv3$, and $\{1,7\}\cup\{5,3\}$ for $v\equiv7$. So every unit
is covered and no unit gives a ramified extension.

In case (ii), $M$ is generated by the non-unit $\sqrt{2v}$, so $M\ne\BQ_2(\sqrt
u)$ for any unit $u$ and the only unit square class in $M$ is $\{1\}$. The
unramified class is $\{5\}$, and the remaining classes $\{3,7\}$, equivalently
$u\equiv3\bmod4$, give ramified extensions.
\end{proof}

\subsubsection{The direct-sum decomposition and the
\texorpdfstring{$8$}{8}-rank distribution}\label{ssec:defect-vanishing}

\begin{proof}[Proof of Theorem~\ref{thm:biquad-decomp-intro}]
The Hasse unit-index assertion is Theorem~\ref{thm:qK}.  By the
characterization following Definition~\ref{def:lambdaK}, it remains to prove
$\lambda_K\peq0$ for the group assertion.  Proposition~\ref{prop:lambda-value}
reduces this to showing that
\[
  E=b_1(\iota_2-1)+w-(b_2-s)+\delta
\]
vanishes in every admissible \(2\)-adic stratum.

Throughout, $b_1=1$ exactly when $\kappa$ is even, i.e.\ when
$-d\equiv3\bmod4$.  There are three cases according to the behaviour of $2$ in
$k_0$.

\emph{$2$ splits $(d_0\equiv1\bmod 8$; then $e=0$, $\kappa_0$ even, $w=0$,
$\iota_2=0)$.} Formula~\eqref{eq:b2-width} gives $b_2-s=b_1$. Both
completions are $\BQ_2$, and
$\BQ_2(\sqrt u)/\BQ_2$
is unramified for $u\equiv1\bmod4$ and ramified for $u\equiv3\bmod4$; so
$\delta=2b_1$. Hence $E=-b_1-b_1+2b_1=0$.

\emph{$2$ is inert $(d_0\equiv5\bmod8$; then $e=0$, $\kappa_0$ even, $w=0$,
$\iota_2=1)$.} Again $b_2-s=b_1$. The completion is the unramified
quadratic extension $L/\BQ_2$, in which $5$ becomes a square, so
$L(\sqrt u)/L$ is
unramified for $u\equiv1,5\bmod8$ and ramified for $u\equiv3,7\bmod8$; thus
$\delta=b_1$. Hence $E=-b_1+b_1=0$.

\emph{$2$ ramifies.} Then $\iota_2=0$ and $w=1$, and there is one dyadic prime
$\mathfrak p$ of $k_0$, so $\delta\in\{0,1\}$ and
$E=1-b_1-(b_2-s)+\delta$.
\begin{itemize}
\item $e=0$ and $d_0\equiv3\bmod4$ $(\kappa_0$ odd$)$: here
  $k_{0,\mathfrak p}=\BQ_2(\sqrt{d_0})$ with $d_0$ a unit, so
  Lemma~\ref{lem:dyadic}(i) gives $\delta=0$ for every $d$.  Since
  $\kappa_0$ is odd, $b_2-s=1-b_1$, and hence
  $E=1-b_1-(1-b_1)=0$.
\item $e=1$: here $k_{0,\mathfrak p}=\BQ_2(\sqrt{d_0})$ with
  $d_0=2q_1\cdots q_s$,
  so Lemma~\ref{lem:dyadic}(ii) gives $\delta=1$ exactly when $-d\equiv3\bmod4$,
  i.e.\ $\delta=b_1$.  Also $b_2-s=1$, so
  $E=1-b_1-1+b_1=0$.
\end{itemize}
This exhausts the three possible behaviours of \(2\) in \(k_0\).
\end{proof}

\begin{proof}[Proof of Corollary~\ref{cor:biquad-intro}]
Taking $2$-ranks in the isomorphism of
Theorem~\ref{thm:biquad-decomp-intro} gives
\[
  \rk_8\Cl_K\peq\sum_{i=0}^2\rk_4\Cl_{k_i}.
\]
Since $h_{k_0}$ is odd, $\rk_4\Cl_{k_0}=0$, which proves the first assertion.

For uniform $d\in\mathcal F_{d_0}(X)$, put
\[
 X_d:=\rk_8\Cl_K,\qquad
 Y_d:=r_4(-d)+r_4(-d_0d).
\]
The first assertion couples these variables so that
\[
 \dtv\bigl(\Law(X_d),\Law(Y_d)\bigr)
 \le \Prob(X_d\ne Y_d)=o(1).
\]
By Theorem~\ref{thm:B} and contraction of total variation under the sum map,
the distribution of $Y_d$ converges in total variation to the convolution of
$\piCL$ with itself, whose mass at $m$ is
\[
 \sum_{a+b=m}\piCL(a)\piCL(b).
\]
The triangle inequality gives the same convergence for $X_d$.
\end{proof}

\begin{remark}[Related biquadratic results]
Prior results for biquadratic fields include the Dirichlet setting of
Fouvry--Koymans and Fouvry--Koymans--Pagano \cite{FK,FKP}, and the density-one
formula of Koymans--Morgan--Smit for $F(\sqrt n)$ with $F$ quadratic
\cite{KMS}.  In the real-base, negative-twist correction used above, the
archimedean local condition changes the displayed constant by one.  Arithmetic
statistics of the Hasse unit index in other biquadratic families were studied
by Chan--Milovic, Milovic, and Koymans--Pagano
\cite{ChanMilovicKuroda,MilovicHasse,KP2022stev}.  Those works concern
different signatures, a different fixed quadratic subfield, or
prime-parametrized families.

The $2$-torsion of biquadratic and, more generally, multiquadratic fields is
treated by Fr\"ohlich and by Koymans--Pagano \cite{Frohlich,HGT}.
The \(2\)-rank of imaginary bicyclic biquadratic fields was studied by
McCall--Parry--Ranalli \cite{McCallParryRanalli}.
Proposition~\ref{prop:rk2} is the direct specialization of the ambiguous class
number formula to the present family.
The resulting $4$-rank identity is likewise a specialization of Kuroda's
formula.  For
$K_n=\BQ(\sqrt n,\sqrt{-n})$, Fouvry--Koymans
\cite[equations~(3.1) and~(3.2)]{FK} obtain
\[
 \rk_4\Cl_{K_n}
 =\rk_2\Cl\bigl(\BQ(\sqrt n)\bigr)
  +\rk_2\Cl\bigl(\BQ(\sqrt{-n})\bigr)
  -\rk_2\Cl_{K_n}+\log_2 Q(n)-1
\]
under hypotheses that make the \(8\)-ranks and all higher ranks vanish.  In the
present family with fixed real subfield \(k_0\), the higher-rank tails cannot be
discarded a priori.  Their difference is recorded by the nonnegative defect
$\lambda_K$, which is shown to vanish only after the dyadic calculation.

\end{remark}

\end{document}